\documentclass[reqno,11pt]{amsart}
\usepackage{amsmath, latexsym, amsfonts, amssymb, amsthm, amscd}
\usepackage{enumitem}
\usepackage{comment}

\usepackage[foot]{amsaddr}

\usepackage{pgfplots}
\pgfplotsset{compat=1.12}
\usetikzlibrary{arrows}
\usepackage{graphicx}

\usepackage{caption}
\usepackage{subcaption}
\usepackage{float}

\usepackage{scalerel}

\DeclareFontFamily{U}{BOONDOX-calo}{\skewchar\font=45 }
\DeclareFontShape{U}{BOONDOX-calo}{m}{n}{
  <-> s*[1.05] BOONDOX-r-calo}{}
\DeclareFontShape{U}{BOONDOX-calo}{b}{n}{
  <-> s*[1.05] BOONDOX-b-calo}{}
\DeclareMathAlphabet{\mathcalb}{U}{BOONDOX-calo}{m}{n}
\SetMathAlphabet{\mathcalb}{bold}{U}{BOONDOX-calo}{b}{n}
\DeclareMathAlphabet{\mathbcalb}{U}{BOONDOX-calo}{b}{n}

\usepackage{graphics,epsf,psfrag}
\usepackage{xcolor}

\numberwithin{equation}{section}

\newtheorem{theorem}{Theorem}[section]
\newtheorem{lemma}[theorem]{Lemma}
\newtheorem{proposition}[theorem]{Proposition}
\newtheorem{cor}[theorem]{Corollary}
\newtheorem{rem}[theorem]{Remark}
\newtheorem{definition}[theorem]{Definition}

\DeclareMathOperator{\p}{\mathbb{P}}
\DeclareMathOperator{\sign}{\mathrm{sign}}

\newcommand{\ind}{\mathbf{1}}

\newcommand{\R}{\mathbb{R}}
\newcommand{\Z}{\mathbb{Z}}

\renewcommand{\tilde}{\widetilde}

\newcommand{\cF}{{\ensuremath{\mathcal F}} }

\newcommand{\cH}{{\ensuremath{\mathcal H}} }

\newcommand{\cS}{{\ensuremath{\mathcal S}} }

\newcommand{\cZ}{{\ensuremath{\mathcal Z}} }

\newcommand{\bP}{{\ensuremath{\mathbf P}} }
\newcommand{\bE}{{\ensuremath{\mathbf E}} }

\DeclareMathSymbol{\leqslant}{\mathalpha}{AMSa}{"36} 
\DeclareMathSymbol{\geqslant}{\mathalpha}{AMSa}{"3E} 
\DeclareMathSymbol{\eset}{\mathalpha}{AMSb}{"3F}     
\newcommand{\dd}{\,\text{\rm d}}             

\newcommand{\bbE}{{\ensuremath{\mathbb E}} }

\newcommand{\bbL}{{\ensuremath{\mathbb L}} }

\newcommand{\bbN}{{\ensuremath{\mathbb N}} }

\newcommand{\bbP}{{\ensuremath{\mathbb P}} }

\newcommand{\bbR}{{\ensuremath{\mathbb R}} }

\newcommand{\bbZ}{{\ensuremath{\mathbb Z}} }

\newcommand{\ga}{\alpha}

\newcommand{\gd}{\delta}
\newcommand{\gep}{\varepsilon}       

\newcommand{\gG}{\Gamma}

\newcommand{\gD}{\Delta}
\newcommand{\gk}{\kappa}
\newcommand{\go}{\omega}

\newcommand{\gl}{\lambda}

\newcommand{\gs}{\sigma}

\makeatletter
\def\captionfont@{\footnotesize}
\def\captionheadfont@{\scshape}

\long\def\@makecaption#1#2{%
  \vspace{2mm}
  \setbox\@tempboxa\vbox{\color@setgroup
    \advance\hsize-6pc\noindent
    \captionfont@\captionheadfont@#1\@xp\@ifnotempty\@xp
        {\@cdr#2\@nil}{.\captionfont@\upshape\enspace#2}%
    \unskip\kern-6pc\par
    \global\setbox\@ne\lastbox\color@endgroup}%
  \ifhbox\@ne 
    \setbox\@ne\hbox{\unhbox\@ne\unskip\unskip\unpenalty\unkern}%
  \fi
  \ifdim\wd\@tempboxa=\z@ 
    \setbox\@ne\hbox to\columnwidth{\hss\kern-6pc\box\@ne\hss}%
  \else 
    \setbox\@ne\vbox{\unvbox\@tempboxa\parskip\z@skip
        \noindent\unhbox\@ne\advance\hsize-6pc\par}%
\fi
  \ifnum\@tempcnta<64 
    \addvspace\abovecaptionskip
    \moveright 3pc\box\@ne
  \else 
    \moveright 3pc\box\@ne
    \nobreak
    \vskip\belowcaptionskip
  \fi
\relax
}
\makeatother
\def\writefig#1 #2 #3 {\rlap{\kern #1 truecm
\raise #2 truecm \hbox{#3}}}

\def\r{{\mathbb R}}
\def\e{{\mathbb E}}
\def\p{{\mathbb P}}

\def\E{{\bf E}}

\def\d{\, \mathrm{d}}

\newcommand{\ttt}{\mathtt{t}}
\newcommand{\ttdown}{\mathtt{t}^\downarrow}
\newcommand{\ttup}{\mathtt{t}^\uparrow}
\newcommand{\tu}{\mathtt{u}}
\newcommand{\tudown}{\mathtt{u}^\downarrow}
\newcommand{\tuup}{\mathtt{u}^\uparrow}
\newcommand{\ts}{\mathtt{s}}
\newcommand{\tsdown}{\mathtt{s}^\downarrow}
\newcommand{\tsup}{\mathtt{s}^\uparrow}
\newcommand{\tv}{\mathtt{v}}
\newcommand{\tvdown}{\mathtt{v}^\downarrow}
\newcommand{\tvup}{\mathtt{v}^\uparrow}
\newcommand{\arrow}{\mathcalb{a}}
\newcommand{\Brev}{B\mystrut^{\scaleto{\mathtt{rv}\mathstrut}{5pt}}}

\newcommand{\sF}{s^\ssup{F}}
\newcommand{\sFR}{s^\ssup{F, \r}}

\newcommand{\ls}{\widehat{l}}
\newcommand{\rs}{\widehat{r}}
\newcommand{\ms}{\widehat{m}}

\newcommand{\ssup}[1] {{{\scriptscriptstyle{({#1}})}}} 

\newcommand{\arcth}{{\rm arctanh}}
\newcommand{\loglog}{\log_{\circ 2}}

\def\law{{\buildrel \mbox{\rm\tiny (law)} \over =}}

\newcommand{\tf}{\textsc{f}}

\newcommand\mystrut{\rule{0pt}{6pt}}

\begin{document}

\title[Strong interaction limit for the Continuum Random Field Ising Chain]{Infinite disorder renormalization fixed point for the Continuum Random Field Ising chain}

\author[O. Collin, G. Giacomin and Y. Hu]{Orph\'ee Collin, Giambattista Giacomin and Yueyun Hu}
\address[OC, GG]{Universit\'e  Paris Cit\'e and Sorbonne Universit\'e, CNRS, Laboratoire de Probabilit{\'e}s, Statistique et Mod\'elisation,
            F-75013 Paris, France}
    \address[YH]{Universit\'e Paris XIII, LAGA, Institut Galil\'ee, F-93430 Villetaneuse, France}
\address{Corresponding author: G. Giacomin, giacomin@lpsm.paris}    

\begin{abstract}
We consider the continuum version of the random field Ising model in one dimension: this  model arises naturally as weak disorder scaling limit 
of the original Ising model. Like for the Ising model,  a spin configuration is conveniently described as a sequence of spin domains with alternating signs (\emph{domain-wall structure}). 
We show that
for fixed  centered  external field and as 
 spin-spin couplings become large, the domain-wall structure scales to a  disorder dependent limit  that coincides with the \emph{infinite disorder fixed point} process introduced by D.~S.~Fisher in the context of zero temperature quantum Ising chains. In particular, our results establish a number of predictions that one can find in \cite{cf:FLDM01}. The infinite disorder  fixed point process for centered external field
 is equivalently described in terms of the process of  \emph{suitably selected}  extrema of a Brownian trajectory  introduced and studied by J. Neveu and J. Pitman \cite{cf:NP89}. This characterization of the infinite disorder fixed point  is one of the important ingredients of  our analysis.
   
\bigskip

\noindent  \emph{AMS  subject classification (2010 MSC)}:
60K37,  
82B44, 
60J70, 
82B28 

\end{abstract}

\maketitle
\tableofcontents

\section{Introduction}
Understanding the  behavior of disordered systems is often extremely challenging, above all when the disorder modifies in a substantial way the behavior of the system with respect to when disorder is absent (\emph{pure model}). 
One of the most basic  and highly studied  models in this field  is the 
ferromagnetic Random Field Ising Chain (RFIC), that is the probability measure on $\{-1,1\}^{V_N}$, $V_N:= \{1,2, \ldots, N\}$, with (discrete)
probability  density proportional to 
\begin{equation}
\label{eq:RFIC}
\exp\left( J\sum_{j=1}^{N} \gs _j \gs_{j-1} + \sum_{j=1}^N h_j \gs_j \right)\, ,
\end{equation}
for $ \gs \in \{-1,1\}^{V_N}$ and  for a choice of the boundary condition $\gs_0$. In addition, $J\ge 0$ is the ferromagnetic interaction that plays in favor of alignment between spins,  and 
the real values $h_1, h_2, \ldots$ 
 form the sequence of external fields: for us they will be (the realization of) a sequence of IID random variables  with 
 $\bbE [ \exp( t h)] < \infty$ for $t $ in a neighborhood of the origin.  In  the end this work is only about  $\bbE[h]=0$, even if at some point, mostly  for presentation issues, the  case $\bbE[h]\neq 0$ will be  considered too. 
 
{It is well known that this model does not exhibit a phase transition, so there is no critical phenomenon. However, a pseudo-critical  behavior appears in the $J \to \infty$  limit (we always consider $N \to \infty$
before $J \to \infty$). We aim at understanding the configurations of the system in this limit:}

\smallskip

\begin{itemize}[leftmargin=0.4 cm]
\item In the pure case, i.e. $h_j=0$ for every $j$, the model is equivalent to a Markov chain with state space $\{-1, +1\}$ with transition probability from $+1$ (respectively, $-1$) to $-1$ (respectively, $+1$) equal to  $1/(1+e^{2J})$. 
Therefore the typical configurations for $J$ large are very long 
\emph{domains} of aligned spins: 
the distance between the  \emph{walls} (i.e., the spin flip locations that identify the interfaces between domains of different sign) are independent geometric variables with mean  $1+e^{2J}$.
\item As soon as disorder is present the model is no longer solvable, but an Imry-Ma argument (see for example \cite[p.~373]{cf:IM}) rapidly leads to guessing that the domain-wall structure changes radically 
and that the expected spatial scale drastically reduces to $J^2/ \text{Var}(h)$. The Imry-Ma argument may be given as follows: if we are in a (say) $-1$ domain region, 
switching to $+1$ in an interval $I_L$ of length $L$ costs two spin flips (i.e., $4J$) but the energy change due to the external fields 
is $2 \sum_{j\in I_L} h_j$ which, for $L$ large, is $2 \sqrt{L\, \text{Var}(h)}$ times a Gaussian standard variable.  Hence, if  $L\propto J^2/\text{Var}(h)$, by choosing 
wisely (and in a disorder dependent fashion!) $I_L$, the energetic gain from the external random field can overcome the spin flip cost. This not only strongly suggests that  $J^2/\text{Var}(h)$
is the correct space scale, but also that the wall positions  heavily depend on the disorder configuration  $(h_j)$. 
 \end{itemize}
\smallskip

The  Imry-Ma argument we just outlined therefore suggests that disorder is \emph{strongly relevant} for the RFIC in the $J \to \infty$ limit. 
This terminology -- at least the term \emph{relevant} -- is not only suggestive of the fact that disorder  changes the behavior of the system, and in a substantial way: it is a standard terminology  
in the Renormalization Group (RG) context (see for example \cite[\S~5.3]{cf:GB} and \cite{cf:IM,cf:Vojta}). 
Disorder is dubbed \emph{relevant} when it changes the RG fixed point (in practice, the large scale behavior of the system) with respect to the pure system: as we are going to explain, in the case we are considering the disorder ends up having an \emph{infinitely strong effect}. 
To make a long story short, it is  in the RG community that a very precise description of the domain-wall structure has been set forth. In fact, in \cite{cf:FLDM01}
 D.~S.~Fisher, P.~Le Doussal and C.~Monthus claim that the $J \to \infty$ domain-wall structure is sharply captured by 
 the \emph{infinite disorder fixed point} identified by D.~S.~Fisher in the framework of disordered quantum  chains \cite{cf:F92,cf:F95}. 
 
 It should be noted that in \cite{cf:FLDM01} the equilibrium case is considered as well as the problem of the approach to equilibrium: for the equilibrium case a vast literature is cited (we single out the extensive review monograph 
  \cite{cf:Luck-book}),  but the precise description of the domain-wall structure,  to our knowledge, first appears in \cite{cf:FLDM01}. We only consider the equilibrium model.
 
  Fisher's infinite disorder fixed point and, more generally, Fisher's approach is expected to apply to a wide universality class of models.  
  We  refer to  the review papers \cite{cf:IM,cf:IM2,cf:Vojta} for extended accounts, but we stress that how wide this universality class is represents a very challenging problem, and not only from a mathematical viewpoint  
  (for example, in \cite{cf:CDDHLS,cf:DR} a special pinning model is shown not to belong to this class, suggesting that general pinning models \cite{cf:GB}, see \cite{cf:IM2} for the RG approach,  do not belong either).  Nevertheless it is worth pointing out that the most basic model to which Fisher's ideas apply is Random Walk in Random Environment (RWRE). 
  The mathematical understanding of  RWRE, above all in  one dimension, is extremely advanced (see \cite{Zeitouni}) and the agreement with Fisher's RG approach has been verified at least in part. We are not going into a detailed discussion  of this issue and we  refer to  \cite{cf:Cheliotis,cf:BF08}. We rather point to the fact that in   \cite{cf:Cheliotis,cf:BF08}, see also \cite{cf:Faggionato}, it is spelled out that  Fisher's infinite disorder fixed point is a stochastic process that was already known in the mathematical literature:
  it is the Brownian motion $\gG$-extrema process introduced and studied by J. Pitman and J. Neveu in \cite{cf:NP89}.  
  In Fisher's approach instead,  the $\gG$-extrema process arises from a dynamical coarse graining procedure that progressively absorbes  the small spin domains into the neighboring ones: the coarse graining procedure is stopped when all the domains have size  $\gG$ or more ($\gG={2J}$ in our case). 
  We insist on the fact that Fisher's fixed point, alias Pitman-Neveu process, just depends on the disorder (that is, on an IID sequence that, in the large scale limit, becomes a white noise or, equivalently, a Brownian motion). {We stress that we have kept the original presentation of the Pitman-Neveu process, which is rather a family of $\gG$-index processes. But, as spelled out in Remark~\ref{rem:NP}, the Pitman-Neveu process enjoys a scale invariance 
 and Fisher's fixed point can be seen, in a more customary fashion, as the  Pitman-Neveu process with $\gG=1$.} 
  Two sources of  randomness are present in a  disordered system: the disorder itself, and the thermal fluctuations of the spins. Characterizing the disorder as  \emph{infinite} in Fisher RG fixed point is directly linked to the fact that disorder (drastically!) dominates over thermal fluctuations whenever Fisher's approach applies. 

\smallskip

The purpose of this work is to give a precise statement  that explains how  Fisher's infinite disorder fixed point captures the domain-wall structure in a continuum version of the one-dimensional Ising chain with centered random external field. The continuum version of the Ising model we consider appears in the literature since  the end of the  60s \cite{cf:MW1} where the authors
{ analyze  the weak disorder scaling limit for the two dimensional Ising model with interaction disorder of \emph{columnar} type \cite{cf:MW1,cf:CGG}:
the key point is that the two dimensional partition function is written in terms of a family of (one dimensional) RFIC transfer-matrices. We signal that the naturally related, but very different issue of  the weak disorder scaling limit of the two dimensional  Ising model with random external field is considered in \cite{cf:BS22}.  }
Weak disorder scaling limits were first introduced in \cite{cf:FrLl} in a different context (Anderson localization in  one dimension). For a review of the literature on this limit we refer to \cite{cf:CGG,cf:CTT} and references therein. Much more generally and  recently, weak disorder scaling limits have attracted  a lot of attention, both because of the mathematicians' efforts to understand disorder relevance in statistical mechanics models and because of  their link with singular SPDEs \cite{cf:AKQ,cf:CSZ}. 
We stress  that, in  spite of the fact that these continuum models arise as \emph{weak} disorder scaling limit of the original models,  they appear to capture phenomena in which disorder \emph{strongly} dominates on large scales. And this is in fact what we prove for the model we consider: we  show that \emph{to leading order} the behavior of  the model is fully captured by  Fisher's infinite disorder fixed point. 

\section{The model and the main results}
\subsection{The Continuum RFIC}
\label{sec:model}
$(N_t)_{t\ge 0}$ denotes a Poisson process with intensity 
\begin{equation}
\label{eq:gep}
\gep\,:=\, e^{-\Gamma}\,.
\end{equation}
We denote by 
 $\bP_\gG$ the law of $(N_t)_{t\ge 0}$ and we introduce the free spin process ${\bf s}_t:=(-1)^{N_{t}}\in \{-1,1\}$ for $t \ge 0$. 
 So, under $\bP_0$, $(N_t)_{t\ge 0}$ is a Poisson process of rate $\exp(-0)=1$. {We follow the standard convention of choosing the right-continuous version of $(N_t)_{t\ge 0}$.}
$\bbP$ denotes the law of $B=(B_t)_{t\ge 0}$,  which is a standard  Brownian motion.

The partition function of the Continuum Random Field Ising Chain (Continuum RFIC) of length $\ell>0$ is
\begin{equation}
\label{eq:Z0}
\begin{split}
 Z^{\mathtt{f}}_{\gG, B_\cdot, \ell}\,:&=\,  \E_\Gamma \left[ \exp\left( \int_{0}^\ell {\bf s}_t \left(\dd B_t+\alpha \dd t\right)\right)    \right]
 \\&=\, 
 \E_0 \left[ \exp\left(- \gG N_\ell + \int_{0}^\ell {\bf s}_t \left(\dd B_t+\alpha \dd t\right)\right)    \right]\Big/ \E_0 \left[ \exp\left(- \gG N_\ell\right)\right] \, ,
 \end{split}
 \end{equation}
with $\ga \in \bbR$ an asymmetry parameter that we introduced for presentation purposes -- it corresponds to $\bbE [h]$ in the (discrete) RFIC -- 
and  the results in this work just concern the case $\ga=0$: so, unless explicitly specified otherwise, $\ga=0$. Note that  in the standard RFIC \eqref{eq:RFIC} 
 we can write the energy due to the nearest-neighbor interaction as $JN$ minus $\gG(=2J)$ times the number of nearest-neighbor spins that have different sign: hence 
 the second line in \eqref{eq:Z0} makes clear  the link between RFIC and Continuum RFIC (see Appendix~\ref{sec:C0} for a deeper analysis of this link).
 
The partition function \eqref{eq:Z0} has the superscript $\mathtt{f}$ because the boundary condition on the right is free (and on the left it is $+1$: note that ${\bf s}_{0}=+1$ since $N_0=0$). But we will mostly work with 
\begin{equation}
\label{eq:Z0+}
 Z_{\gG, B_\cdot, \ell}\,:=\,  \E_\Gamma \left[ \exp\left( \int_{0}^\ell {\bf s}_t \left(\dd B_t+\alpha \dd t\right)\right)  ; \, {\bf s}_\ell=+1  \right]\, .
 \end{equation}
That is, the boundary conditions are $+1$ on both ends:
we use the standard notation $\E_\Gamma [X; A]=\E_\Gamma [X \ind_A]$ with $X$  a random variable and $A$ an event.
 There is no fundamental reason to choose these boundary conditions and in fact 
the results we prove hold with any choice of boundary conditions $\pm 1$ on the right and on the left (see Remark \ref{rem:bc}). Hence the results hold also for 
the free model  corresponding to \eqref{eq:Z0}. But for definiteness, and without loss of generality, we focus on \eqref{eq:Z0+}.




We denote by $\mu_{\gG, B_\cdot, \ell}$ the  Gibbs measure associated with $Z_{\gG, B_\cdot, \ell}$, in particular for $t \in [0, \ell]$
\begin{equation}
\mu_{\gG, B_\cdot, \ell}\left( \mathbf{s}_t=+1 \right) \, =\, 
  \E_\Gamma \left[ \exp\left( \int_{0}^\ell {\bf s}_t \left(\dd B_t+\alpha \dd t\right)\right) ; \, \mathbf{s}_t=+1, \, 
  {\bf s}_\ell=+1 \right]\Big / Z_{\gG, B_\cdot, \ell}\, .
\end{equation}
{By construction, this Gibbs measure gives probability one to the right-continuous configurations $\mathbf{s}_\cdot$ taking values in $\{-1,+1\}$ and 
such that $\mathbf{s}_0=\mathbf{s}_\ell=+1$.}

\subsection{Neveu-Pitman $\gG$-extrema and Fisher's RG fixed point}
\label{sec:Gextrema}

To present Fisher's RG fixed point and our result we need some notation. For the next paragraph it is useful to refer to
Fig.~\ref{fig1}.
\smallskip

For every continuous real-valued process $\xi=(\xi_t)_{t\ge 0}$,   let  
$T_\xi(x)$ be the first hitting of $x\in \r $, that is
\begin{equation}
T_\xi(x)\,:=\,  \inf\left\{t>0: \xi_t=x\right\}\, .
\end{equation}  
Furthermore define 
\begin{equation}
 \xi^\uparrow_t\,:=\,  \sup_{0\le r\le s\le t} (\xi_s-\xi_r), \qquad \xi^\downarrow_t:= \sup_{0\le r\le s\le t} (\xi_r-\xi_s)\,,  
 \end{equation}
 and for every $r>0$  
\begin{equation}
 \label{eq:completenotation1}
 \ttup_\xi(r)\,:=\, \inf\left\{t>0: \xi^\uparrow_t>r\right\}, 
 \qquad \ttdown_\xi(r)\,:=\, \inf\left\{t>0: \xi^\downarrow_t>r\right\} \, . 
 \end{equation}
Let 
\begin{equation}
 \label{eq:completenotation2}
\begin{split}
 \tuup_\xi(r) \,&:=\,  \sup\left\{0\le t \le \ttup_\xi(r): \xi_t= \inf_{0\le s \le \ttup_\xi(r)} \xi_s\right\}\, , 
 \\
\tudown_\xi(r)\,&:=\, \sup\left\{0\le t \le \ttdown_\xi(r): \xi_t= \sup_{0\le s \le \ttdown_\xi(r)} \xi_s\right\}\,. 
\end{split}
\end{equation}

When $\xi=B$, we will remove the subscript $\xi$ in $\ttup, \ttdown, \tuup$ and $\tudown$. 

\smallskip

Let us consider first the standard Brownian motion $(B_t)_{t \ge 0}$ (issued from $0$).
 If $\ttup(\Gamma) < \ttdown(\Gamma)$, we can define alternatively the sequence of rises and drops of height $\Gamma$ such that  $\ttup_1:= \ttup(\Gamma)< \ttdown_1(\Gamma):=\ttdown(\Gamma) <\ttup_2(\Gamma) <\ttdown_2(\Gamma) <...  $: For  $n\ge 2$, \begin{equation}
 \begin{split}
 \ttup_n(\Gamma)\,&:=\, \inf\left\{t>\ttdown_{n-1}(\Gamma): \sup_{\ttdown_{n-1}(\Gamma) \le r \le s \le t} (B_s- B_r) > \Gamma\right\}\,,
 \\
 \ttdown_n(\Gamma)\,&:=\,\inf\left\{t>\ttup_n(\Gamma): \sup_{\ttup_n(\Gamma) \le r \le s \le t} (B_r- B_s) > \Gamma\right\}\,.
 \end{split}
 \end{equation}
  Let also $\tuup_1(\Gamma)=\tuup(\Gamma)$, $\tudown_1(\Gamma)=\tudown(\Gamma)$ and for $n\ge 2$ $\tuup_n (\Gamma)\in [ \ttdown_{n-1}(\Gamma), \ttup_n(\Gamma)]$ and $\tudown_n(\Gamma) \in [\ttup_n(\Gamma), \ttdown_n(\Gamma)]$ be such that  
  \begin{equation}
   B_{\tuup_n (\Gamma)}\,=\,\inf_{\ttdown_{n-1}(\Gamma) \le s \le \ttup_n(\Gamma)} B_s, 
   \qquad B_{\tudown_n(\Gamma)}\,=\,\sup_{\ttup_n(\Gamma) \le s \le 
 \ttdown_n(\Gamma)} B_s\, .
 \end{equation}

\begin{figure}[ht]
\begin{center}
  \includegraphics[width=0.95\linewidth]{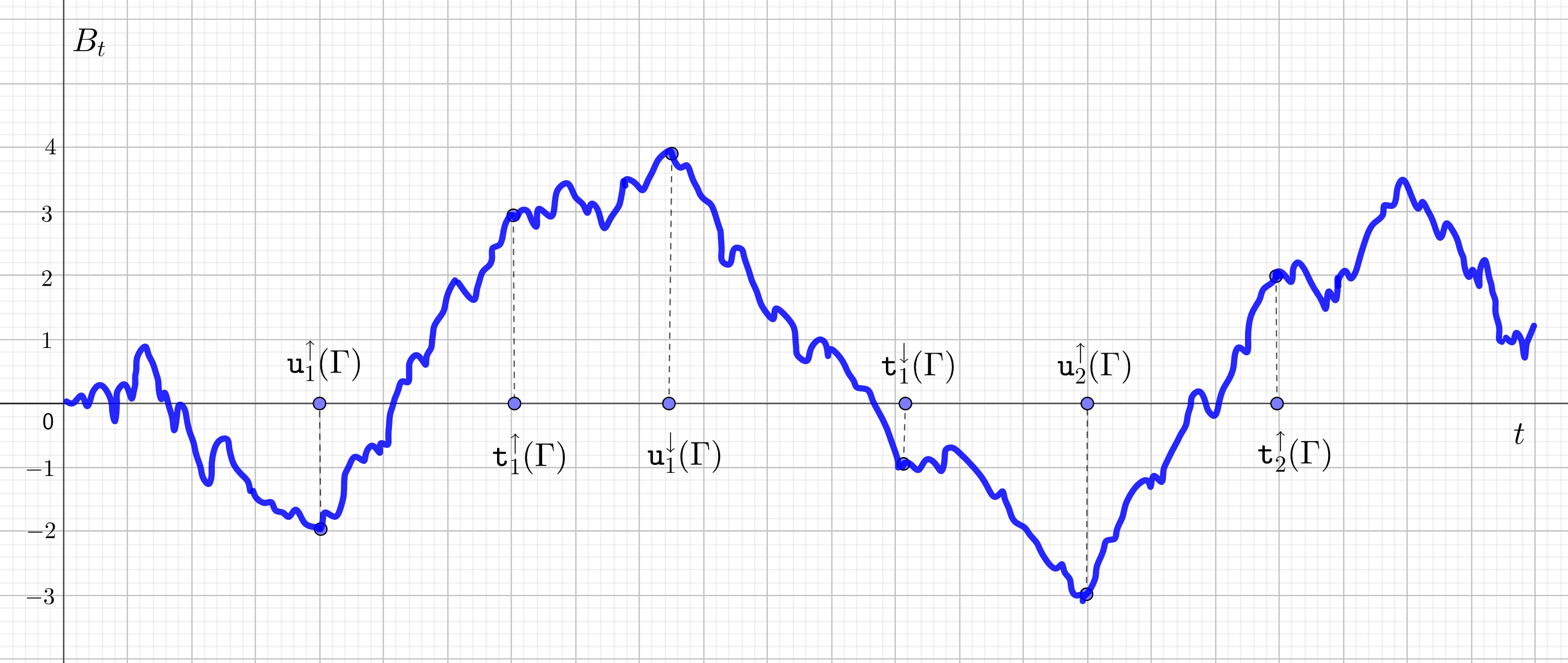}
  \end{center}
  \caption{
  \scriptsize  
  A Brownian trajectory  for which we identify three $\Gamma$-extrema, $\gG=5$. In this  case 
  $\ttup(\Gamma)<\ttdown(\Gamma)$. }  
  \label{fig1}
\end{figure}

 Similarly if $\ttdown(\Gamma)< \ttup(\Gamma) $, we let $\ttdown_1(\Gamma):=\ttdown(\Gamma)< \ttup_1(\Gamma):= \ttup(\Gamma)<  \ttdown_2(\Gamma) < \ttup_2(\Gamma)  <...$ We define the sequences $\tudown_n(\Gamma)$ and $\tuup_n(\Gamma)$ similarly. 
 
 Denote by $0<\tu_1(\Gamma)< \tu_2(\Gamma)< ...$ the ordered sequence formed by  $\{\tudown_n(\Gamma), \tuup_n(\Gamma)\}_{n=1, 2 , \ldots}$. Following Neveu and Pitman \cite{cf:NP89}, $(\tu_n(\Gamma))_{n\ge1}$ are called the times of $\Gamma$-extrema of $(B_t)_{t\ge 0}$.
We associate to the sequence $(\tu_n(\Gamma))_{n=1,2, \ldots}$ the auxiliary (or \emph{arrow}) sequence $(\arrow_n)_{n=1,2, \ldots}$:
if $\arrow_n=+1$ (respectively, if $\arrow_n=-1$),  then $\tu_n(\Gamma)$ is a maximum (respectively, a minimum). 
 
 We then introduce the Fisher trajectory $\sF_\cdot=(\sF_t)_{t\ge 0}$. For $t \ge 0$ we set
 \begin{equation}
 \label{eq:sFisher}
 \sF_t\, :=\,  \sum_{n=1}^\infty
  \arrow_{n} \ind_{ (\tu_{n-1}(\Gamma), \tu_n(\Gamma))}(t) \, \in \, \{-1,0,+1\}\,,
 \end{equation}
 where for this definition we have introduced $\tu_0(\Gamma)=0$. 
 In an informal way, $\sF_t=+1$ (respectively $-1$) if $t$ is in an 
\emph{ascending stretch} (respectively, in a  \emph{descending stretch}) and $\sFR_t=0$ if there is a $\gG$-extremum at $t$.

 \smallskip
 
 \begin{rem}
 \label{rem:NP}
 Note that the $\ttup_n(\gG)$ and $\ttdown_n(\gG)$ times are stopping times, but the ${\tt u}_n(\gG)$'s are not stopping times.
 Nevertheless they generate a remarkable renewal structure and relevant to us is notably that
  $((\tu_{n+1}(\gG)- \tu_n(\gG), \vert B_{\tu_{n+1}(\gG)}-B_{\tu_{n}(\gG)}\vert))_{n=1, 2, \ldots}$ is an IID sequence of 
  random vectors in 
  $(0, \infty)^2$ whose distribution is explicitly known \cite{cf:NP89}. In particular, both entries of the vector have exponentially decaying probability tails and the expectation is $( \gG^2,2\gG)$ {and the following scaling invariance holds: for every $\gG>0$
  \begin{equation}
  \left(\frac{\tu_{2}(\gG)- \tu_1(\gG)}{\gG^2}, \frac{\vert B_{\tu_{2}(\gG)}-B_{\tu_{1}(\gG)}\vert}\gG\right) \stackrel{\mathrm{law}}=
(\tu_{2}(1)- \tu_1(1), \vert B_{\tu_{2}(1)}-B_{\tu_{1}(1)}\vert)\,.
  \end{equation}}
   \end{rem}
 
 \smallskip
 
 The claim in \cite{cf:FLDM01} is that, in the limit $J\to \infty$, the spin configuration  is fully captured by  
 $s^\ssup{F}_\cdot$. The computations are performed by assuming precisely ${\bf s}_\cdot=s^\ssup{F}_\cdot$,
 but this is clearly impossible because the thermal fluctuations induce wall fluctuations, i.e. fluctuations of the loci where the spin domains switch signs, of at least $O(1)$. The authors of \cite{cf:FLDM01} are aware of this fact and they typically do not push their claims beyond what is expected to hold true (and, as we will explain in Section~\ref{sec:exactresults}, for this model one can compute exactly the free energy for every $J$: this can of course be exploited in a number of ways, in particular as a sanity    
check). 
 
 \smallskip
 
 Here is our main result:
 
\medskip

  \begin{theorem}
 \label{th:main}  
  We set $\alpha=0$.   For every $\gG>0$ and for   almost every Brownian trajectory $B_\cdot$ the following limit  exists
  \begin{equation}
  \label{eq:main-3}
  \lim_{\ell \to \infty}
  \frac1\ell \int_0^\ell  \ind_{{\bf s}_t \neq \sF_t} \dd t\, =:\, D_\gG\, ,
  \end{equation} 
  in $ \mu_{\gG, B_\cdot, \ell}$-probability. Moreover $D_\gG$ is a positive constant  and 
   \begin{equation}
   \label{eq:main-2}  
   D_\gG \overset{\gG \to \infty}= O \left(\frac {\loglog \gG}{\Gamma}\right)\, ,
   \end{equation} 
   where
   $\loglog (\cdot):= \log (\log(\cdot))$.  
 \end{theorem}

 \medskip

 A formula for $D_\gG$ appears in \eqref{eq:formain1}.

 In order to properly discuss  Theorem~\ref{th:main}, in particular explaining the relations with the existing literature and exposing the open issues that it raises, we introduce some of the mathematical tools we exploit. We just anticipate that
 $\liminf_{\gG\to\infty} \gG D_\gG \ge 1/2$ (see Remark~\ref{rem:overlap}).

\subsection{A two sided view of the model}
\label{sec:defs}
It is helpful from the technical viewpoint to consider the model on 
$[a, b]$, any $a<b$, instead of simply $[0, \ell]$. In view of this generalized setup it is natural to introduce
the standard bilateral Brownian motion
$B=(B_t)_{t \in \bbR}$.


For the general case of the model on $[a,b]$ the free process is $({\bf s}_t)_{t\ge a}$ with 
${\bf s}_t= (-1)^{N_{t-a}}$. 
The partition function of the Continuum RFIC on $[a,b]$ becomes
\begin{equation}
\label{eq:Z}
 Z_{\gG, B_\cdot, a,b}\,:=\,  \E_\Gamma \left[ \exp\left( \int_{a}^b {\bf s}_t \dd B_t \right) ; \, {\bf s}_b=+1  \right]\, ,
 \end{equation}
 and we note that we  set $\ga=0$. The stochastic integral in \eqref{eq:Z} should be understood as $\int_{0}^{b-a} {\bf s}_{a+t} \dd B^\ssup{a}_t$, with
 \begin{equation} \label{def:Bat} 
  B_t^\ssup{a}\,:=\,  B_{t+a}-B_a, \qquad t\ge 0. 
 \end{equation}  
 The associated Gibbs measure is denoted by $\mu_{a,b}= \mu_{\gG, B_\cdot, a, b}$.

 The reason to prefer this setting is because it permits to explain (and exploit) in a more natural way a symmetry that is present in the system and that we present now. It will be practical to work also with $\eta_t:=(1-{\bf s}_t)/2 \in \{0, 1\}$. In particular
\begin{equation}
\label{eq:ZBab}
 Z_{\gG, B_\cdot, a,b }\,:=\,  e^{B_b- B_{a}}\E_\Gamma \left[ \exp\left( -2\int_{a}^b \eta_t \dd B_t \right) ; \, {\eta}_b=0  \right]\, ,
 \end{equation}
 and $Z_{\gG, B_\cdot, a,b}e^{B_{a}-B_b}$ is an equivalent partition function since $ e^{B_a- B_{b}}$ does not involve the spin configuration, hence it yields the same Gibbs measure $\mu_{\gG, B_\cdot, a, b}$. 

By the Markov property of the Poisson process  we have for every $t\in (a, b)$
\begin{multline}
\mu_{a,b}\left( {\bf s}_t=+1\right)\, =\, \mu_{a,b}\left( \eta_t=0\right)\, =\\
\frac{
 \bE_\Gamma \left[ \exp\left( -2\int_{a}^t \eta_s \dd B_s \right);\, \eta_t =0   \right]
 \bE_\Gamma \left[ \exp\left( -2\int_t^b \eta_s \dd B_s \right);\, \eta_b =0\Big\vert \eta_t=0   \right]
}
{ \bE_\Gamma \left[ \exp\left( -2\int_{a}^b \eta_s \dd B_s \right);\, \eta_b =0   \right]}\, ,
\end{multline}
and by decomposing the denominator according to whether  $\{\eta_t=0\}$ or $\{\eta_t=1\}$ 
we see that 
\begin{equation}
\label{eq:muelleta}
\mu_{a,b}\left( {\bf s}_t=+1\right)\, =\, \frac{L^\ssup{a}_t R^\ssup{b}_t}{1+ L^\ssup{a}_t R^\ssup{b}_t}\, ,
\end{equation}
 where we introduced
\begin{equation}
\label{eq:L}
L^\ssup{a}_t\,  =\,  \frac
{ \bE_\Gamma \left[ \exp\left( -2\int_{a}^t \eta_s \dd B_s \right);\, \eta_t =0   \right]}
{\bE_\Gamma \left[ \exp\left( -2\int_{a}^t \eta_s \dd B_s \right);\, \eta_t =1   \right]}
\, \quad {\text{for every } t>a},
\end{equation}
and 
\begin{equation}
\label{eq:R}
R^\ssup{b}_t\, =\, 
\frac
{
 \bE_\Gamma \left[ \exp\left( -2\int_t^b \eta_s \dd B_s \right);\, \eta_b =0\Big\vert \eta_t=0   \right]
}
{
 \bE_\Gamma \left[ \exp\left( -2\int_t^b \eta_s \dd B_s \right);\, \eta_b =0\Big\vert \eta_t=1   \right]
} \, \quad {\text{for every } t<b}.
\end{equation}

We note that in \eqref{eq:L}, $\bP_\Gamma(\eta_a=0)=1$. 
There is a symmetry between $L^\ssup{a}_\cdot$ and $R^\ssup{b}_\cdot$ that is somewhat hidden by the expressions \eqref{eq:L} and \eqref{eq:R} and we present it here by pointing out that 
(see  Section~\ref{sec:SDE}):
\smallskip

\begin{itemize}
\item
$L_\cdot=L^\ssup{a}_\cdot$ is the unique strong solution of the SDE driven by the bilateral Brownian motion $B_\cdot$
  \begin{equation}  
  \label{eq:Lsde}
  \dd L_t\, =\, 2   L_t \dd B_t + \left(\gep (1-L_t^2)+  2L_t \right) \dd t\, ,
    \end{equation}
for $t \ge a$, initial condition $L_{a}= \infty$ and we recall that $\gep= e^{-\gG}$.
\item
$(R^\ssup{b}_{-t})_{t\ge -b}$ solves the same SDE, with the very same initial conditions, but with $B_\cdot$
replaced by $(-B_{-t})_{t \in \bbR}$. Explicitly, if we set $\mathrm{R}_t:=R^\ssup{b}_{-t}$ and 
$\Brev_t:=B_{-t}$
\begin{equation}
\label{eq:Rsde}
\dd \mathrm{R}_t
\, =\, -2   \mathrm{R}_t \dd \Brev_t + \left(\gep (1-\mathrm{R}_t^2)+  2\mathrm{R}_t\right) \dd t\,,
\end{equation}
and we solve this equation for $t\ge -b$ with $ \mathrm{R}_{-b}=\infty$.
\end{itemize}

\smallskip

While a priori the fact that $L_{a}=\mathrm{R}_{-b}=\infty$ may look like a source of problems and this issue is treated in Sec.~\ref{sec:SDE}, we anticipate that
$\tilde L_t= 1/L_t$ solves the same SDE, but with $B_t$ replaced by $-B_t$ and (above all)  initial condition $0$. Therefore  $\tilde L_\cdot$ falls into the standard existence and uniqueness framework. 

\smallskip

\begin{rem}
\label{rem:a=-b} This symmetry becomes more transparent if one considers the case $b=-a=\ell$ and if one thinks of the statistical mechanics origin of the model. 
Moreover  \eqref{eq:Lsde} and \eqref{eq:Rsde}  imply that 
$(L^\ssup{-\ell}_t)_{t\ge -\ell}$ and $(R^\ssup{\ell}_{-t})_{t\ge -\ell}$ have the same law.
\end{rem}

\smallskip
    
It is not difficult to  show (see \cite{cf:CGG}) that  the SDE \eqref{eq:Lsde} (and of course \eqref{eq:Rsde})  is reversible with respect to its unique invariant density: this is more practically presented in terms of
\begin{equation}
\label{eq:l}
l^\ssup{a}_t\, :=\,  \log L^\ssup{a}_t \\ \text{ and } \ \ r^\ssup{b}_t\, :=\,  \log R^\ssup{b}_t ,
\end{equation}
for which we have
\begin{equation}
\label{eq:l-SDE}
\dd l^\ssup{a}_t \, =\,  - U_\gG'\left(l^\ssup{a}_t\right) \dd t + 2 \dd B_t\, ,
\end{equation}
with the strongly confining potential $U_\gG(x):= \gep ( \exp(x)+ \exp(-x))$. Therefore the  density 
of the unique invariant probability $p_\gG$ of \eqref{eq:l-SDE} is 
\begin{equation}
\label{eq:pGa}
x \mapsto \frac 1{2 K_0(\gep)}\exp\left(-\frac 12 U_\gG(x)\right)\, =\, \frac 1{2 K_0(\gep)}\exp\left(-\frac \gep 2 \left( e^x+e^{-x}\right)\right)\, ,
\end{equation}
where $K_0(\gep)$ is the modified Bessel function of second kind with index $0$ evaluted in $\gep= e^{-\gG}$
(given explicitly in \eqref{eq:BesselK}, see \cite[Ch.~10]{cf:DLMF}), in particular $2 K_0(\gep)= 2\gG+2 (  \log 2- \gamma_{\textrm{EM}})+O(\gep^2 \log \gG)$ for $\gG \to \infty$, where $\gamma_{\textrm{EM}}$ is the Euler-Mascheroni constant. 
Therefore this density is close for $\gG$ large to the indicator function of $(-\gG, \gG)$, divided by the normalization $2\gG$. Note that  
$l^\ssup{a}_\cdot$ is nearly a Brownian motion (times $2$) as long as it keeps away from $\pm \log \gep= \mp \gG$. 
When it approaches $\pm \gG$ and, even more, when it tries to go far from $(-\gG, \gG)$, a strong recalling force acts on it. 
Analogous observations of course apply to $r^\ssup{b}_\cdot$.

\begin{figure}[ht]
\begin{center}
  \includegraphics[width=0.8\linewidth]{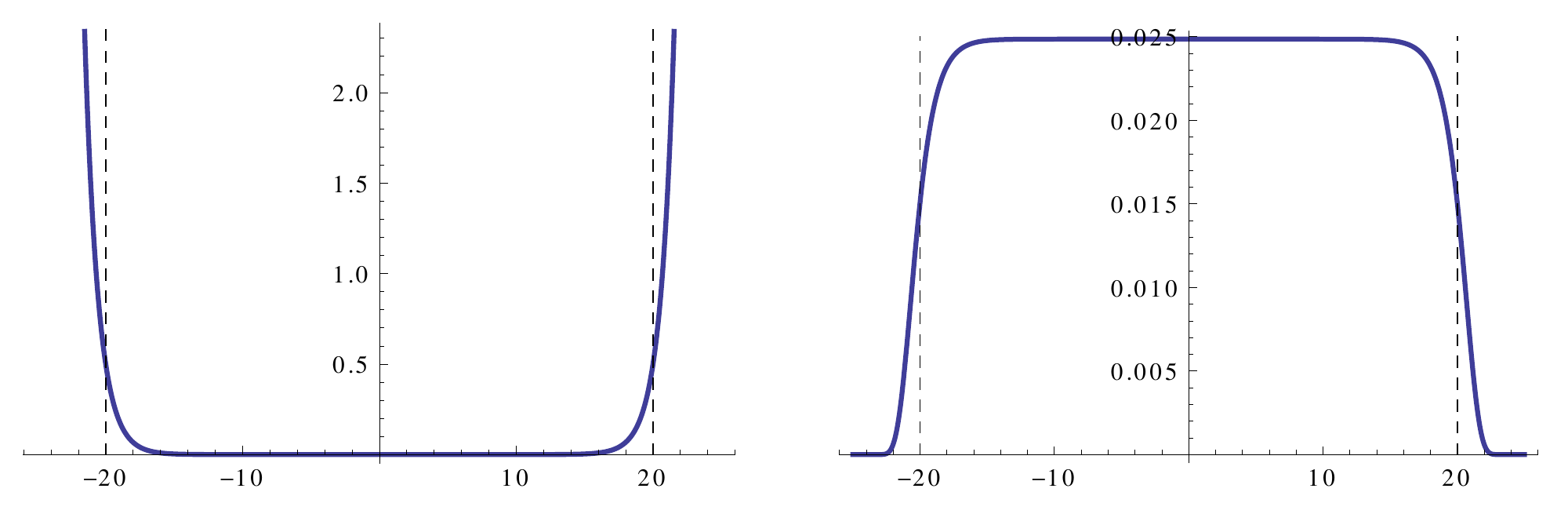}
  \end{center}
  \caption{
  \scriptsize The plot of the potential $U_\gG$ on the left and of the invariant density $p_\gG$ on the right, both for $\gG=20$. The strong proximity of $U_\gG$ with $\infty \ind_{[-\gG, \gG]^\complement}$ and of $p_\gG$ with  the uniform measure on $[-\gG, \gG]$ motivates the introduction of the simplified model of Section~\ref{sec:simplified} that is going to be an important tool in our analysis.}
  \label{fig2}
\end{figure} 

It is also useful to introduce $m^\ssup{a,b}_t:= l^\ssup{a}_t+r^\ssup{b}_t$ and, with this notation, \eqref{eq:muelleta} has the particularly compact expression
\begin{equation}
\label{eq:muelleta2}
\mu_{a,b}\left( {\bf s}_t=+1\right)\, =\, \frac{1}{1+ \exp\big(-m^\ssup{a,b}_t\big)}\ \ \ \text{ and } \ \ \
\mu_{a,b}\left( {\bf s}_t=-1\right)\, =\, \frac{1}{1+ \exp\big(m^\ssup{a,b}_t\big)}\, .
\end{equation}
\smallskip

Without surprise, see \cite{cf:CGG},
the law of $l^\ssup{a}_t$ (respectively $r^\ssup{b}_t$) converges for $t\to \infty$ (respectively, $t\to - \infty$) to $p_\gG$. 
In Section~\ref{sec:SDE} we are going to show the following related result:
\smallskip

\begin{lemma}
\label{th:invtime}
For every $t\in \bbR$ and 
for almost every $B_\cdot$ we have that
the almost sure limits $\lim_{a \to -\infty} l^\ssup{a}_t \, =:\, l_t$ and  
$\lim_{b \to \infty} r^\ssup{b}_t \, =:\, r_t$ exist. 
\end{lemma}
\medskip

Lemma~\ref{th:invtime} and several related results are included in Proposition~\ref{th:lonR}: in particular the fact that
the law of $l_t$ and of $r_t$ is $p_\gG$ for every $t$. Moreover 
 $(l_t)_{t\in \bbR}$, respectively $( \mathrm{r}_t)_{t \in \bbR}$, $\mathrm{r}_t:= r_{-t}$, can be characterized as
the only \emph{strong} solution for $t \in \bbR$ of 
\begin{equation}
\label{eq:l-SDE1}
\dd l_t \, =\, -U_\gG'\left(l_t\right) \dd t +2 \dd B_t\, ,
\end{equation}
respectively of
\begin{equation}
\label{eq:r-SDE1}
\dd \mathrm{r}_t \, =\, -U_\gG'\left(\mathrm{r}_t\right) \dd t -2 \dd \Brev_t\, ,
\end{equation}
and we recall that  $\Brev_t:= B_{-t}$.

Of course we set 
\begin{equation}
\label{eq:m}
m_t\,= \, l_t+r_t\, .
\end{equation}

\smallskip

Now we are going to state another fundamental result (proven in Section~\ref{sec:maintech}).
For this result we need to introduce also
the sequence of $\gG$-extrema of a bilateral Brownian motion. At this stage we do this informally: Appendix~\ref{sec:alt}
is fully dedicated to this issue. But the key point is that
by defining the $\gG$-extrema times starting from time $0$, as we did, we create a sequence
$(\tu_n(\gG))_{n=1,2, \dots}$ in which $\tu_1(\gG)$ is different in nature from  $\tu_2(\gG), \tu_3(\gG), \ldots$
because the latter ones have an ascending $\gG$-stretch on one side and a descending one on the other.
Instead,  $\tu_1(\gG)$ has a $\gG$-stretch only on the right.
If we repeat the procedure starting from an arbitrary time $a<0$, instead of $0$, the time $\tu_1(\gG)$ found starting 
from $0$ may not be in the sequence of $\gG$-extrema that one finds starting from $a$. Nevertheless,
$\tu_2(\gG), \tu_3(\gG), \ldots$ are $\gG$-extrema times in $(\tu_1(\gG), \infty)$ even if we define them starting from $a<0$. 
This stability allows to define the full sequence $(\tu^\ssup{\r}_{n}(\gG))_{n \in \bbZ}$ (the indexation is chosen by requiring that $\tu^\ssup{\r}_{0}(\gG)<0$ and $\tu^\ssup{\r}_{1}(\gG)\ge 0$)
of the $\gG$-extrema of a bilateral Brownian motion (see Appendix~\ref{sec:alt} for details). 
The point is that the sequence of the $\gG$-extrema of a bilateral Brownian motion partitions $\bbR$ into ascending $\gG$-stretches, descending $\gG$-stretches and the discrete set of $\gG$ extrema. Therefore we can 
define $(\sFR_t)_{t\in\r}$, the Fisher trajectory for the bilateral Brownian motion: for every $t\in \bbR$, $\sFR_t=+1$ (respectively $-1$) it $t$ is in an ascending (respectively descending) stretch, and $\sFR_t=0$ if there is a $\gG$-extremum at $t$. In a formula we obtain 
\begin{equation}
\label{eq:sFRalt0}
 \sFR_t\, =\,  \sum_{n=-\infty}^\infty \arrow^\ssup{\r}_{n} \ind_{ \left(\tu^\ssup{\r}_{n-1}, \tu^\ssup{\r}_{n}\right)}(t)
 \,,
 \end{equation}
 that implicitly defines $(\arrow^\ssup{\r}_{n})_{n \in \bbZ}$.
 For the next result, 
 {which is a reformulation of the first part of Theorem \ref{th:main} with a more explicit expression for the constant $D_\gG$,}
 we actually need only $\sFR_0$ and part of the arguments in Appendix~\ref{sec:alt} are about recovering $\sFR_0$ without building the whole sequence.

\medskip

\begin{proposition}
 \label{th:formain} 
 For every $\gG>0$ the following limit exists
 \begin{equation}
 \label{eq:formain1}  
   \lim_{\ell\to\infty}\e\left[\frac1\ell \int_0^\ell  \mu_{\gG, B_\cdot, \ell}\left({\bf s}_t \neq \sF_t\right) \dd t \right]
   \, = \, \bbE\left[ \frac{1}{1+ \exp\big(m_0\sFR_0\big)}\right]\, = : D_\gG \,.
   \end{equation}
 Furthermore, for every $\gG>0$ and for almost every $B_\cdot$, 
\begin{equation}
 \label{eq:formain2}  
   \frac1\ell \int_0^\ell   \mu_{\gG, B_\cdot, \ell}\left({\bf s}_t \neq s^\ssup{F}_t\right) \dd t \, \underset{\ell \to\infty} \longrightarrow D_\gG
   \end{equation}
   and, even more, in $ \mu_{\gG, B_\cdot, \ell}$-probability,
  \begin{equation}
  \label{eq:formain3}
  \frac1\ell \int_0^\ell  \ind_{{\bf s}_t \neq \sF_t} \dd t\, \underset{\ell \to\infty} \longrightarrow D_\gG\, .
  \end{equation} 
\end{proposition}

\medskip
 
Proposition~\ref{th:formain} is an ergodic type statement that is intuitive in the light of the identities in \eqref{eq:muelleta2}, so that, for example in order to establish \eqref{eq:formain2}, we are interested in the $\ell \to \infty$ limit of 
\begin{equation}
 \label{eq:formain-te}  
 \frac1\ell \int_0^\ell  
 \frac{1}{1+ \exp\big(m^\ssup{0, \ell}_t \sF_t\big)}  
 \dd t  \,.
\end{equation}
We stress that using the Dominated Convergence Theorem \eqref{eq:formain1} can be recovered from \eqref{eq:formain2} which is itself implied by \eqref{eq:formain3}.

\smallskip
\begin{rem}
\label{rem:pointwise}
From \eqref{eq:muelleta2} and Lemma~\ref{th:invtime} we readily see that $\mu_{a,b}({\bf s}_t \neq s^\ssup{F}_t)$ converges a.s. as $a \to - \infty$ and $b \to +\infty$ to 
$\left(1+ \exp\big(m_t\sFR_t\big)\right)^{-1}$. If we call  $\mu_{-\infty,+\infty}({\bf s}_t \neq \sFR_t)$ the limit of $\mu_{a,b}({\bf s}_t \neq s^\ssup{F}_t)$,  
 by Dominated Convergence and by applying the argument also to $B^\ssup{t}_\cdot$ {- defined in \eqref{def:Bat} -} we have that for every $t$
\begin{equation}
\bbE \left[ \mu_{-\infty,\infty}\left({\bf s}_t \neq \sFR_t\right)\right]\, =\, 
 \bbE \left[ \mu_{-\infty,\infty}\left({\bf s}_0 \neq \sFR_0\right)\right]\, =\, 
\, \bbE\left[\frac1{1+ \exp\big(m_0\sFR_0\big)}\right]\, = \,   D_\gG.
\end{equation}
This provides a pointwise result that complements Theorem~\ref{th:main}: it says for example that 
 $\lim_{\gG \to \infty}\bbP(\mu_{\gG, B_\cdot, -\infty,\infty}\left({\bf s}_t \neq \sFR_t \right)> D'_\gG)=0$ whenever 
 $\lim_{\gG\to\infty} D'_\gG/D_\gG= \infty$.
\end{rem}
\smallskip

We remark also that $l_0$ and $r_0$ are independent, because $l_0$ depends only on $(B_t)_{t \le 0}$ and 
 $r_0$ depends only on $(B_t)_{t \ge 0}$. On the other hand, $l_0$ and $r_0$ have the same law $p_\gG$. However this is not of much help to evaluate $D_\gG$. Rather, we will aim at making the $B_\cdot$ dependence  of 
 $l_0$ and $r_0$ explicit, so that we will understand the joint behavior of 
 the random variables $m_0$ and $\sFR_0$ and, in turn, the behavior of $D_\gG$.
 
 We complete this section by remarking that, by the expression for $D_\gG$ in  \eqref{eq:formain1}, 
  showing that  $D_\gG$ tends to zero amounts to showing  that (for 
  $\gG\to \infty$ and with high probability) $m_0$ is very large and positive (respectively negative)
when $\sFR_0=+1$ (respectively, $\sFR_0=-1$). 
This is handled in Section \ref{sec:proof-main}.

\subsection{About the free energy and the infinite disorder RG fixed point}
\label{sec:exactresults}

A remarkable fact for this model is that the free energy density $\tf_\ga(\gG)$ is explicit \cite{cf:MW1,cf:Luck-book,cf:CTT,cf:CGG}.
The free energy density is defined up to an additive constant: if we multiply the partition function by 
a factor that does not depend on ${\bf s}_\cdot$, but which may depend on the parameters of the system ($\gG$, the system size $\ell$,  the disorder $B_\cdot$) the Gibbs measure, i.e. the model itself, is not modified.  In this sense and just for the sake of this discussion we introduce
\begin{equation}
\label{eq:calZ0}
 \cZ^{\mathtt{f}}_{\gG, B_\cdot, \ell}\,:=\, \E_0 \left[ \exp\left(\ell -\gG N_\ell+\int_{0}^\ell {\bf s}_t \left(\dd B_t+\alpha \dd t\right)\right)    \right]\, ,
 \end{equation}
 and analogous formula for \eqref{eq:Z0+}. We recall that, under $\bP_0$, $\eta_\cdot$ is the standard Poisson process of rate $\exp(-0)=1$, see \eqref{eq:Z0}, and   one directly checks that
 $ \cZ^{\mathtt{f}}_{\gG, B_\cdot, \ell}/Z^{\mathtt{f}}_{\gG, B_\cdot, \ell}= \bE_0[ \exp(\ell - \gG N_\ell)]=
 \exp(\gep \ell)$. So we define the free energy density $\tf_\ga(\gG)$ to be the  exponential rate of growth of 
 $ \cZ^{\mathtt{f}}_{\gG, B_\cdot, \ell}$. More precisely  we have that (a.s. and in $\bbL^1$, see \cite{cf:CGG})
\begin{equation}
\label{eq:fe}
\tf_\ga(\gG)\,:=\,
\lim_{\ell \to \infty} \frac 1 \ell \log  \cZ_{\gG, B_\cdot, \ell}\, =\, \lim_{\ell \to \infty} \frac 1 \ell \log  \cZ^{\mathtt{f}}_{\gG, B_\cdot, \ell}\, =\, \ga  + \frac{e^{-\gG}K_{\ga -1}\left( e^{-\gG}\right)}{K_{\ga }\left( e^{-\gG}\right)} \, ,
 \end{equation}
where 
\begin{equation}
\label{eq:BesselK}
K_\ga(x)\, =\, \frac 12 \int_0^\infty \frac 1 {y^{1+ \ga}} \exp \left(-\frac x 2 \left( y +\frac 1y \right) \right)\dd y\, ,
\end{equation}
is the modified Bessel function of second kind of index $\ga$ evaluated in $x \in (0, \infty)$. 
This can be seen in various ways, in particular with the formalism we have introduced it is not difficult 
 to see that for $\ga=0$ a.s.
\begin{equation}
\label{eq:onlyonesided}
\tf_0(\gG)\, =\,   \lim_{\ell \to \infty} \frac {1 } \ell \int_0^\ell {e^{-\gG}} \exp(-l^\ssup{0}_t) \dd t
\, =\,  {e^{-\gG}} \int_\bbR   e^{-l} p_\gG(l) \dd l\, ,
 \end{equation}
 where for the first equality see Remark~\ref{rem:Lyap12}  and for the second one
we have used  that the law of $l^\ssup{0}_t$ converges to $p_\gG$ when $t\to \infty$  \cite{cf:CGG}. By using the explicit expression for $p_\gG$, given in \eqref{eq:pGa}, we recover \eqref{eq:fe} for  $\ga=0$. Moreover,  by exploiting the known Bessel asymptotic behaviors \cite{cf:DLMF,cf:CGG}, we see that
\begin{equation}
\label{eq:fe0}
\tf_0(\gG)\,=\,
 \frac{e^{-\gG}K_{ -1}\left( e^{-\gG}\right)}{K_{0 }\left( e^{-\gG}\right)} 
\overset{\gG \to \infty}= 
\frac 1 {\gG + \log 2 - \gamma_{\textrm{EM}}}+ O\left(\exp(-2\gG) \right)\, . 
 \end{equation}
 where $\gamma_{\textrm{EM}}=0,5772\ldots$ is the Euler-Mascheroni constant. We refer to  \cite{cf:GG22} for the analogous result 
 for (discrete) RFIC.
 
 What we just presented, see in particular \eqref{eq:onlyonesided},  shows that, in order to compute the free energy density, we just need the one-sided process $l_\cdot$ and, ultimately, the key is given by the invariant probability $p_\gG$ of $l_\cdot$.
If we want to get more information on the system we can differentiate the free energy with respect to the available parameters: for example for almost every $B_\cdot$
\begin{equation}
\label{eq:DW-formula}
\lim_{\ell \to \infty} \mu_{\gG, B_\cdot,  \ell} \left( \frac {N_\ell} \ell \right)\, =\,  
-\partial _\gG \tf_0(\gG) \, =\, \frac 1 {\left(\gG + \log 2 - \gamma_{\textrm{EM}}\right)^2}+ O\left(\exp(-\gG) \right)
\overset{ \gG \to \infty} \sim \frac 1 {\gG^2}\, .
\end{equation} 
where we recall that $N_\ell$ is the total number of Poisson events on $[0, \ell]$, so that \eqref{eq:DW-formula} yields the density  of domain walls. 
Two comments are in order:

\smallskip

\begin{enumerate}[leftmargin=*]
\item The asymptotic behavior in \eqref{eq:fe0} follows once again by known asymptotic behaviors  \cite[(10.25.2) and (10.27.4)]{cf:DLMF}. And one can push this computation further: notably, by taking one more derivative in $\gG$ one obtains that the variance, under $\mu_{\gG, B_\cdot,  \ell}$, of $N_\ell / \sqrt{\ell}$ converges for $\ell \to \infty$ a.s..
The limit is explicit in terms of Bessel functions.  
Therefore, the wall density concentrates around its mean value as $\ell$ tends to $\infty$.
\item A scaling argument yields that the free energy density for the model with $B_\cdot$ replaced by
$\gl B_\cdot$ is $e^{-\gG} K_{-1}(e^{-\gG}/ \gl^2) / K_{0}(e^{-\gG}/ \gl^2)$. Hence, by differentiating with respect 
to $\gl$, one obtains the \emph{density of disorder energy }
\begin{equation}
\lim_{\ell \to \infty}
\frac 1 \ell \mu_{\gG, B_\cdot,  \ell}\left( \int_0^\ell {\bf s}_t \dd B_t\right) \overset{\mathrm{a.s.}}=
\frac{\gep^2 \left(K_0(\gep)^2+K_2(\gep) K_0(\gep)-2 K_1(\gep)^2\right)}{K_0(\gep)^2} \overset{ \gG \to \infty} \sim \frac 2 \gG\, .
\end{equation}
\end{enumerate} 

\medskip

\begin{rem}
\label{rem:overlap}
A well-known integration by part argument frequently used in the statistical mechanics of disordered systems, see Lemma~\ref{th:overlap}, 
  shows that, when disorder is Gaussian (like in our case),
the disorder energy density is directly related to the \emph{overlap}. In our model the result is that ($\mu_\ell= \mu_{\gG, B_\cdot,  \ell}$)
\begin{equation}
\label{eq:overlap}
 \lim_{\ell \to \infty} \frac1\ell \e \left[ \mu^{\otimes 2}_\ell   \left(  \int_0^\ell \ind_{\{ {\bf s}_t^\ssup{1} \neq  {\bf s}_t^\ssup{2}\}} \dd t   \right)\right]\overset{\mathrm{a.s.}}= \frac 12
 \lim_{\ell \to \infty}
\frac 1 \ell \mu_{\ell}\left( \int_0^\ell {\bf s}_t \dd B_t\right) \overset{ \gG \to \infty} \sim \frac 1 \gG\, ,
 \end{equation}
  where ${\bf s}_\cdot^\ssup{1}, {\bf s}_\cdot^\ssup{2}$ denote two independent copies of ${\bf s}_\cdot$ under $\mu^{\otimes 2}_\ell$.  An interesting implication of \eqref{eq:overlap} follows by observing that ${\bf s}_t^\ssup{1} \neq {\bf s}_t^\ssup{2}$ implies that ${\bf s}_t^\ssup{1} \neq \sF_t$ or ${\bf s}_t^\ssup{2} \neq \sF_t$, so that $\mu^{\otimes 2}_\ell\left( {\bf s}_t^\ssup{1} \neq {\bf s}_t^\ssup{2}\right)\le 2  \mu_{  \ell}\left( {\bf s}_t\neq \sF_t\right)$.
  Therefore, recalling \eqref{eq:formain1}, we have
  \begin{equation}
  \label{eq:liminfDG}
  \liminf_{\gG \to \infty} \gG D_\gG \,\ge\,  \frac 1  2\, .
  \end{equation} 
\end{rem}

\medskip

As we have already pointed out, our main result (Theorem~\ref{th:main}) is a quantitative and mathematically rigorous version of the fact that D.~Fisher's infinite disorder fixed point does capture the leading order behavior of the domain-wall structure in the $\gG\to \infty$ limit. 
{This result does not allow to recover the extremely sharp control on some observables that one can extract from   the explicit formula for the free energy of the Continuous RFIC, but it is really different in nature:}

\smallskip

\begin{enumerate}[leftmargin=*]
\item
Via the explicit knowledge of  the free energy density of the Continuum RFIC, see 
 \eqref{eq:fe} and 
\eqref{eq:fe0}, the relevant observables can be computed explicitly
in this model (see for example \eqref{eq:DW-formula}). In   
   \cite{cf:FLDM01}  these observables are computed beyond leading order  in $\gG\to \infty$ using the infinite disorder fixed point and, by comparison with the exact results, one readily sees that the validity of the computations based on the infinite disorder fixed point does not  go beyond leading order {because of  the \emph{thermal noise}, which forces a positive discrepancy density $D_\gG$, with the lower bound in Remark~\ref{rem:overlap}}. This point is taken up again in Section~\ref{sec:open}.
\item But one  cannot infer, or even guess, the pathwise behavior from the knowledge of the free energy density. The pathwise behavior must of course be compatible  on large scale with the correct behavior of the observables computed by differentiating the free energy density, but the prediction of the pathwise behavior comes from D.~Fisher's RG procedure.  The RG procedure is highly nontrivial and does not depend at all  of the explicit, model dependent, free energy computations. 
\end{enumerate}

\smallskip

\begin{rem}
\label{rem:2}
 For a quantitative comparison of our results with the ones in  \cite{cf:FLDM01} one should keep in mind that the constant $\gG$ in  \cite{cf:FLDM01} is twice \emph{our} $\gG$. The value of $\gG$ chosen in   \cite{cf:FLDM01} is the interaction 
 energy paid to generate a new domain, i.e. two walls. But, for example, the domain density \cite[(45) and (39)]{cf:FLDM01} is $1/(2J)^2$ (times the variance of the disorder), as in our case. 
\end{rem}

\subsection{Open and related problems}
\label{sec:open}

\subsubsection{About more complete results for the Continuum RFIC} 
Our main result, Theorem~\ref{th:main}, controls the proximity of
${\bf s}_\cdot$
and $\sF_\cdot$ in a global sense and in probability (see Remark~\ref{rem:pointwise} for a pointwise result, always in probability). 
In Remark \ref{rem:cpmp-U} we will explain  that our results  say that the process $\sign(m_\cdot)$, $m_\cdot$ defined in \eqref{eq:m}, 
is a better approximation of ${\bf s}_\cdot$ 
than $\sF_\cdot$ is.
On the other hand, $\sF_\cdot$ is related in a much more elementary and intuitive way to the Brownian trajectory than $\sign(m_\cdot)$.

But of course $\sF_\cdot$ and $\sign(m_\cdot)$ are just functions of the disorder, while thermal noise is  present in the system, that is   
in ${\bf s}_\cdot$. As a matter of fact, as discussed  in Section~\ref{sec:exactresults},
the infinite disorder RG fixed point $\sF_\cdot$ captures the behavior of the free energy only up to the order $1/ \gG$ (and similar considerations for other
quantities of interest like the density of domain walls or the density of energy). It is not difficult to argue  
that the precise location  of the domain walls fluctuates on a spatial scale of order $1$ and, since the density of the domain walls is $1/ \gG^2$, these (thermal) fluctuations 
lead to a correction of order $1/\gG^2$ to the free energy density   and to the density of energy. But thermal noise may also affect domains that are \emph{borderline} from the energetic viewpoint. In fact, the domains we observe are the ones determined by the $\gG$-extrema and the absolute value of the energy (due to the disorder!) in a domain (i.e., the difference 
between the value of $B_\cdot$ at two neighboring $\gG$-extrema) is at least $\gG$: if this energy differs from $\gG$ just by a constant, the thermal noise may generate a global spin flip in the domain. It is not difficult to argue that such domain switch phenomena contribute to the free energy a term which is $o(1/ \gG^2)$, but they are  contributing in a more substantial $O(1/\gG)$ way to the overlap.

With respect to this issue,  \eqref{eq:fe0} (more precisely, just the second order expansion
$\tf_0(\gG)= 1/ \gG - ( \log 2 - \gamma_{\textrm{EM}})/ \gG^2 +o(1/ \gG^2)$)
  strongly suggests (via a non trivial computation that we do not reproduce here) that the domain-wall boundaries actually should behave in the $\gG\to \infty$ limit as diffusions in a (quenched) random potential given by a bilateral three-dimensional Bessel process.  Establishing such a result is well beyond what we  prove here, {but the considerations we just developed  lead to conjecturing that a finer approximation of the spin  configurations is obtained, starting from the Fisher configuration, by randomizing the location of the domain wall locations (according to independent diffusions in a
  $3$d Bessel environment), with the addition of another  mechanism that randomly  suppresses the domains whose disorder energy content is close 
  to $\gG$.}
   
%
%

Another direction in which Theorem~\ref{th:main} could be improved, notably in view of Theorem~\ref{th:mainsimp} for the simplified model (to which we hinted to in the caption of Fig.~\ref{fig2} and to which we dedicate Section~\ref{sec:simplified}), is understanding the behavior of $D_\gG$ in a sharper way. We do not know whether  $\loglog \gG$ 
 can be replaced by a constant.

\subsubsection{About the (discrete) RFIC} 
{After the completion of this work we  have been able to establish  \cite{cf:CGH} a result similar to Theorem~\ref{th:main} for the (discrete) RFIC. The central tool in the discrete setting is the transfer-matrix method that in the disordered systems context, see for example 
 \cite{cf:DH83,cf:NL86,cf:GGG17,cf:GG22}, leads to considering random matrix products and the corresponding dynamics induced on a naturally associated 
 projective space \cite{cf:Viana}.  We refer to  \cite{cf:CGH} for details and discussion about the relation with the present work. }

\subsubsection{About the $\ga \neq 0$ case}
 D.~Fisher, P.~Le Doussal and C.~Monthus  in \cite{cf:FLDM01} deal also with the case $\ga \neq 0$ (that we use here as a synonymous of $\bbE[h]\neq 0$). And the infinite disorder RG fixed point in the non centered case is considered in a number of other occasions, starting with the original works \cite{cf:F92,cf:F95}, see \cite{cf:IM,cf:Vojta,cf:GB} and references therein. 
 Strictly speaking, the limit process, or RG fixed point, for $\ga\neq 0$ is not the Neveu-Pitman process. But it is the process that describes the $\gG$-extrema (in terms of distances and height differences) of the Brownian motion with constant drift $\ga$: in this sense, it is the direct generalization of the Neveu-Pitman process (and it has been studied mathematically in \cite{cf:Faggionato}). Also in this case one is dealing with a process of independent random variables in $[0, \infty)^2$, but they are not IID, rather the sequence of the variables with even (respectively odd) indexes are IID. This is simply due to the fact that (let us suppose that $\ga>0$) ascending stretches are favored over the descending ones. The infinite disorder RG fixed point is explicit also for $\ga \neq 0$, but,
unlike for the $\ga=0$ case, it is clear that for $\ga \neq 0$ the infinite disorder RG fixed point gives just an approximate (or qualitative) picture of what really happens 
(mathematical arguments can be set forth in this direction, but this is implicitly admitted  also in the physical literature \cite{cf:IM,cf:Vojta} where the $\ga=0$ case is considered to be \emph{exact} and the $\ga\neq 0$ case  is claimed to be only \emph{asymptotically exact}, possibly referring to $\ga \to 0$). The  RG claims for $\ga\neq 0$ cannot hold because they do not agree
 with the  precise asymptotic behavior of the free energy 
density: these can be found in \cite{cf:DH83,cf:GGG17} for  the RFIC with $\ga \neq 0$  and  the free energy density of the Continuum RFIC is exactly known for every $\ga$, see \eqref{eq:fe}. On the other hand, as pointed out at the end of Section~\ref{sec:exactresults}, the knowledge of the free energy density, as complete as it may be, is very far from  
disclosing  the pathwise domain-wall structure. And understanding the domain-wall structure for the case $\ga\neq 0$ is, to our opinion and knowledge, an open issue.

\subsection{Organization of the paper}

We start, Section~\ref{sec:simplified}, with introducing and analyzing the \emph{simplified model},  {announced in the caption of Fig.~\ref{fig2}}.
In particular we prove Theorem~\ref{th:mainsimp} which is the analog of Theorem~\ref{th:main} in the simplified framework.
This section is definitely  useful for the intuition, but it is also needed because the simplified model 
enters the main proof as a tool. 

Our main result, Theorem~\ref{th:main}, is proven in Section~\ref{sec:proof-main}. This is a short section, because it relies on
two results, Lemma~\ref{th:modelsclose} and Proposition~\ref{th:formain}. 

Lemma~\ref{th:modelsclose} and Proposition~\ref{th:formain} are proven in Section~\ref{sec:maintech}, 
after the detailed analysis of the \emph{one-sided processes} $l_\cdot$ and $r_\cdot$ that can be found in Section~\ref{sec:SDE}. 

In App.~\ref{sec:B} we have collected several  properties on  Brownian motion: some are general and known, but some are less standard and more specific to our problem. In App.~\ref{sec:alt} we give characterizations of the sequence of $\gG$-extrema for the bilateral Brownian motion and more details about how this sequence is built. In App.~\ref{sec:C0} we show that the Continuum RFIC can be recovered as a scaling limit of the standard (i.e., discrete) RFIC. We also sketch the proof of the overlap formula 
\eqref{eq:overlap}.

\section{The simplified model}
\label{sec:simplified}

We now define the simplified versions of processes $l^\ssup{a}$ and $r^\ssup{b}$ {(informally mentioned in the caption of
Fig.~\ref{fig2})}, which we denote by $\ls^\ssup{a}$ and $\rs^\ssup{b}$. 
In this model the steps of our proofs are much less technical, but they correspond  in a precise way to the steps of the proof for the Continuum RFIC. In particular  Lemma \ref{th:invtime} and Proposition \ref{th:formain} (except for \eqref{eq:formain3}: this is related to Remark~\ref{rem:completereducedmodel?}) have strict analogs in this simplified context and one can estimate in a simple way the analog of the expression for $D_\gG$ in  \eqref{eq:formain1}.

\subsection{Definition and analysis of the simplified model}
\label{simpmodel} 

Recall that $U_\gG(x)=\varepsilon (e^x+e^{-x})$ is the potential confining processes $l^\ssup{a}$ and $r^\ssup{b}$, see Fig.~\ref{fig2}. We interpret it as a potential well with \emph{soft barriers} at $\Gamma$ and $-\Gamma$, as is attested by the shape of the invariant measure $p_\Gamma$. 
The present heuristic consists in replacing this pair of soft barriers by a pair of hard walls placed precisely at $\Gamma$ and $-\Gamma$. In other words, we rewrite equation \eqref{eq:l-SDE} and the analog for the $r$ process with a potential that is zero in the interval $[-\Gamma, \Gamma]$ and infinite outside of it. Let us give the formal definition of $\ls^\ssup{a}$ and $\rs^\ssup{b}$. Instead of starting $\ls^\ssup{a}$ at value $\infty$, we start it at $\Gamma$, since the infinite potential outside the pair of walls is supposed to bring $\ls^\ssup{a}$ instantly to $\Gamma$: we set $\ls^{\:(a)}_a =\Gamma$. 
Then, the process $\ls^\ssup{a}$ is constrained to remain in $[-\Gamma, \Gamma]$ ever after and $\left(\ls_t^{\:(a)}, \mathfrak{L}^\ssup{-}_t, \mathfrak{L}^\ssup{+}_t\right)_{t\ge a}$ is the unique solution of  the following Skorokhod problem (see \cite[equation (1.6)]{Kruk07})
 \begin{equation} 
 \label{eq:lsimp-SDE} 
   \ls_t^\ssup{a}\, =\, \Gamma + \mathfrak{L}^\ssup{-}_t- \mathfrak{L}^\ssup{+}_t+ 2  (B_t-B_a) \, \in [-\Gamma, \Gamma]\,  , \qquad t \,\ge\, a\, , 
\end{equation}
  where  $\mathfrak{L}^\ssup{-}_t$ and $\mathfrak{L}^\ssup{+}_t$, for $t\ge a$, are continuous non-decreasing functions that vanish at $a$ and satisfy 
  \begin{equation}
  \int_a^\infty \ind_{\left\{ \ls_t^\ssup{a}>-\Gamma\right\}} \dd \mathfrak{L}^\ssup{-}_t\,=\,0
  \ \ \text{ and } \ \ 
   \int_a^\infty \ind_{\left\{ \ls_t^\ssup{a}<\Gamma\right\}} \dd \mathfrak{L}^\ssup{+}_t=0 \, . 
  \end{equation} 
  
  We refer to formula (1.14) in \cite{Kruk07} for a representation of $\ls_t^\ssup{a}$ in terms of $(B_t)$.

  
  The reader who is not familiar with Skorokhod reflection problems should not get frightened at this point.  In fact, the process $\ls^{\:(a)}_\cdot$ may  be  understood as follows: $\ls^{\:(a)}$ is driven by the Brownian motion $B$ (times 2) except if it \emph{saturates}  at $\Gamma$ or at $-\Gamma$, in which case it is forbidden to cross these boundaries. 
  The process $\rs^{\:(b)}_\cdot$ behaves similarly but in reversed time: the process $(\rs_{- t}^\ssup{b})_{t\ge -b}$ is started at $\rs_{b}^\ssup{b}=\Gamma$ and is given by the solution of the Skorokhod problem \eqref{eq:lsimp-SDE}, except that $B$ is replaced by $-\Brev$ there, with the same constraint to stay in $[-\Gamma, \Gamma]$.
   We also define $\ms_\cdot^\ssup{a,b}$ by
 \begin{equation}
 \ms_t^\ssup{a,b}\, :=\, \ls_t^\ssup{a}+\rs_t^\ssup{b} \, ,
 \end{equation}
 for $ a \le t \le b$.
  These two simplified processes, and the associated $\ms$ process, constitute what we call the \emph{simplified model}.  
   
 \smallskip
 \begin{rem}
 \label{rem:completereducedmodel?}
 We do not know whether there exists a spin process analogous to process $(\mathbf{s}_t)_{t \in [a, b]}$ under law $\mu_{\gG, B_{\cdot},a,b}$ for which the probability that the spin at $t$ is equal to $+1$ would be equal to $(1+\exp(-\ms_t^\ssup{a,b}) )^{-1}$ (recall \eqref{eq:muelleta2}).
 \end{rem}
 \smallskip  
 
  The simplified model has the virtue of providing explicit expressions, which are related to the structure of the $\Gamma-$extrema of $B_\cdot$ in a very direct way: we are going to explore this  next.
   
   \smallskip
   
 We recall that we have introduced in  Section~\ref{sec:Gextrema} the spin trajectory  $\sFR_\cdot$ -- i.e.,  Fisher infinite disorder RG fixed point over the whole of $\bbR$ -- via a global procedure, see in particular \eqref{eq:sFRalt0} (this is detailed in   App.~\ref{sec:alt-1}). 
   As hinted to just before Proposition~\ref{th:formain}, 
   $\sFR_\cdot$ may
   be introduced also via a local procedure, see App.~\ref{sec:alt-2}, in particular \eqref{eq:sFR}. For what we are doing now, it suffices to focus on the local approach.
   
   For conciseness, we derive the explicit expression of $\ls_0^\ssup{a}$, when $a$ is large and negative. This is the objective of the next lemma. Recall from  Section~\ref{sec:Gextrema} the definitions  of  $\ttup(\gG)$, $\ttdown(\gG)$, $\tuup(\gG)$, $\tudown(\gG)$, $\tu_1(\Gamma)$ and $\arrow_1$. We consider the analoguous quantities for the Brownian motion $\Brev$ and set: \begin{equation}
\label{eq:reversed-rv}
\tsup(\gG)\, :=\, -\ttup_{\Brev}(\gG)\, , \ \
\tsdown(\gG)\, :=\, -\ttdown_{\Brev}(\gG)\, , \ \ 
\tvup(\gG)\, :=\, -\tuup_{\Brev}(\gG)\, , \ \
\tvdown(\gG)\, :=\, -\tudown_{\Brev}(\gG)\, ,
\end{equation}
as well as $\tv_1(\gG)\, :=\, -\tu_{\Brev, 1}(\gG)$ and  $\arrow'_1:=\arrow_{\Brev, 1}$.
   \medskip
  
\begin{lemma}
\label{th:invtimesimp}
	For every $a \le \max(\tsup(\Gamma), \tsdown(\Gamma))$, $\ls_0^\ssup{a}$
	does not depend on the value of $a$, so we write $\ls_0$ for $\ls_0^\ssup{a}$ for $a\le \max(\tsup(\Gamma), \tsdown(\Gamma))$, and we have
  \begin{equation}
 \label{eq:invtimesimp}
  \ls_0\, = \arrow'_1 \gG - 2B_{\tv_1(\gG)} =
  \begin{cases}
   +\Gamma-2\sup_{[\tsdown(\Gamma),0]} B_\cdot&  \text{ if } \tsdown(\Gamma)>\tsup(\Gamma)\, ,
   \\
   -\Gamma-2\inf_{[\tsup(\gG),0]} B_\cdot &   \text{ if } \tsup(\Gamma)>\tsdown(\Gamma)\, .
  \end{cases}
  \end{equation}
  	Respectively, for every $b \ge \min(\ttup(\Gamma), \ttdown(\Gamma))$, $\rs_0^\ssup{b}=: \rs_0$ does not depend on $b$ and
  \begin{equation} \label{def-rs0}
  \rs_0 \, = -\arrow_1 \gG + 2B_{\tu_1(\gG)} =
  \begin{cases}
   -\Gamma+2\sup_{[0,\ttdown(\Gamma)]} B_\cdot & \text{ if } \ttdown(\Gamma)<\ttup(\Gamma)\, ,\\
   +\Gamma+2\inf_{[0,\ttup(\Gamma)]} B_\cdot&  \text{ if } \ttup(\Gamma)<\ttdown(\Gamma)\, .
  \end{cases}
  \end{equation}
  As a consequence, when $a \le \max(\tsup(\Gamma), \tsdown(\Gamma))$ and $b \ge \min(\ttup(\Gamma), \ttdown(\Gamma))$ we have
  \begin{equation}
  \ms^\ssup{a, b}_0 \, = \ms_0 : = (\arrow'_1 - \arrow_1)\gG +2(B_{\tu_1(\gG)}-B_{\tv_1(\gG)})\, ,
  \end{equation}
  and $\sFR_0=\sign(\ms_0)$ with the convention that $\sign(0)=0$.
  
  These results extend to every $t\in\mathbb{R}$ by replacing $B_\cdot$ by $B_\cdot^\ssup{t}$ everywhere, thus defining processes $t\mapsto \ls_t$, $t\mapsto \rs_t$ and $t\mapsto \ms_t$ on the whole of $\mathbb{R}$. In particular we have $\sFR_t=\sign(\ms_t)$ for all $t\in\r$.
\end{lemma}
  	
  This lemma should be considered as the corresponding version of Lemma \ref{th:invtime} for the simplified model, with the further advantage that, here, we are able to derive explicit expressions for the limiting processes and these expressions relate to the sequence of $\gG$-extrema.
  We have in particular that for $a\to-\infty$ and $ b\to \infty$
 \begin{equation}
 \label{eq:mhatab}
 \lim  \ms_t^\ssup{a,b}\, =\, \ls_t+\rs_t = \ms_t\, ,
  \end{equation}
  but in fact, from Lemma~\ref{th:invtimesimp} one can extract substantially stronger statements. For example that, for almost every $B_\cdot$ 
  and for every sufficiently large $\ell$   one can find $\ell_0$ (not depending on $\ell$) and $\ell_1(\ell)$, with $\lim_{\ell \to \infty} \ell_1(\ell)/ \ell =1$, 
  such that $ \ms_t^\ssup{0,\ell}= \ms_t$ for every $t \in [\ell_0, \ell_1(\ell)]$. 
  	
  \begin{proof}
  Let us assume for example that $\tsdown(\Gamma)>\tsup(\Gamma)$ and let us take $a$ such that $a\le\tsdown(\Gamma)$. 
  Recall that $\tvdown(\Gamma)$ is the (almost surely) unique time $\tv\in[\tsdown(\Gamma),0]$ such that  
  \begin{equation} \label{vdownGamma} B_{\tv}\, =\, \sup_{[\tsdown(\Gamma),0]} B_\cdot\, .
  \end{equation} 
  We proceed in three steps (see Figure \ref{fig:compa}):
  \begin{enumerate}
  \item At time ${\tsdown(\Gamma)}$, we simply remark that $\ls^\ssup{a}_{\tsdown(\Gamma)}\in[-\Gamma, \Gamma]$.
  \item {We remark that the rise in $B_\cdot$ from time $\tsdown(\Gamma)$ to time $\tvdown(\Gamma)$ is of size $\gG$, that $\tvdown(\Gamma)$ is an upward record time after time $\tsdown(\Gamma)$ and that $B_\cdot$ exhibits no drop of size $\gG$ inside $[\tsdown(\Gamma),\tvdown(\Gamma)]$, hence we obtain that $\ls_{\tvdown(\Gamma)}^\ssup{a}=\Gamma$.
  \item On $[\tvdown(\Gamma),0]$, $B_\cdot$ remains not larger than $B_{\tvdown(\Gamma)}$ and exhibits no drop of size $\Gamma$, hence $\ls^\ssup{a}$ evolves according to $2B_\cdot$ and we obtain that 
  \begin{equation}
  \ls_0^\ssup{a}\,=\,\Gamma+2(B_0-B_{\tvdown(\Gamma)})\,=\, \Gamma-2B_{\tvdown(\Gamma)}\, ,
  \end{equation}  which is the announced value of $\ls_0$ in the case when $\tsdown(\Gamma)>\tsup(\Gamma)$.}
   \end{enumerate}
   
  {Let us make the argument more explicit, by taking an analytic approach. At time $\tsdown(\Gamma)$, $\ls$ is in $[-\Gamma, \Gamma]$ by definition. From this time on, $\ls$ evolves as $2B_\cdot$ until it hits $-\gG$ or $\gG$. By definition of $\tsdown(\Gamma)$ and $\tvdown(\Gamma)$, we have that $B_{\tvdown(\Gamma)}-B_{\tsdown(\Gamma)}=\gG$ and that for every $t\in[\tsdown(\Gamma), \tvdown(\Gamma)], B_t\ge B_{\tsdown(\Gamma)}$. Therefore, $t\mapsto \ls^\ssup{a}_{\tsdown(\Gamma)}+2(B_t-B_{\tsdown(\Gamma)})$ \emph{will not} hit $-\gG$ before time $\tsdown(\Gamma)$ but it \emph{will} hit $\gG$ before time $\tsdown(\Gamma)$. Introducing
  \begin{equation}
  \tau:=\inf\left\{t>\tsdown(\Gamma):\ls^\ssup{a}_{\tsdown(\Gamma)}+2(B_t-B_{\tsdown(\Gamma)})\in\{-\gG, \gG\}\right\}
  \end{equation}
  we get that $\tau$ is not larger than $\tvdown(\Gamma)$ and that $\ls^\ssup{a}_\tau=\ls^\ssup{a}_{\tsdown(\Gamma)}+2(B_\tau-B_{\tsdown(\Gamma)})=\gG$.
  Now, from time $\tau$ on, $\ls^\ssup{a}$ evolves as $2B_\cdot$ reflected at $\gG$, until it touches $-\gG$. Explicitly: $\ls^\ssup{a}_t=\gG-2\max_{s\in[\tau, t]} (B_s-B_t)$ until time 
  \begin{equation}
  \tau':=\inf\left\{t>\tsdown(\Gamma):\gG-2\max_{s\in[\tau, t]} (B_s-B_t)=-\gG\}\right\}
  \end{equation}
  Since we assumed $\tsdown(\Gamma)>\tsup(\Gamma)$, there is no drop of size $\gG$ in $B_\cdot$ between $\tsdown(\Gamma)$ and 0, \emph{a fortiori} between $\tau$ and 0. Therefore $\tau'$ is not smaller than 0,   we get $\ls_{\tvdown(\Gamma)}^\ssup{a}=\Gamma$ and 
  $\ls^\ssup{a}_0= \gG - 2\max_{s\in[\tau, 0]} (B_s-B_0)= \gG - 2B_{\tvdown(\Gamma)}.$
  }
  
  Replacing $B_\cdot$ by $-B_\cdot$, we get the corresponding result in the case when $\tsup(\Gamma)>\tsdown(\Gamma)$ because, as we pointed out, the result does not depend on the initial condition.
  The result for process $\rs^\ssup{b}$ follows by changing $B_\cdot$ into $-\Brev_\cdot$.
  Finally, the identity $\sFR_0=\sign(\ms_0)$  follows from the analysis in Appendix \ref{sec:alt}, specifically it is just a more concise version of \eqref{eq:sFR}.
  \end{proof}
  
    \begin{figure}[ht!]
\begin{center}
  \includegraphics[width=\linewidth]{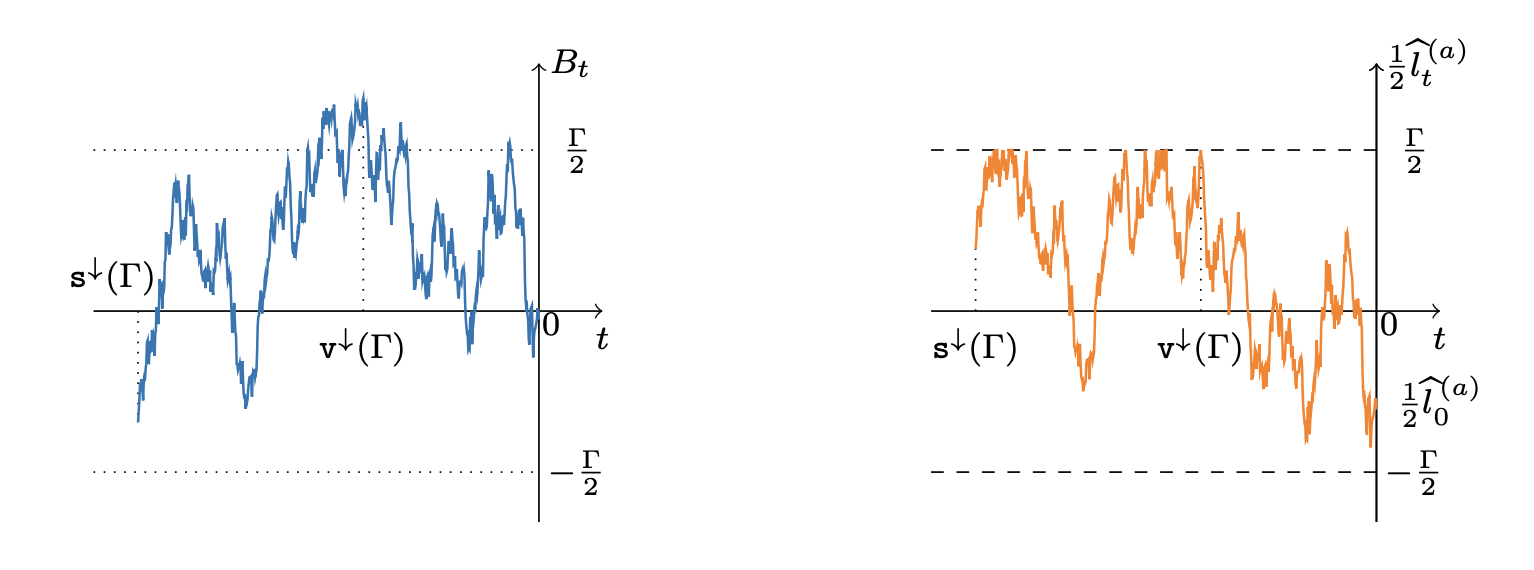}
  \end{center}
  \caption{
  \scriptsize  An illustration of Steps (1), (2) and (3) in the proof of Lemma \ref{th:invtimesimp}. In particular, $\ls_t^\ssup{a}\,=\,\Gamma+2(B_t-B_{\tvdown(\Gamma)})$ for any $t \in [\tvdown(\Gamma), 0].$  }
  \label{fig:compa}
  \end{figure} 
  
  
%


\medskip
  
  \subsection{Main theorem for the simplified model}
  Now we are going to state and prove  the analog of our main theorem (Theorem~\ref{th:main}) for the simplified model.  In fact, due to Remark \ref{rem:completereducedmodel?} we formulate only the analog of \eqref{eq:formain2} (the analog of \eqref{eq:formain1} follows), but the result is more precise in terms of the large $\gG$ dependence:
  \medskip
  
 \begin{theorem}
 \label{th:mainsimp}  
  For almost every Brownian trajectory $B_\cdot$, the following limit exists and does not depend on $B_\cdot$:
  \begin{equation}
   \label{eq:mainsimp}
  \lim_{\ell\to\infty}\frac1{\ell} \int_0^\ell \frac{1}{1+ \exp\big(\ms_t^\ssup{0, \ell}\sF_t\big)} \dd t \, =:\,  \widehat D_\gG\,.
   \end{equation}
   Moreover
   \begin{equation}
    \widehat D_\gG
  \overset{\gG \to \infty} \sim \frac{\log 2}{\gG}\, .
   \end{equation}
 \end{theorem}
 \medskip

\begin{proof}
  It is not difficult to show that the ergodic result \eqref{eq:formain2} holds for the simplified model, in the sense that 
  \begin{equation}
  \label{eq:formainsimp}
  	\lim_{\ell\to\infty}\frac1\ell \int_0^\ell \frac{1}{1+ \exp\big(\ms_t^\ssup{0, \ell}\sF_t\big)} \dd t\,=\, \bbE\left[ \frac{1}{1+ \exp\big(\ms_0\sFR_0\big)}\right] =:\widehat D_\gG \, .
  	\end{equation}
   We only sketch the proof of \eqref{eq:formainsimp}
    (because it is very close to the argument in the proof of \eqref{eq:formain2} itself, see Section~\ref{sec:proof-formain}).
   The idea is to show first
   that almost surely
 \begin{equation}
 \label{eq:preformainsimp}
  \lim_{\ell\to\infty}\frac1\ell \int_0^\ell \frac{1}{1+ \exp\big(\ms_t^\ssup{0, \ell}\sF_t\big)} \dd t\, =\, \lim_{\ell\to\infty}\frac1\ell \int_0^\ell  \frac{1}{1+ \exp\big(\ms_t \sFR_t\big)}\dd t\, ,
  \end{equation}
  by an adequate control of the domain where the two integrands differ: see \eqref{eq:mhatab} (and the observation right after it), moreover the substitution of $ \sF_t$ with $\sFR_t$ is also straightforward because they can possibly differ only when $t \in [0, \tu_1(\Gamma)]$  a.s.. 
  Finally, one goes from \eqref{eq:preformainsimp} to \eqref{eq:formainsimp} by applying (the time-continuous version of) Birkhoff's Ergodic Theorem.
  
  To conclude the proof of Theorem \ref{th:mainsimp}, we analyze the expectation in the right-hand side of \eqref{eq:formainsimp}. Observe that $\ls_0$ is measurable with respect to $\mathcal{F}((B_t)_{t\le 0})$, while  $\rs_0$ is measurable with respect to $\mathcal{F}((B_t)_{t\ge 0})$, hence they are independent. 
  Furthermore, they follow the uniform law on $[-\Gamma, \Gamma]$
  (see Lemma \ref{th:excthlemma}$(ii)$).
  Now the key point is that the simplified model captures exactly Fisher's strategy - or the other way around - in the sense that $\sFR_0=\sign(\ms_0)$, as stated in Lemma \ref{th:invtimesimp}. Therefore:
  \begin{equation}
  \label{eq:comp-U}
  \begin{split}
    \widehat D_\gG = &\, \bbE\left[ \left(1+ \exp\big(|\ms_0|\big)\right)^{-1}\right] \\
     = &\, \int_{-\Gamma}^\Gamma \int_{-\Gamma}^\Gamma
     \left(1+ \exp\big(|l+r|\big)\right)^{-1} \frac{\dd l}{2\Gamma} \frac{\dd r}{2\Gamma}\\ 
     = &\,  \frac{1}{4\Gamma^2}\int_{0}^{2\Gamma} 
     \left(1+ \exp\big(|s|\big)\right)^{-1} (2\Gamma-s)\dd s
     \\
     =& \frac{\log 2}{\Gamma} -\frac{\pi^2}{24\,   \gG^2} +\frac1{2\gG ^2} \mathrm{Li}_2(- \exp(-2\gG))\, ,
     \end{split}
   \end{equation}
   where $\mathrm{Li}_2(\cdot)$ is the Polylogarithm of order $2$ that has an analytic behavior at the origin, with  $ \mathrm{Li}_2(x)\sim x$ for $x$ near $0$ \cite[(25.12.10)]{cf:DLMF}.
\end{proof}

\section{Proof of Theorem~\ref{th:main} }
\label{sec:proof-main}

{We give the proof of Theorem~\ref{th:main} assuming the validity of 
 Lemma \ref{th:invtime} and Proposition \ref{th:formain}, stated in the introduction, and of 
 Lemma~\ref{th:modelsclose} which is stated below. 
 The first result is proven in the next section, see Proposition~\ref{th:lonR},   
 and to the other two we devote Section~\ref{sec:maintech}.
 A look at  Proposition \ref{th:formain} makes clear that Theorem~\ref{th:main} follows if we obtain an adequate upper bound
 on}
   the expression for $D_\gG$ given  in \eqref{eq:formain1}. As already discussed in Section~\ref{sec:defs} (and proven in Proposition~\ref{th:lonR}), the law of $(l_0,r_0)$ is simply $p_\Gamma \otimes p_\Gamma$, with $p_\gG$ defined in \eqref{eq:pGa}, but knowing the law of $(l_0,r_0)$ is not of much help because we need to relate $(l_0,r_0)$ and 
$\sFR_0$ in a pathwise sense, that is for (almost) every trajectory of
$B_\cdot=(B_t)_{t\in\R}$. In our study of the simplified model, we established that $\sFR_0$ is equal to  the sign of $\ms_0=\ls_0+\rs_0$. We are going to show that this remains \emph{essentially} true for the real model. This comes from the next lemma and it is made completely explicit in the corollary that follows it.

\medskip
\begin{lemma}
\label{th:modelsclose}
    For every $\gk>0$, there exists a constant $C_\gk$ such that, with probability $1-O(\Gamma^{-\gk})$, we have:
    \begin{equation}\label{lapproxls}
    \left| l_0-\ls_0\right |\, \le \, \loglog \Gamma + C_\gk\qquad \text{ and } \qquad \left| r_0-\rs_0\right|\, \le \,\loglog\Gamma+C_\gk\,.
    \end{equation}
    \end{lemma}
    
    \medskip
    
We postpone the proof of Lemma~\ref{th:modelsclose} to Section~\ref{sec:modelsclose}. Of course in 
 Lemma~\ref{th:modelsclose} time $0$ can be replaced with any time $t\in \bbR$: this is 
 just a matter of using $B_\cdot^\ssup{t}$ instead of $B_\cdot$ and the event underlying probability $1-O(\Gamma^{-\gk})$
 changes accordingly.
 
 Moreover, using that $\ls_0$ and $\rs_0$ are uniformly distributed over $[-\Gamma, \Gamma]$, from Lemma~\ref{th:modelsclose} for $\gk=1$ we directly obtain:
 
 \medskip

\begin{cor}\label{th:modelsclosecor}
	With probability $1-O\left({\loglog \Gamma}/{\Gamma}\right)$, $m_0$ has the same sign as $\ms_0$.
\end{cor}

\medskip

	Now, we are ready to finish the proof of our main theorem.

\smallskip
\begin{proof}[Proof of Theorem~\ref{th:main}] 
   We estimate the expression for $D_\gG$ in \eqref{eq:formain1}. We recall that with the convention $\sign(0)=0$
  \begin{equation} 
 \sFR_0\, =\, \sign \left(  \ms_0\right)\, ,
  \end{equation}
   and
    we introduce:
   \begin{equation}
   s^\ssup{m}_0\, = \, \sign\left(m_0 \right)\, .
   \end{equation}
     
We handle the expression for $D_\gG$ in \eqref{eq:formain1} \emph{when $\sFR_0$ is replaced by $s^\ssup{m}_0$ there}, so using the independence of $l_0$ and $r_0$:
    \begin{equation}
    \label{eq:useBessel2}
    \begin{split}
    \e\left[\left(1+\exp\big(m_0s^\ssup{m}_0\big)\right)^{-1}\right]
    & = \,\e\left[\left(1+\exp(|m_0|)\right)^{-1}\right]\\
    & =\, \int_{\mathbb{R}^2} \left(1+\exp(|l+r|)\right)^{-1} \;  p_\gG\otimes p_\gG(dl, dr)\\
    &=\,  \frac{\log 2}{\Gamma}+ o \left(\frac 1 \gG\right)\, ,
    \end{split}
    \end{equation}
   where in the last step
    we used that for every $\eta>0$, every $\gG$ sufficiently large and every $x \in \bbR$ we have $p_\gG(x) \ge (1-\eta) u_\gG(x(1+\eta))$, with $u_\gG$ the density of the uniform probability on $[-\Gamma, \Gamma]$, 
    and $p_\gG(x) \le u_\gG(x) + \ind_{(\gG, \infty)}(\vert x \vert) \exp(-(1/2) \exp(\vert x \vert - \gG))/(2\gG)$. 
    
    To conclude we use the fact that $(1+\exp(\cdot))^{-1}$ is bounded and Corollary \ref{th:modelsclosecor}, which yield
\begin{equation}
\e\left[\left(1+\exp\big(m_0\sFR_0\big)\right)^{-1}-\left(1+\exp\big(m_0 s^\ssup{m}_0\big)\right)^{-1}\right]\,=\, O\left(\frac{\loglog \Gamma}{\Gamma}\right)\, .
\end{equation}
    The proof of Theorem~\ref{th:main} is therefore complete.
    \end{proof}
    
    \medskip
    
 \begin{rem}
 \label{rem:cpmp-U}
As suggested by the above proof, instead of Fisher's trajectory $(\sF_t)_{t\in[0, \ell]}$ we can consider the trajectory $\left(\sign(m^\ssup{0, \ell}_t)\right)_{t\in[0, \ell]}$. Observe that this trajectory minimises the expectation 
 \begin{equation}
	\frac1\ell \int_0^\ell \mu_{\gG, B_\cdot, 0,\ell}\left({\bf s}_t\neq {S}_t\right)\dd t
 \end{equation}
 over all   stochastic processes $({S}_t)_{t\in[0, \ell]}$ defined in the same probability space as $B_\cdot$ and   taking values in $\{-1,0,+1\}$.
 By the  same argument of proof, also in this case $m^\ssup{0, \ell}_\cdot$ can be  replaced in the limit by $m_\cdot$ and the analog of Theorem \ref{th:main} holds with $D_\gG$ now equal to 
 \begin{equation}
 \e\left[\left(1+\exp\big(m_0s^\ssup{m}_0\big)\right)^{-1}\right] \,  = \,\e\left[\left(1+\exp(|m_0|)\right)^{-1}\right].
 \end{equation}
 \end{rem}
 
  \begin{rem}
 \label{rem:cpmp-U2}
 One can of course make the estimate \eqref{eq:useBessel2} precise to all orders: here we give the second order in the expansion. The  law of $l_0+r_0$, i.e. $p_\gG * p_\gG$, is 
 \begin{equation}
 \label{eq:convgG}
 p_\gG * p_\gG (x) \, =\, \frac{
 K_0\left( 2 \gep \cosh(x/2) \right)
 }{
 2\left( K_0( \gep)^2\right)
 }\,,
 \end{equation}
 which is an even distribution, close to the convolution of the uniform law over $[-\gG, \gG]$ with itself.
 Using $K_0(x)= \log (1/x)+ C+O(x^2 \log (1/x))$ for $x \searrow 0$, $C=\log 2 -\gamma_{\mathrm{EM}}=0.1159\ldots$,
 we obtain that
 \begin{equation}
 \label{eq:convgG-2}
 p_\gG * p_\gG (x)- \frac 1 {2 \gG}
 \left( 1- \frac x{2\gG}\right)
 \, =\, - \frac{C + \log\left( 1-e^{-x} \right)}{2 \gG^2}+ O\left( \frac{ \max(1, \vert x\vert)} {\gG^3}\right)\, , 
 \end{equation}
 so that, recalling \eqref{eq:comp-U} for the contribution due to the approximation with uniform laws, we obtain that the expression in  \eqref{eq:useBessel2} is equal to 
 \begin{equation}
  \frac{\log 2} {\gG} - \frac{\pi^2} {24\gG^2} - 
  \frac 1 {2\gG^2} \int_0^\infty 
  \left( 1+\exp(x)\right)^{-1}
  \left(C + \log\left( 1-e^{-x} \right)\right) + O\left( \frac 1{\gG^3}\right)\, .
 \end{equation}
 That is  
 \begin{equation}
 \e\left[\left(1+\exp(|m_0|)\right)^{-1}\right]\,=\,  \frac{\log 2} {\gG} - \frac 1 {\gG^2}\left( \frac {\pi^2}{24} + \frac32 (\log(2))^2 - \gamma_{\mathrm{EM}}\log 2 \right)
 + O\left( \frac 1{\gG^3}\right)\, .
 \end{equation}
 \end{rem}

\section{Analysis of the one sided processes}
\label{sec:SDE}

Recall the definitions of Section~\ref{sec:defs}, 
notably that $a<b$,  $L^\ssup{a}_t$ is defined in \eqref{eq:L} and that $R^\ssup{b}_t$ is defined in \eqref{eq:R}.
Also $\mathrm{R}_t:=R^\ssup{b}_{-t}$ and $\Brev_t:=B_{-t}$ are introduced in Section~\ref{sec:defs}.
{The notion of strong solution to a stochastic differential equation with time running from $0$ to $\infty$ is given for example in \cite[p.~366]{RY} and it directly generalizes to the case in which $t\in [s,\infty)$, for every $s\in \bbR$.
We say that $(L_t)_{t>a} $ is a strong solution of \eqref{eq:Lsde} if 
$(L_t)_{t\ge a'} $ is a strong solution of \eqref{eq:Lsde}
for every $a'>a$. In particular, $(L_t)_{t\ge a'}$ is adapted to the filtration $(\sigma(B_s-B_{a'}, a'\le s\le t ))_{t\ge a'}$.
Analogous definition for $(\mathrm{R}_t)_{t>-b}$.
}
\medskip

{
\begin{lemma}
\label{th:SDEs0} 
$(L_t)_{t> a} =(L^\ssup{a}_t)_{t>a}$ is a strong solution of \eqref{eq:Lsde}  and satisfies 
$\lim_{t\searrow a} L^\ssup{a}_t= \infty$ a.s.. 
$(\mathrm{R}_t)_{t>-b}$ is a strong solution of \eqref{eq:Rsde} and satisfies 
$\lim_{t\searrow -b} \mathrm{R}_t= \infty$ a.s..
\end{lemma}
}
\medskip

\begin{proof} {The proof is given in two steps: the first one exploits a time inversion trick on the Poisson process $(\eta_t)$, via the equilibrium version of this process, to reduce the analysis of $R^\ssup{b}$
to the analysis of  
$L^\ssup{a}$. Then the analysis of  $L^\ssup{a}$ is performed by exploiting
that $L^\ssup{a}$ (see see \eqref{eq:L}) is a ratio of two stochastic processes  defined on the same probability space and  adapted to the same 
 Brownian filtration. It is then practical to consider these two stochastic processes as one stochastic process taking values in $\bbR ^2$. 
We then show that, by It\^o formula, this two dimensional stochastic process solves a linear stochastic differential system: this system is treated  in detail in \cite[Sec.~2]{cf:CGG}, where it is shown that both components are strictly positive except at the initial time and that the ratios of the two components solve differential equations.
In fact, the two ratios  solve the same stochastic differential equation, except for the sign of the driving Brownian motion and for the initial condition.
 The initial condition of the ratio we consider is $\infty$, which makes a bit unpleasant the analysis, but this disappears if one considers the other ratio that has initial condition $0$ and directly falls into the standard theory of  solutions to stochastic differential equations. 
}

{The fact that solutions are strong is no surprise: as we just pointed out, the  stochastic differential equations we deal with are ultimately even coming from a linear system.  
But we stress it because it implies a number of useful properties: for example that, given $s\in (a,b)$, $(L_t^\ssup{a})_{t\in (a,s]}$ and 
$(R_t^\ssup{b})_{t\in [s,b)}$ are independent. }

\smallskip 

Let us therefore begin by arguing that it suffices to establish the result for $L^\ssup{a}_\cdot$.
As we have anticipated, for this it is practical to work with the equilibrium process $(\eta_t)_{t\in \bbR}$:
this is defined by stipulating that the wall locations are given by an homogeneous Poisson process on $\bbR$ with intensity
$\gep=\exp(-\gG)$ and the sign of the spins between walls is determined once the sign $\eta_0$ of the domain that contains the origin is chosen. The sign of  $\eta_0$ is a Bernoulli
of parameter $1/2$, independent of the Poisson process.  We call  $\bP_\Gamma ^\mathtt{eq}$ the law of this process and we remark that  
\begin{equation} \label{eq:LaEq}
L^\ssup{a}_t\, =\,  \frac
{ \bE^\mathtt{eq}_\Gamma \left[ \exp\left( -2\int_a^t \eta_s \dd B_s \right);\, \eta_a=0 \,, \eta_t =0   \right]}
{\bE^\mathtt{eq}_\Gamma \left[ \exp\left( -2\int_a^t \eta_s \dd B_s \right);\, \eta_a=0,\, \eta_t =1   \right]}
\, .
\end{equation}
In fact, since $\bP^\mathtt{eq}_\gG(\eta_a=0)=\bP^\mathtt{eq}_\gG(\eta_a=1)=1/2$, the process $(\eta_s)_{a\le s \le t}$ under $\bP^\mathtt{eq}_\Gamma(\cdot \,|\, \eta_a=0)$ has the same law as under $\bP_\Gamma$,  so \eqref{eq:LaEq} holds. Analogously 
\begin{equation}
R^\ssup{b}_t\, =\, 
\frac{
 \bE^\mathtt{eq}_\Gamma \left[ \exp\left( -2\int_t^b \eta_s \dd B_s \right);\, \eta_b =0\, , \eta_t=0   \right]
}
{
 \bE^\mathtt{eq}_\Gamma \left[ \exp\left( -2\int_t^b \eta_s \dd B_s \right);\, \eta_b =0\, ,\eta_t=1   \right]
}\,.
\end{equation}
{Now we remark also that $(\eta_{-s})_{s \in \bbR}$ and  $(\eta_{s})_{s \in \bbR}$ have the same finite dimensional laws, and 
if we set $\overline{ \eta}_{-s}:= \lim_{s'\searrow -s} \eta_{-s'}$ then $(\overline{\eta}_{-s})_{s \in \bbR}$ and  $(\eta_{s})_{s \in \bbR}$ have the same law,} so
\begin{equation}
\begin{split}
R^\ssup{b}_t\, &=\, 
\frac{
 \bE^\mathtt{eq}_\Gamma \left[ \exp\left( -2\int_t^b \eta_{-s} \dd B_s \right);\, \eta_{-b} =0\, , \eta_{- t}=0   \right]
}{
 \bE^\mathtt{eq}_\Gamma \left[ \exp\left( -2\int_t^b \eta_{b -s} \dd B_s \right);\, \eta_{-b} =0\, ,\eta_{-t}=1   \right]
}
\\
&=\, 
\frac{
 \bE^\mathtt{eq}_\Gamma \left[ \exp\left( 2\int_{-b}^{-t} \eta_{s} \dd \Brev_{s} \right);\, \eta_{-b} =0\, , \eta_{- t}=0   \right]
}{
 \bE^\mathtt{eq}_\Gamma \left[ \exp\left( 2\int_{-b}^{-t} \eta_{s} \dd \Brev_{s} \right);\, \eta_{-b} =0\, ,\eta_{-t}=1  \right]
}\, ,
\end{split}
\end{equation}
and let us remark that the symmetry is more evident in the case $a=-b$.
These expressions show  that it suffices to prove the statement for $L^\ssup{a}_\cdot$ and this is what we do next.

\smallskip

For the rest of the proof we can use the more compact expression \eqref{eq:L} for $L^\ssup{a}_t$
 and let us note that $\lim_{t \searrow a} L^\ssup{a}_t= \infty$ a.s. is immediate from 
 \eqref{eq:L}  because ${\bf s}_t= (-1)^{N_{t-a}}$, hence 
 $\bP_\gG(\eta_a=0)=0$. We set for $j=1$ and $2$
 \begin{equation} 
 \label{eq:X12}
 X_j(t)\,:=\,e^{\gep t}\, 
 { \bE_\Gamma \left[ \exp\left( -2\int_{a}^t \eta_s \dd B_s \right);\, \eta_t =j-1   \right]}
 \, .
\end{equation}
The process $(X_1(t),X_2(t))_{t \ge a}$ is clearly adapted to the filtration $(\sigma(B_s-B_a, a\le s \le t))_{ t\ge a}$ and we claim that it solves the linear SDE system
\begin{equation}
\label{eq:sys12}
\dd X_1(t)\, =\, \gep X_2 (t) \dd t \ \ \text{ and } \ \ 
\dd X_2(t)\, =\, \left( \gep X_1(t) +2 X_2(t) \right) \dd t - 2 X_2(t) \dd B_t\, ,
\end{equation}
for $t \ge a$ with $(X_1(a),X_2(a))=(1,0)$.
Since $L_t= X_1(t)/ X_2(t)$, the statement follows by applying It\^o formula (note that we have partly kept the same notations as in  \cite{cf:CGG}: in our case  $\gs=2$ and $L_t=1/Y_t$, $\gs$ and $Y_\cdot$ defined in  \cite{cf:CGG}).  In order to establish \eqref{eq:sys12} we remark that $\eta_t=\ind_{\eta_t=1}$,  $1-\eta_t=\ind_{\eta_t=0}$ and we 
set
\begin{equation}
\mathbf{X}_j(t)\,:=\,e^{\gep t}\, 
 \exp\left( -2\int_{a}^t \eta_s \dd B_s \right)\ind_{\eta_t=j-1}
 \, .
\end{equation}
By rewriting in integral formulation we have for $t>a$ and with $ \mathbf{X}(t):=  \mathbf{X}_1(t)+ \mathbf{X}_2(t)$
\begin{equation}
\begin{split}
\mathbf{X}_1(t) \, &=\, \mathbf{X}_1(a)+ \gep \int_a^t \mathbf{X}_1(s) \dd s 
-2 \int_a^t \mathbf{X}_1 (s) \eta_s \dd B_s +2  \int_a^t \mathbf{X}_1 (s) \eta_s^2 \dd s
- \int_a^t \mathbf{X}(s\text{\footnotesize{-}\!}) \dd \eta_s
\\
&=\, \mathbf{X}_1(a)+ \gep \int_a^t \mathbf{X}_1(s) \dd s - \int_a^t \mathbf{X}(s\text{\footnotesize{-}\!}) \dd \eta_s
\\
&=\, \mathbf{X}_1(a)+ \gep \int_a^t \mathbf{X}_1(s) \dd s - \int_a^t \mathbf{X}(s\text{\footnotesize{-}\!}) \left(\gep (1-2\eta_s)\dd s + \dd M_{\gep,s} \right)\, ,
\end{split}
\end{equation} 
where $\mathbf{X}(s\text{\footnotesize{-}\!})= \lim _{t \nearrow s}\mathbf{X}(t)$. Moreover   from the first to the second line the Brownian  integral term and the It\^o term disappear because the integrands contain $\eta_s (1-\eta_s)=0$ and in the next step we have used 
$\eta_t= \gep \int_a^t (1-2\eta_s)\dd s + M_{\gep,t}$ where $(M_{\gep,t})$ is a martingale (in fact $M_{\gep,t}= \int_a^t (1-2\eta_{s\text{\footnotesize{-}\!}}) \dd N_{\gep,s}$, where $N_{\gep,\cdot}$ is the martingale associated to
the Poisson process $N_\cdot$, i.e. $N_t= \gep t+ N_{\gep,t}$: in particular, $M_t \in \bbL ^p$ for every $p$).
Since $  \mathbf{X}(s)(1-2\eta_s) =  \mathbf{X}_1(s)-  \mathbf{X}_2(s)$, by taking the expectation with respect to the Poisson process (recall that $X_j(t)= \bE_\gG [\mathbf{X}_j(t)]$) we readily see that 
\begin{equation}
X_1(t) \, =\, X_1(a)+ \gep \int_a^t X_2 (s) \dd s \, ,
\end{equation} 
which is the first equation in \eqref{eq:sys12}. 
For the second one we write 
\begin{equation}
\mathbf{X}_2(t) \, =\, \mathbf{X}_2(a)+ \gep \int_a^t \mathbf{X}_2(s) \dd s 
-2 \int_a^t \mathbf{X}_2 (s) \dd B_s +2  \int_a^t \mathbf{X}_2 (s) \dd s
+ \int_a^t \mathbf{X}(s\text{\footnotesize{-}\!}) \dd \eta_s\, ,
\end{equation} 
where $\eta_s$ disappears in  Brownian integral and in the It\^o term because
 $\eta_s^2=\eta_s$. By using again $  \mathbf{X}(s)(1-2\eta_s) =  \mathbf{X}_1(s)-  \mathbf{X}_2(s)$ and by taking the expectation we recover the second 
 equation in \eqref{eq:sys12}. This completes the proof of Lemma~\ref{th:SDEs0}.
 \end{proof}

 \smallskip

\begin{rem}
\label{rem:bc0}
Lemma~\ref{th:SDEs0} is easily generalized 
to arbitrary fixed boundary conditions: 
if instead of boundary conditions $+1$ on the left we have $-1$ then 
$L_t^\ssup{a}$ solves the very same SDE, but $L_a^\ssup{a}=0$ (no need to taking the limit in this case).
Exactly the same statement holds if we change the boundary on the right.
And  it is useless to repeat the arguments we have just detailed in the case of plus boundary conditions. In fact 
it suffices to remark that if we use that 
$(\eta_t)_{t \in \bbR}$ has, under $\bP^\mathtt{eq}_\Gamma$, the same law as 
$(1-\eta_t)_{t \in \bbR}$, so
 \eqref{eq:LaEq} becomes 
 \begin{equation} \label{eq:LaEq-2}
L^\ssup{a}_t\, =\,  \frac
{ \bE^\mathtt{eq}_\Gamma \left[ \exp\left( 2\int_a^t \eta_s \dd B_s \right);\, \eta_a=1 \,, \eta_t =0   \right]}
{\bE^\mathtt{eq}_\Gamma \left[ \exp\left( 2\int_a^t \eta_s \dd B_s \right);\, \eta_a=1,\, \eta_t =1   \right]}
\, ,
\end{equation}
so $1/L^\ssup{a}_t$ becomes the expression we need with $-1$ left boundary condition, except that $B_\cdot$ is replaced
by $-B_\cdot$. But, as already pointed out in the introduction,  It\^o formula  yields that
 $1/L^\ssup{a}_t$ solves the same SDE as  $L^\ssup{a}_t$ with $B_\cdot$ replaced
by $-B_\cdot$, and with the obvious change in the initial condition. Therefore we see that the change in the boundary condition for the spin system just leads to changing the boundary condition for the SDE.
\end{rem}

\smallskip

For what follows we step to work with $l^\ssup{a}_t=  \log L^\ssup{a}_t$ and $r^\ssup{b}_t =  \log R^\ssup{b}_t$, introduced in \eqref{eq:l}. 
 
\medskip
 
 \begin{rem}
 \label{rem:Lyap12}
 Let us set $a=0$ for conciseness.
 One can integrate the first equation in \eqref{eq:sys12} to obtain 
 \begin{equation}
 \label{eq:Lyap12}
 \log X_1(t)= \gep \int_0^t \frac1{L^\ssup{0}_s}\dd s=   \gep \int_0^t \exp\left(-l^\ssup{0}_s\right)\dd s\, ,
 \end{equation}
 which relates the (almost sure) exponential rate of growth of $X_1(\cdot)$ (that is the free energy density up to the additive constant $\gep$, see \eqref{eq:X12}) with the  ergodic properties of $l^\ssup{0}_\cdot$. An analogous argument  can be applied to $X_2(t)$ and this leads to a formula that differs from  \eqref{eq:Lyap12}: nonetheless, one can show that $X_1(\cdot)$ and $X_2(\cdot)$ have the same rate of growth 
 (see \cite[Th.~1.1]{cf:CGG}).
 \end{rem}
 
 \medskip
 
 Without loss of generality, we just consider  $l^\ssup{a}_t$: one readily checks that 
 \begin{equation}
 \label{eq:la}
 \dd l^\ssup{a}_t \, =\,  2 \dd B_t - \gep \left( e^{l^\ssup{a}_t} - e^{-l^\ssup{a}_t} \right) \dd t \, ,
 \end{equation}
 for every $t \ge a$. It is the unique strong solution of this SDE, but since $ l^\ssup{a}_a= \infty$ we need to make this statement precise. 
 We are therefore going to look at the problem with more general initial condition and establish some monotonicity properties. 
 This will allow us also to deal with $a \to -\infty$ and with the construction of a (unique) solution on the whole of $\bbR$. Some of the technical results that we state and prove here will be of use also outside of this section. 
 
 \smallskip
 
 We start with a comparison lemma.
 
 \medskip

\begin{lemma}\label{l:comparison} Let $(\omega_t)_{t\ge 0}$ be a $({\cF}_t)_{t\ge0}$-Brownian motion and $ f^*:   \r \to \r$ be a locally Lipschitz function. Assume $x_t, x^*_t,  a_t, f_t$ to be $({\cF}_t)_{t\ge0}$-adapted real-valued continuous processes such that almost surely, for every $t\ge 0$,
\begin{equation}  
\begin{split}
x_t\, &=\,  x_0 + 2 \omega_t + \int_0^t f_s \dd s\, , \\
x^*_t\,&=\,  x^*_0 + 2 \omega_t + \int_0^t \left(a_s+f^*(x^*_s)\right) \dd s\, ,
\end{split}
\end{equation}
with
\begin{equation}
x_0\, \le\,  x^*_0 \ \ \ \text{ and } \ \ \ 
f_t \,\le \, a_t+ f^*(x_t)\, .
  \end{equation} 
  Then  a.s. $ x_t \le x^*_t$  for every $  t\ge 0$.   

Similarly if $x_0 \ge  x^*_0$ and $f_t \ge   a_t+ f^*(x_t)$, then a.s. we  have $ x_t \ge x^*_t$ for every $t\ge 0$.   
  \end{lemma}

\medskip

\smallskip

We put the constant $2$ before $\omega_t$ only for notation reasons in the applications of this lemma. When $a_s$ (resp. $f_s$) is a deterministic function of $s$ (resp. of $(s, x_s)$), this lemma is a particular case of the general comparison theorems for one-dimensional diffusions, see \cite{IW77}. 

\medskip

\begin{proof}
 By considering the processes $x_t$ and $x^*_t$ stopped at $\inf\{t\ge 0: |x_t|+|x^*_t|\ge n\}$ for  $n=1,2, \ldots$, we can assume that they are bounded. Then there exists $K>0$ such that for every $s\ge 0$ we have  $| f^*(x_s) - f^*(x^*_s)| \le K | x_s - x^*_s|$.
Since $x_t-x^*_t= x_0- x^*_0 + \int_0^t (f_s - a_s-f^*(x^*_s)) \dd s$ we see that $t \mapsto x_t-x^*_t$ is $C^1$. {Then, considered as a $({\mathcal F}_t)_{t\ge0}$-continuous semimartingale  the process $x_t-x^*_t$ has   vanishing local times.  Applying the Tanaka formula (\cite{RY}, Theorem VI.1.2), } we have 
\begin{equation}
\begin{split} 
  \max(0, x_t-x^*_t)  \,
&=\,  \max(0, x_0- x^*_0)   +    \int_0^t   \ind_{\{x_s > x^*_s\}}(f_s- a_s-f^*(x^*_s))  \dd s  
\\
&\le \,
 \int_0^t  \ind_{\{x_s > x^*_s\}} (f^*(x_s) - f^*(x^*_s))   \dd s 
 \\
 &\le \, 
 K  \int_0^t   \max(0, x_s  -  x^*_s)    \dd s\, . 
 \end{split}
 \end{equation}
 We conclude by Gronwall's Lemma that $  \max(0, x_t-x^*_t)  =0$.
 \end{proof}

\medskip

Let us introduce the unique strong solution $l_t^\ssup{a, x}$ of \eqref{eq:la} with initial condition $x\in\R$, that is
the only continuous process, adapted to the filtration $\sigma(B_s-B_a, a\le s \le t), t\ge a$, that satisfies a.s. 
\begin{equation} \label{eds-la}
l_t^\ssup{a, x}\, =\,x + 2 (B_t-B_a) - \gep \int_a^t \left( e^{l_s^\ssup{a, x}} - e^{-l_s^\ssup{a, x}}\right) \dd s\, ,
\end{equation}
for every $t\ge a$. 

\smallskip

We collect in the next lemma useful upper and lower bounds on $l_t^\ssup{a, x}$ that follow by applying Lemma~\ref{l:comparison}.
\medskip

\begin{lemma}
\label{th:lbub-on-l}
For every $x \in \bbR$ and every $t>a$ we have $ \underline{l}^\ssup{a, x}_{t}\le l^\ssup{a, x}_t \le  \overline{l}^\ssup{a, x}_{t}$ with
\begin{equation} 
\label{eq:ell-ax-low}   
\underline{l}^\ssup{a, x}_{t} \,  := \, x + 2 (B_t-B_a)
    - \log \left( 1+ \varepsilon e^x   \int_a^t e^{2 (B_s- B_a)   } \dd  s\right)   \, ,
 \end{equation}
 and 
  \begin{equation}    \label{eq:ell-ax-upp} 
  \overline{l}^\ssup{a, x}_{t}\, :=\,  \theta^x_t  - \log \left(1+ \varepsilon \int_a^t e^{\theta^x_s} \dd s \right)\, , 
\end{equation}
with
\begin{equation}
\label{eq:thetata}
\theta^x_t\, :=\, 
 x +  2  (B_t-B_a) + \varepsilon e^{-x} \int_a^t e^{- 2 (B_s-B_a)} \dd s + \varepsilon^2 \int_a^t \int_a^s e^{2 (B_r-B_s)} \dd r \dd s\, .
\end{equation}
Moreover for every $t>a$ a.s. we have
\begin{equation} 
\label{eq:lblinfty}
\lim_{x \to \infty} \underline{l}^\ssup{a, x}_{t}
 \, =\,  
 -\log \varepsilon  - \log   \int_a^t e^{2 (B_s-B_t)   } \dd s \, ,
 \end{equation}
 and 
 \begin{equation} 
\label{eq:ublinfty}
\limsup_{x \to \infty} \overline{l}^\ssup{a, x}_{t}
 \, \le \,   
 -\log \varepsilon + \varepsilon^2 \int_a^t \int_a^s e^{2 (B_r-B_s)} \dd r \dd s   -  \log   \int_a^t e^{2 (B_s- B_t)} \dd s\, .
 \end{equation}
\end{lemma}

\medskip

{The idea behind the proof of this lemma is that we are dealing with a diffusion  between two barriers. In particular we can bound the solution from below 
in terms of the solution found when we suppress the left barrier (and this bound will not be bad as long as the solution without left barrier does not approach the region where 
the left barrier was located). And when there is only one barrier the solution is explicit. 
}

\medskip

\begin{proof} {For the lower bound, we argue that $\underline{l}^\ssup{a, x}_{s}$ was defined such as to satisfy the stochastic differential equation
\begin{equation}\label{eq:SDEunderlinel}
 \underline{l}^\ssup{a, x}_{t}\, =\, x + 2 (B_t-B_a)  -\int_a^t \gep e^{ \underline{l}^\ssup{a, x}_{s}} \dd s\, ,
\end{equation} 
It\^o formula indeed yields that this SDE is equivalent to the SDE
$ \dd \exp\left(-\underline{l}^\ssup{a, x}_{t}+  2 B_t\right)= \gep \exp(2 B_t) \dd t$, which can be integrated immediately.}
By applying Lemma~\ref{l:comparison} we  establish the lower bound. 
For \eqref{eq:lblinfty} it suffices to observe that
\begin{equation}  
\underline{l}^\ssup{a, x}_{t} \,  = \, 
    - \log \left( e^{-x - 2 (B_t-B_a)}+ \varepsilon    \int_a^t e^{2 (B_s- B_t)   } \dd  s\right)   \, ,
 \end{equation}
 and the passage to $x\to \infty$ is straightforward.

 \smallskip
 
 For the upper bound we first exploit the lower bound that we have just established to obtain that for $t>a$
 \begin{equation}
 l^\ssup{a, x}_{t} \,\le\,  x +  2  \left(B_{t}-B_a\right) + \varepsilon \int_a^t \left(e^{-\underline{l}^\ssup{a, x}_{u}} - e^{l^\ssup{a, x}_{u}}\right) \dd u\, ,
 \end{equation}
 hence, by Lemma~\ref{l:comparison}, we have that 
 $ l^\ssup{a, x}_t \le  \overline{l}^\ssup{a, x}_{t}$ if we define $ (\overline{l}^\ssup{a, x}_{t})_{t \ge a}$ as the unique solution of 
 \begin{equation}
 \label{eq:SDEwiththeta}
 \begin{split}
 \overline{l}^{a,x}_{t} \,&=\,  x +  2  \left(B_{t}-B_a\right) + \gep \int_a^t e^{-\underline{l}^\ssup{a, x}_{u}} \dd u  - \gep \int_a^t  e^{\overline{l}^\ssup{a, x}_{u}} \dd u
 \\
 &=\,  \theta^x_t - \gep \int_a^t  e^{\overline{l}^\ssup{a, x}_{u}} \dd u
 \, ,
 \end{split}
 \end{equation} 
 for every $t>a$: 
 in the second line we have used that
 \begin{equation}
\label{eq:thetata2}
\theta^x_t\, =\,  x +  2  \left(B_{t}-B_a\right) + \gep \int_a^t e^{-\underline{l}^\ssup{a, x}_{u}} \dd u \, ,
\end{equation}
which can be verified by plugging in the explicit expression for $\underline{l}^\ssup{a, x}_{u}$ given in \eqref{eq:ell-ax-low} ($\theta^x_t$ is defined in \eqref{eq:thetata}). 
 The proof of the upper bound is now completed by remarking that the stochastic process explicitly given in \eqref{eq:ell-ax-upp} solves
 \eqref{eq:SDEwiththeta}:
{in fact \eqref{eq:SDEwiththeta} is the same as \eqref{eq:SDEunderlinel} with $x+2(B_t-B_a)$ replaced by $\theta^x_t$ and can be solved in the same way.}
 
 We are left with proving \eqref{eq:ublinfty}.
 For this we remark that from \eqref{eq:ell-ax-upp} and \eqref{eq:thetata} we have 
 \begin{equation}    
  \overline{l}^\ssup{a, x}_{t}\, \le \,  \theta^x_t  - \log \left(1+ \varepsilon \int_a^t e^{x+2(B_s-B_a)} \dd s \right)\,=\,  \theta^x_t -x-  \log \left(e^{-x}+ \varepsilon \int_a^t e^{2(B_s-B_a)} \dd s \right)  . 
\end{equation}
From \eqref{eq:thetata} we readily see that
$\lim_{x \to \infty}(\theta^x_t-x)=2  (B_t-B_a) + \varepsilon^2 \int_a^t \int_a^s e^{2 (B_r-B_s)} \dd r \dd s$ a.s. and 
  \eqref{eq:ublinfty} follows by rearranging the terms.
\end{proof}

\medskip

We give here also another result in the same spirit as Lemma~\ref{l:comparison} that is useful in order to extend the bounds in Lemma~\ref{th:lbub-on-l} to random times that are not stopping times. In fact, the result is fully deterministic:

\medskip

\begin{lemma}
\label{th:forTrandom}
Given $T\in \bbR$ and a continuous function $t \mapsto b_t$ we consider a continuous function  $\ell_\cdot$ that solves
\begin{equation}
\ell_t\, =\, \ell_T+2(b_t-b_T) - \gep \int_T^t \left( e^{\ell_s} - e^{-\ell_s} \right) \dd s\, ,
\end{equation}
for every $t\ge T$, for a given value $\ell_T$. Then for every $t\ge T$,
\begin{eqnarray}
\label{eq:forTrandom}
\ell_t\, &\ge &\, \ell_T+2(b_t-b_T)- \log \left( 1+ \gep e^{\ell_T} \int_T^t e^{2(b_s-b_T)} \dd s\right)\,,
\\
\ell_t\, &\le &\, \ell_T+2(b_t-b_T) + \log \left( 1+ \gep e^{-\ell_T} \int_T^t e^{-2(b_s-b_T)} \dd s\right).  \label{eq:forTrandom_upp}
\end{eqnarray}
%
%
\end{lemma}

\medskip

\begin{proof} Note that we may apply \eqref{eq:forTrandom} to $-\ell$ and $-b$ instead of $\ell$ and $b$ and obtain \eqref{eq:forTrandom_upp}. Then it is enough to show \eqref{eq:forTrandom}.

Call $u_t$ the right-hand side in \eqref{eq:forTrandom} and set $x_t:= u_t-\ell_t$.
One can directly check by taking the time derivative of  $u_t- 2(b_t-b_T)$, i.e. the logarithmic part of the expression, 
 and by integrating that 
 \begin{equation}
u_t\, =\, \ell_T+2(b_t-b_T) - \gep \int_T^t e^{u_s}  \dd s\, . 
\end{equation}
 Therefore $t \mapsto x_t$ is $C^1$ and, since also $x\mapsto \max(0, x)^2$ is $C^1$, we obtain
 \begin{equation}
 \frac{\dd}{\dd t} \max\left(0, x_t\right)^2\, =\, 
 2  \max\left(0, x_t\right) \left(-\gep e^{\ell_t} \left(e^{x_t}-1\right) - \gep e^{-\ell_t} \right) \, \le \, 0\, .
 \end{equation}
 Since $x_T=0$, we get \eqref{eq:forTrandom} and complete the proof.
\end{proof}

\medskip


\medskip

\begin{proposition}
\label{th:Laxinfty}
For every $x<x'$ we have
$l_t^\ssup{a, x}< l_t^\ssup{a, x'}$ for every $t\ge a$, hence $l_t^\ssup{a, -\infty}:= \lim_{x \searrow - \infty} l_t^\ssup{a, x}$ 
and $l_t^\ssup{a, \infty}:= \lim_{x \nearrow + \infty} l_t^\ssup{a, x}$ are well defined for every $t$. 
Moreover, almost surely we have that $t \mapsto \vert  l_t^\ssup{a, \pm \infty}\vert$ are bounded
over any compact subset of $(a, \infty)$, $ \lim_{t \searrow a} l_t^\ssup{a, \pm\infty}= \pm\infty$ and for every $s,t$ with $t>s>a$,  we have
\begin{equation}
\label{eq:SDEinfty}
l_t^\ssup{a, \pm\infty} \, =\,l_s^\ssup{a, \pm\infty} + 2 (B_t-B_s) - \gep \int_s^t \left( e^{l_u^\ssup{a, \pm\infty}} - e^{-l_u^\ssup{a, \pm\infty}}\right) \dd u\, ,
\end{equation}
In addition,  $l_\cdot^\ssup{a, + \infty}$ coincides with $l^\ssup{a}_\cdot= \log L_\cdot^\ssup{a}$ given in \eqref{eq:L}. 
\end{proposition}

\medskip

\begin{proof} The first remark is that it suffices to argue for $x \nearrow \infty$ because $-l_t^\ssup{a, -x}$ solves the same SDE as $l_t^\ssup{a, x}$ except for the driving $B_\cdot$ that is replaced by $-B_\cdot$. 
The claimed monotonicity follows directly from Lemma~\ref{l:comparison} and the claimed local boundedness of 
 $l_t^\ssup{a, \pm\infty}$ follows directly from the explicit upper and lower bounds in  \eqref{eq:ublinfty} and \eqref{eq:lblinfty}. 
 And  \eqref{eq:lblinfty} implies also that $ \lim_{t \searrow a} l_t^\ssup{a,\infty}= \infty$.
 
 The validity of \eqref{eq:SDEinfty} follows because we know that  \eqref{eq:SDEinfty}  holds if $\pm \infty$ is replaced by $x$ and, for almost every $B_\cdot$, we can pass to $x \nearrow \infty$  by applying the Dominated Convergence Theorem for the integral term, because $l_u(a,x)\in [l_u(a,-\infty), l_u(a,\infty)]$ and because we have just proven that a.s. $\sup_{u \in [s, t]}\vert  l_u(a,\pm \infty)\vert < \infty$. 
 
Finally, in order to show that  $l_\cdot^\ssup{a, + \infty}=l^\ssup{a}_\cdot$ we remark that they solve the same SDE, i.e. \eqref{eq:SDEinfty}, and that
$\lim_{t \searrow a} l_t^\ssup{a, + \infty}=  l_t^\ssup{a} = \infty$. An application of It\^o formula shows that both 
$L_\cdot^\ssup{1}:=\exp(-l_\cdot^\ssup{a, + \infty})$ and $L_\cdot^\ssup{2}:=\exp(-l^\ssup{a}_\cdot)=1/L_\cdot^\ssup{a}$ solve the same SDE as the one solved by $L_\cdot^\ssup{a}$, that is \eqref{eq:Lsde}, but with 
 $B_\cdot$ replaced by $-B_\cdot$, that is for $j\in \{1,2\}$
  \begin{equation}  
  \dd L^\ssup{j}_t\, =\, -2   L^\ssup{j}_t \dd B_t + \left(\gep (1-(L^\ssup{j}_t)^2)+  2L^\ssup{j}_t \right) \dd t\, ,
    \end{equation}
 for every $t>a$ and this time  $L_a^\ssup{1}=L_a^\ssup{2}=0$.  
 {Since there is a unique  solution to such a SDE, the proof of Proposition~\ref{th:Laxinfty} is complete}.
 \end{proof}

\medskip

Let $\cF_t:=\sigma(B_u-B_s, -\infty<s<u\le t)$ for $t\in \r$. 
\begin{proposition}
\label{th:lonR}
There exists a unique continuous random process $(l_t)_{t \in \bbR}$ adapted to $\{ \cF_t\}_{t \in \bbR}$ that solves 
\begin{equation}
\label{eq:lonR}
l_t\, =\,l_s + 2 (B_t-B_s) - \gep \int_s^t \left( e^{l_u} - e^{-l_u}\right) \dd u\, ,
\end{equation}
for every $s<t$. Moreover, a.s. we have that $l_t=\lim_{a \to -\infty} l^\ssup{a,x}_t$ for every $x \in [- \infty, \infty]$ and every $t \in \bbR$. The process  $(l_t)_{t \in \bbR}$ is stationary and the law of $l_t$ is $p_\gG$ (given in \eqref{eq:pGa}).
\end{proposition}
 
 \medskip

  \begin{rem}
  \label{rem:bc}
 For conciseness, we have chosen to focus on plus boundary conditions. But, as pointed out in  
  Remark~\ref{rem:bc0}, the change in the boundary conditions in the spin system simply lead to changing 
  the initial condition  $L^\ssup{a}_a$ from $\infty$ to $0$ (and the same for $R^\ssup{b}_b$). This simply amounts to changing 
  $l^\ssup{a,\infty}_t$ to $l^\ssup{a,-\infty}_t$ and Proposition~\ref{th:lonR} guarantees that this is irrelevant in the $a \to - \infty$ limit. 
  \end{rem}
  
\medskip  
    
In the proof of Proposition \ref{th:lonR} we are going to use the following result. Recall that $U_\gG(x):= \gep ( \exp(x)+ \exp(-x))$.  
 \medskip
 
 \begin{lemma}
 \label{th:OCbound}
 If  $(l^\ssup{1}_t)_{t\in (a, \infty)}$ 
 and $(l^\ssup{2}_t)_{t\in (a, \infty)}$, with $a\in [-\infty, \infty)$, are two continuous adapted processes solving on $(a, \infty)$ the SDE
 \begin{equation}
 \label{eq:lgGSDE}
\dd l_t \, =\, -\gep \left( e^{l_t} - e^{-l_t}\right) \dd t + 2 \dd B_t\, ,
\end{equation}
and if there exists $t_0>a$ such that $l^\ssup{2}_{t_0}>l^\ssup{1}_{t_0}$, then $l^\ssup{2}_t>l^\ssup{1}_t$ for every $t>a$ and 
\begin{equation}
\label{eq:OCbound}
\tanh \left( \frac{l^\ssup{2}_t-l^\ssup{1}_t}4 \right)\, \le \, e^{-2 \gep (t-s)}\tanh \left(\frac{ l^\ssup{2}_s-l^\ssup{1}_s}4 \right) \, ,
\end{equation}
for every $t>s>a$. 
 \end{lemma}
 \medskip
 
 Note that, since $\vert \tanh(\cdot) \vert \le 1$, \eqref{eq:OCbound} implies that for every $t >a$,  
 \begin{equation}
\label{eq:OCbound2}
 {l^\ssup{2}_t-l^\ssup{1}_t} \, \le \, 4\, \arcth\left(  e^{-2 \gep (t-a)} \right) \, .
\end{equation}
 
 \medskip
 
 \begin{proof} Note that solutions of \eqref{eq:lgGSDE} form a stochastic flow of $C^\infty$-diffeomorphisms of $\r$ (see \cite{Kunita82}, Theorem II.6.1). In particular solutions {do not meet}, so $l^\ssup{2}_t>l^\ssup{1}_t$ for all $t>a$.
 Then we set $\gD_t:= l^\ssup{2}_t-l^\ssup{1}_t>0$, so
\begin{equation}
\dd \gD_t \, =\, -2 \gep \left( \sinh \left(  l^\ssup{2}_t \right)-  \sinh \left( l^\ssup{1}_t \right)\right) \dd t\, =: f_\gep (t) \dd t\, ,
\end{equation}
and $\gD_\cdot$ is $C^1$. 
Now we remark that, for $x >0$, the minimum of the convex function  $u \mapsto  \sinh(u+x)- \sinh(u)$ is reached at
$u=-x/2$, hence it is equal to $2 \sinh(x/2)$ and 
$f_\gep(t) \le -4\gep \sinh(\gD_t/2)$. So, for every $s >a$, we introduce the solution $D_\cdot$ to the 
ODE $\dd D_t / \dd t=  -4\gep \sinh(D_t/2)$ with initial condition $D_s=\gD_s>0$. Explicitly
\begin{equation}
D_t\, =\, 4\, \mathrm{arctanh} \left( \tanh\left(\frac{\gD_s}4\right) e^{- 2 \gep (t-s)} \right)\, ,
\end{equation}
for every $t\ge s$. Note that $D_t\in (0, D_s)$ for every $t>s$.
By proceeding as in the proof of Lemma~\ref{l:comparison},  we consider 
$\max(0,\gD_t - D_t) $ for $t\ge s$:
\begin{equation}
\begin{split}
\max(0,\gD_t - D_t) \, &=\, \int_s ^t \ind_{\{\gD_u> D_u\}}  \left( f_\gep(u) + 4 \gep \sinh(D_u /2) \right) \dd u
\\ 
&\le \, -4 \gep \int_s ^t \ind_{\{\gD_u> D_u\}}  \left(   \sinh(\gD_u /2)- \sinh(D_u /2) \right) \dd u\, \le \, 0\, ,
\end{split}
\end{equation} 
so $0<\gD_t \le D_t$ which implies that for every $t>s$ 
\begin{equation}
\tanh( \gD_t /4) \, \le \, \tanh( D_t/4)\, =\, e^{- 2 \gep (t-s)}\tanh( \gD_s/4)   \,,
\end{equation} 
and the proof of Lemma~\ref{th:OCbound} is complete.
 \end{proof}

 \begin{proof}[Proof of Proposition~\ref{th:lonR}]
 Another application of Lemma~\ref{l:comparison} yields that, if $a'<a$, then $l_t^\ssup{a', + \infty} \le l_t^\ssup{a, + \infty}$ and $l_t^\ssup{a', - \infty} \ge l_t^\ssup{a, - \infty}$ for every $t>a$. 
 Therefore we can define both $\overline{l}_t:= \lim_{a \to - \infty} l_t^\ssup{a, + \infty}$ and $\underline{l}_t:= \lim_{a \to - \infty} l_t^\ssup{a, - \infty}$ and note that they are a.s. locally bounded by construction. Also by construction we have that both $\overline {l}_\cdot$ and $\underline {l}_\cdot$ are adapted to the filtration $(\cF_t)_{t\in \bbR}$. By the 
 very same argument used to show \eqref{eq:SDEinfty} we obtain that both $\overline{l}_\cdot$ and $\underline{l}_\cdot$ solve the SDE \eqref{eq:lgGSDE} for every $t\in \bbR$.
 
We are now going use Lemma~\ref{th:OCbound} (with $a= -\infty$) to show that there is only one adapted continuous solution on the whole of $\bbR$ to such an equation.  Let us suppose that there are two solutions: let us call them $l^\ssup{1}_\cdot$ and $l^\ssup{2}_\cdot$ and assume that $l^\ssup{1}_t<l^\ssup{2}_t$ for some $t\in\bbR$. 
Lemma~\ref{th:OCbound} yields that for every $s<t$ 
\begin{equation}
\tanh\left(\frac{l^\ssup{2}_t- l^\ssup{1}_t}4 \right) \, \le \, e^{- 2 \gep (t-s)} \tanh\left(\frac{l^\ssup{2}_s- l^\ssup{1}_s}4\right)   \, \le \, e^{- 2 \gep (t-s)}\overset{s \to -\infty} \longrightarrow 0 \,,
\end{equation} 
which is incompatible with $l^\ssup{2}_t- l^\ssup{1}_t>0$. Therefore $l^\ssup{2}_t= l^\ssup{1}_t$ and the uniqueness statement is proven.

The fact that the law of $l_t$ is $p_\gG$ follows by observing that $(l_t)_{t \in \bbR}$ is stationary: this is a consequence 
of the fact that $(l_t)_{t \in \bbR}$ is a strong solution of \eqref{eq:lonR} and therefore $l_t$ is a measurable function of
$B_\cdot$ and $l_{t+s}$ is the same function of the Brownian motion $B^\ssup{s}_\cdot$, hence $l_t $ and $l_{t+s}$ have the same law.  
By the ergodic properties of the SDE \eqref{eq:lonR} we know that $(l_{t+s})_{s>0}$ converges in law to
$p_\gG$ for $s\to \infty$ and the proof is complete.
 \end{proof}

\section{Proof of main technical results: Lemma~\ref{th:modelsclose} and Proposition~\ref{th:formain}}
\label{sec:maintech}

\subsection{Proof of Lemma~\ref{th:modelsclose} }
\label{sec:modelsclose}

\begin{proof}[Proof of Lemma~\ref{th:modelsclose}]
We recall that the stationary process $(l_t)_{t \in \bbR}$ is defined via Proposition~\ref{th:lonR} and that the law of $l_t$ is $p_\gG$.
Recall also from \eqref{eq:reversed-rv} the definition  of $\tsup(\Gamma)$ and $\tsdown(\Gamma)$ which are stopping times for the reversed Brownian filtration. Recall from \eqref{vdownGamma} that $\tvdown(\gG)$ is the unique time in $ [\tsdown(\Gamma),0]$ such that 
\begin{equation}B_{\tvdown(\gG)}\, =\, \sup_{t\in [\tsdown(\Gamma),0]} B_t\, .
\end{equation}
 
Pick $\gk>0$. In strict analogy with the proof of Lemma $\ref{th:invtimesimp}$, our program is resumed in three steps: on the event $\{\tsdown(\Gamma)>\tsup(\Gamma)\}$ and excluding an event of probability 
   $O(\Gamma^{-\gk})$   \smallskip 
  \begin{itemize} [leftmargin=20pt]
  \item  \emph{Step 1}: establishing that 
  \begin{equation}
  -\Gamma-\loglog\gG-C_{\gk,1} \,\le\, l_{\tsdown(\Gamma)} \,\le\, \Gamma+\loglog\gG+C_{\gk,1}\, ,
  \end{equation}
  \item \emph{Step 2}: exploiting that there is a rise in $B_\cdot$ of size $\Gamma$ between times $\tsdown(\Gamma)$ and $\tvdown(\gG)$ to show that
  \begin{equation} \gG-\loglog\gG - \underline{C_{\gk,2}} \,\le\, l_{\tvdown(\gG)}\, \le\,\gG+\loglog\gG + \overline{C_{\gk,2}} \, , 
  \end{equation}
  \item \emph{Step 3}: exploiting that for $t \in [\tvdown(\gG),0]$ the process $l_\cdot $ evolves approximately according to $2B_\cdot$ to show that
  \begin{equation}
  \gG-2B_{\tvdown(\gG)}-\loglog\gG-\underline{C_{\gk,3}}\,\le\, l_0\,\le\, \gG-2B_{\tvdown(\gG)}+\loglog\gG + \overline{C_{\gk,3}}\, ,
  \end{equation} 
   \end{itemize}
with constants $C_{\gk,1}, \underline{C_{\gk,2}}, \overline{C_{\gk,2}}, \underline{C_{\gk,3}}$ and $\overline{C_{\gk,3}}$ depending on $\gk$. Setting $C_\gk=\max(\underline{C_{\gk,3}}, \overline{C_{\gk,3}})$, we will have shown that $|l_0-\ls_0|\le \loglog\gG + C_\gk$ on the event $\{\tsdown(\Gamma)>\tsup(\Gamma)\}$ and excluding an event of probability $O(\Gamma^{-\gk})$. Replacing $B$ by $-B$, the same will hold on event $\{\tsup(\Gamma)>\tsdown(\Gamma)\}$, so Lemma \ref{th:modelsclose} will follow.

\medskip

Actually, in the proofs of \emph{Step 1}, of \emph{Step 2} and of the lower bound in \emph{Step 3}
we will work without assuming that $\tsdown(\Gamma)>\tsup(\Gamma)$: the statements  remain of course true on this event, which has probability $1/2$.

\smallskip 
	
\subsubsection{Proof of Step 1.}
  Since $\ttdown(\Gamma)$ is a stopping time, $(B_{\ttdown(\Gamma)+t}-B_{\ttdown(\Gamma)})_{t\ge0}$ is a standard Brownian motion, independent of $(B_t)_{0\le t\le \ttdown(\Gamma)}$. The same argument holds replacing $B$ with $\Brev$, so $(B_{\tsdown(\Gamma)+t}-B_{\tsdown(\Gamma)})_{t\le0}$ is independent of $(B_t)_{\tsdown(\Gamma)\le t\le 0}$ and has the same law as $(B_t)_{t\le0}$. From this, we deduce that $l_{\tsdown(\Gamma)}$  has law $p_\Gamma$ (and is independent of $(B_t)_{\tsdown(\Gamma)\le t\le 0}$)
hence
\begin{equation}
\begin{split}
\bbP\left( \vert l_{\tsdown(\Gamma)}\vert > \gG +\loglog\gG +C_{\gk, 1}\right) \, 
& =\, 2 \int_{\gG +\loglog\gG +C_{\gk, 1}}^\infty p_\gG (x) \dd x \, 
\\
& \le \, \frac {1} {K_0(\gep)} \int_{\gG +\loglog\gG +C_{\gk, 1}}^\infty e^{- \frac{\gep} 2 e^x} \dd x 
\\
&  =\, 
\frac {1} {K_0(\gep)} 
\int_{\frac{e^{C_{\gk, 1}}}{2}\log\gG}^\infty \frac{e^{-y}}y
 \dd y\, ,
\end{split}
\end{equation}
 and, since $K_0(\gep)\sim - \log \gep = \gG$ we have, as soon as $1+{e^{C_{\gk, 1}}}/2\ge\gk$, that $\vert l_{\tsdown(\Gamma)}\vert \le \gG +\loglog\gG +C_{\gk, 1}$ with probability $1-O(\Gamma^{-\gk})$.
  
  \smallskip

  \subsubsection{Proof of Step 2:  lower bound.}
  
  By applying Lemma~\ref{th:forTrandom}, \eqref{eq:forTrandom} with $T=\tsdown(\Gamma)$ and $t=\tvdown(\Gamma)$, we have that
  \begin{equation}\label{eq:llow} 
  \begin{aligned}
   l_{\tvdown(\gG)} & \, \ge \,  l_{\tsdown(\gG)}+2(B_{\tvdown(\gG)}-B_{\tsdown(\Gamma)})-\log\left(1+\gep e^{l_{\tsdown(\gG)}} \int_{\tsdown(\Gamma)}^{\tvdown(\gG)} e^{2\left(B_s-B_{\tsdown(\gG)}\right)} \dd s\right)\\   
   & \, = \, -\log\left(e^{-l_{\tsdown(\Gamma)}-2\Gamma}+\gep  \int_{\tsdown(\Gamma)}^{\tvdown(\gG)} e^{2\left(B_s-B_{\tvdown(\gG)}\right)} \dd s\right)\, ,
   \end{aligned}
  \end{equation}
  where we also used that $B_{\tvdown(\Gamma)}-B_{\tsdown(\Gamma)}=\Gamma$.
Note that $\int_{\tsdown(\gG)}^{\tvdown(\Gamma)} e^{2 (B_s-B_{\tvdown(\gG)} )} \dd s$ is distributed as $\int_{\tudown(\gG)}^{\ttdown(\gG)} e^{-2(B_{\tudown(\gG)}-B_t)} \dd t $ which is less than  $\int_0^{\ttdown(\gG)}  e^{-2(B_{\tudown(\gG)}-B_t)} \dd t$. It follows from lemma \ref{th:excthlemma} $(iii)$ and an application of Markov inequality that there exists a positive constant $C_{\kappa,4}$ such that with probability $1-O(\Gamma^{-\gk})$
\begin{equation}
\int_{\tsdown(\Gamma)}^{\tvdown(\gG)} e^{2\left(B_s-B_{\tvdown(\gG)}\right)} \dd s\,\le\,  C_{\gk, 4}\log \Gamma\, .
\end{equation}
Therefore,
using $l_{\tsdown(\Gamma)} \ge -\Gamma-\loglog\gG-C_{\gk, 1}$ (i.e., \emph{Step 1}), we conclude that with probability $1- O(\Gamma^{-\gk})$
\begin{equation}
l_{\tvdown(\gG)}\, \ge\,  \Gamma - \loglog \Gamma-\underline{C_{\gk,2}}\, ,
\end{equation} 
where $\underline{C_{\gk,2}}=\log\left(e^{C_{\gk,1}}+C_{\gk, 4}\right)$.
  
\subsubsection{Proof of Step 2:  upper bound.}
	The  idea is to use a bound on the modulus of continuity of $B_\cdot$: we first give a sketch of the argument.  Choose $F>0, a>0$ and $\lambda>0$ such that $2\lambda\le a e^F$. Assume that $B_\cdot$ cannot make a rise larger than $\lambda$ inside any interval of length $a$. When $l_\cdot$ is above $\Gamma+F$, then the repulsion, which acts essentially like $-e^F \dd t$, makes $l$ smaller by at least $a e^F$ in an interval of length $a$, while $B_\cdot$ can make it  larger of   at most $2\lambda$. Thus $l_\cdot$ cannot stay above $\Gamma+F$ over an interval of length longer than $a$, so it cannot grow beyond $\Gamma+F+2\lambda$. 
	
	\smallskip
	
	In order to fill in the details, we start with:
	
	\smallskip
	
	\begin{lemma}
	\label{lemuppboundstep2}
	Let $t>0$ and $b:[0,t]\to \mathbb{R}$ be a continuous function. Consider $\ell :[0,t]\to \mathbb{R}$ solution of the ODE 
	\begin{equation}
	  \ell_s\,=\, \ell_0+2(b_s-b_0)-2\varepsilon \int_0^s \sinh(\ell_u) \dd u\, ,
	  \end{equation}
	for every $s\in [0, t]$ (and for a given initial condition $\ell_0$).
	Let $F>0, a>0$ and $\lambda>0$. Assume moreover that
	 \begin{enumerate}
	 \item     $ |b_s-b_{s'}| \le \lambda$
	 if  $s,s'\in[0,t]$ and $|s-s'|\le a$;
	 \item $ \lambda\le\gep a \sinh(\Gamma+F)$.
	 \end{enumerate}
	 Then, provided that the initial condition satisfies
	 $\ell_0\le \Gamma+F$, 
	 we have
	 \begin{equation}
	 \label{uppboundls}
	 \sup_{s\in[0, t]}\ell_s\, \le\,  \Gamma+F+2\lambda\, .
	 \end{equation}
	 \end{lemma}
\smallskip

Lemma~\ref{lemuppboundstep2} applied to our context yields (together with \emph{Step 1}):
\smallskip 

	\begin{cor}
	\label{cor:uppboundstep2}
	For all $\gk>0$, there exists a constant $\overline{C_{\gk,_2}}$ such that with probability $1-O(\Gamma^{-\gk})$,
	we have that 
	\begin{equation}  
	\max_{s\in[\tsdown(\Gamma),0]}l_s\, \le\,  \Gamma+\loglog \Gamma+\overline{C_{\gk,_2}}.
	\end{equation} 
	\end{cor}
	
	Corollary \ref{cor:uppboundstep2} yields in particular the upper bound in \emph{Step 2}.	
	
\medskip

\begin{proof}[Proof of Lemma~\ref{lemuppboundstep2}]
Consider an arbitrary  $s\in[0,t]$. If $l_s\le \Gamma+F$ there is nothing to prove, so let us  assume that $l_s>\Gamma+F$. Define
\begin{equation}
\varrho\, :=\, \sup\left\{u\in[0,s]\colon
\, 
\ell_u\le \Gamma+F\right\}\, .
\end{equation}
Then, on the interval $[\varrho,s]$, $\ell_\cdot$ lies above $\Gamma+F$. Using this and assumption $(1)$  we obtain that  for every $u\in[\varrho, (\varrho+a)\wedge s]$
	\begin{equation}
	\label{equppcontrGa}
	\begin{split}
	\ell_u \, & =\,  \ell_\varrho+2(b_u-b_\varrho)-2\varepsilon \int_\varrho^u \sinh(\ell_v) \dd v\\
	& \le \Gamma+F+2\lambda-2\varepsilon (u-\varrho) \sinh(\Gamma+F)\, .
	\end{split}
\end{equation}
	Now, we argue that $s\le \varrho+a$. In fact, if $s>\varrho+a$ then, by applying  \eqref{equppcontrGa} to $u=\varrho+a$ and by assumption $(2)$, we obtain that 
\begin{equation} 
\ell_{\varrho+a} \, \le\,  \Gamma+F+2\lambda-2\gep a \sinh(\Gamma+F)\,  \le\,  \Gamma+F\, ,
\end{equation}
	which is impossible by the definition of $\varrho$ and because   $s>\varrho+a$.
	We conclude that $s\le \varrho+a$ and therefore inequality \eqref{equppcontrGa} applies to $u=s$, hence
\begin{equation}
l_s  \, \le\,  \Gamma+F+2\lambda\, .
\end{equation}
	This concludes the proof of Lemma~\ref{lemuppboundstep2}.
\end{proof}
	
	\medskip

\begin{proof}[Proof of Corollary~\ref{cor:uppboundstep2}]
	Choose $C_{\gk, 5}\ge C_{\gk, 1}$ and $\lambda\in (0, 1/2)$ and set $F=\loglog \Gamma+C_{\gk, 5}$ and $a=e^{-F}$. 
So $\gep a \sinh(\gG+F)$ is smaller than $1/2$, but approaches $1/2$ for $\gG \to \infty$,  hence 	$\gl \le \gep a \sinh(\gG+F)$ for $\gG$ sufficiently large, i.e.\
	assumption $(2)$ of Lemma~\ref{lemuppboundstep2} holds true. By the Brownian continuity modulus estimate \eqref{BM1}, provided that $e^{C_{\gk, 5}}> 3(\gk+3)/ \gl^2$ we have that, with probability $1-O(\Gamma^{-\gk})$, the 
	supremum of $\vert B_{s}-B_{s'}\vert$, for $s,s' \in [-\Gamma^3,0]$ and $\vert s-s' \vert \le a$, is 
	bounded by $\lambda$. 
	On the other hand, 
	$-\tsdown(\Gamma)$ is distributed as $\gG^2 \ttdown(1)$ and  $\bbE[e^{\ttdown(1)}]<\infty$ according to Lemma \ref{f:1} $(i)$ (since $1<\pi^2/8$), so by Markov inequality we have that, with probability $1-O(e^{-\Gamma})$,
	 $\tsdown(\Gamma)\ge -\Gamma^3$.
	 Therefore by using \emph{Step 1} and by applying Lemma~\ref{lemuppboundstep2} we obtain Corollary~\ref{cor:uppboundstep2} with $\overline{C_{\gk, 2}}=C_{\gk, 5}+2\lambda$.
\end{proof}
	
	Now we turn to the proof of \emph{Step 3}. 

	\subsubsection{Proof of Step 3: lower bound}
		
	By applying Lemma~\ref{th:forTrandom}, \eqref{eq:forTrandom} with $T=\tvdown(\Gamma)$ and $t=0$,
   we obtain that
  \begin{equation}
  \begin{aligned}
  l_0 & \, \ge\, l_{\tvdown(\Gamma)}+2(B_0-B_{\tvdown(\Gamma)}) -\log\left(1+\gep e^{l_{\tvdown(\Gamma)}}  \int_{\tvdown(\Gamma)}^0 e^{2\left(B_s-B_{\tvdown(\gG)}\right)} \dd s\right)\, \\
  & \, = \, -\log\left(e^{-l_{\tvdown(\Gamma)}}+ \gep \int_{\tvdown(\Gamma)}^0 e^{2\left(B_s-B_{\tvdown(\gG)}\right)}\dd s\right)-2B_{\tvdown(\Gamma)}\, ,  
  \end{aligned}
  \end{equation}
where in the second line we used that $B_0=0$. Note that $\int_{\tvdown(\Gamma)}^0 e^{2 \left(B_s-B_{\tvdown(\gG)} \right)} \dd s$ is distributed as $\int_0^{\tudown(\gG)}  e^{-2\left(B_{\tudown(\gG)}-B_t\right)} \dd t $ which is less than  $\int_0^{\ttdown(\gG)}  e^{-2(B_{\tudown(\gG)}-B_t)} \dd t$. Then Lemma \ref{th:excthlemma}$(iii)$ and an application of Markov inequality imply that there exists a positive constant $C_{\gk, 6}$ such that with probability $1-O(\Gamma^{-\gk})$
\begin{equation}
\int_{\tvdown(\gG)}^0 e^{2\left(B_s-B_{\tvdown(\gG)}\right)} \dd s\,\le\,  C_{\gk, 6}\log \Gamma\, .
\end{equation}
Therefore,
using $l_{\tvdown(\gG)} \ge \Gamma-\loglog \gG-\underline{C_{\gk, 2}}$ (i.e., \emph{Step 2}), we conclude that with probability $1- O(\gG^{-\gk})$
\begin{equation}
l_0\, \ge\,  \Gamma - 2B_{\tvdown(\gG)} - \loglog \Gamma-\underline{C_{\gk, 3}}\, ,
\end{equation} 
where $\underline{C_{\gk, 3}}=\log\left(e^{\underline{C_{\gk, 2}}}+C_{\gk, 6}\right)$.

	\subsubsection{Proof of Step 3: upper bound}
		
	By applying Lemma~\ref{th:forTrandom}, \eqref{eq:forTrandom_upp}  with $T=\tvdown(\Gamma)$ and $t=0$,
   we obtain that
  \begin{equation}
  \begin{aligned}
  l_0 & \, \le\, l_{\tvdown(\Gamma)}+2(B_0-B_{\tvdown(\Gamma)}) +\log\left(1+\gep e^{-l_{\tvdown(\Gamma)}}  \int_{\tvdown(\Gamma)}^0 e^{-2\left(B_s-B_{\tvdown(\gG)}\right)} \dd s\right)\\
  & \, = \, \log\left(e^{l_{\tvdown(\Gamma)}}+ e^\gG \int_{\tvdown(\Gamma)}^0 e^{-2\left(\gG-(B_{\tvdown(\gG)}-B_s)\right)}\dd s\right)-2B_{\tvdown(\Gamma)}\, ,
  \end{aligned}
  \end{equation}
  where in the second line we used that $B_0=0$ and we wrote $\gep=e^\gG\times e^{-2\gG}$. Note that conditioning on the event $\{\tsdown(\Gamma)>\tsup(\Gamma)\}$, $\int_{\tvdown(\gG)}^0 e^{-2\left(\gG-(B_{\tvdown(\gG)}-B_s)\right)} \dd s$ is distributed like $\int_0^{\tudown(\gG)}  e^{-2(\Gamma-(B_{\tudown(\gG)}-B_t))} \d t$ conditioned on $\{ \ttdown(\gG)<\ttup(\gG)\}$.  It follows from lemma \ref{th:excthlemma}$(iv)$ and an application of Markov inequality that there exists a positive constant $C_{\gk, 7}$ such that with probability $1-O(\Gamma^{-\gk})$, on the event $\{\tsdown(\Gamma)>\tsup(\Gamma)\}$, we have 
\begin{equation}
\int_{\tvdown(\gG)}^0 e^{-2\left(\gG-(B_{\tvdown(\gG)}-B_s)\right)}  \dd s\,\le\,  C_{\gk, 7}\log \Gamma\, .
\end{equation}

Using $l_{\tvdown(\gG)} \le \Gamma+\loglog \gG+\overline{C_{\gk, 2}}$ (i.e., the upper bound in \emph{Step 2}), we conclude that with probability $1- O(\gG^{-\gk})$, on the event $\{\tsdown(\Gamma)>\tsup(\Gamma)\}$,
\begin{equation}
l_0\, \le\,  \Gamma - 2B_{\tvdown(\gG)} + \loglog \Gamma+\overline{C_{\gk, 3}}\, ,
\end{equation} 
where $\overline{C_{\gk, 3}}=\log\left(e^{\overline{C_{\gk, 2}}}+ C_{\gk, 7}\right)$.

This completes the proof of \emph{Step 3} and therefore also the proof of Lemma~\ref{th:modelsclose} is complete.
\end{proof}

\subsection{Proof of Proposition~\ref{th:formain}} 
\label{sec:proof-formain}
\begin{proof}[Proof of Proposition~\ref{th:formain}]

First we observe that we can replace $s^\ssup{F}$ with $\sFR$. In fact these two trajectories, as discussed at length in App.~\ref{sec:alt},
coincide on $(\tu_{1}(\gG), \infty)$ (actually even on $[0, \infty)$ if $\tu_1(\gG)$ is a true $\gG$-extremum), hence
\begin{equation}
\left \vert \int_0^\ell  \ind_{{\bf s}_t \neq \sFR_t} \dd t  -
\int_0^\ell   \ind_{{\bf s}_t \neq \sFR_t} \dd t \right \vert \, \le \, \tu_{1}(\gG)\, .
\end{equation}
So from now on we study the convergence (in $\e[\, \mu_{\gG, B_\cdot, \ell} [\, .\, ]\, ]$-expectation, and for almost every $B_\cdot$ in $\mu_{\gG, B_\cdot, \ell}$-expectation and in $\mu_{\gG, B_\cdot, \ell}$-probability) of 
\begin{equation}
   \label{eq:defIwo(bc)}   
   I_\ell\left(({\bf s}_t )_{t\in[0, \ell]}\right):=\frac1\ell \int_0^\ell \ind_{{\bf s}_t \neq \sFR_t} \dd t \, .
   \end{equation}
toward $D_\gG$. 


Recall from Section~\ref{sec:defs} that $m_t^\ssup{a,b}= l^\ssup{a}_t+r_t^\ssup{b}$ and {that, from
\eqref{eq:muelleta2}, one obtains}
\begin{equation}\label{eq:forformain2}
 \mu_{\gG, B_\cdot, a, b}\left({\bf s}_t  \ne \sFR_t\right) =\left(1+\exp\big(m_t^\ssup{a,b}\sFR_t\big)\right)^{-1} .
\end{equation}  
Note that, by Lemma \ref{th:invtime}  , as $a\to- \infty$ and $b \to \infty$  the limit
of $\mu_{\gG, B_\cdot,a, b}\left({\bf s}_t \not= \sFR_t\right)$ is equal to 
\begin{equation}\label{eq:forformain3}
\left(1+\exp\big(m_t\sFR_t\big)\right)^{-1}
=:\mu_{\gG,B_\cdot,-\infty,\infty}\left({\bf s}_t \not= \sFR_t\right).
\end{equation}
%
%

From Lemma \ref{th:OCbound} we derive the following Lemma, which quantifies the convergence in \eqref{eq:forformain3}.
\medskip

\begin{lemma}\label{th:OCboundforformain}
Choose any $\gG>0$, any trajectory $B_\cdot$ and any real numbers $a < t < b$. Then
\begin{equation}
\left|\mu_{\gG,B_\cdot, a, b}\left({\bf s}_t \not= \sFR_t\right)- \mu_{\gG,B_\cdot,-\infty, \infty}\left({\bf s}_t \not= \sFR_t\right)\right| 
\le A_\gep(t-a)+A_\gep(b-t)\, ,
\end{equation}
where we denoted $A_\gep(t)=2 \, \arcth\left(e^{-2\gep t}\right)$ for every $t>0$.
\end{lemma}
\medskip

We observe that $A_\gep$ is integrable on $(0, \infty)$ and $\int_0^\infty A_\gep(t) \dd t=\pi^2/(8 \gep)$. 
\medskip 

\begin{proof}
Using \eqref{eq:forformain2} and \eqref{eq:forformain3} we have
\begin{equation} \label{eq:0lto-infinf}
\begin{split}
  \Big|\mu_{\gG,B_\cdot, a, b}\left({\bf s}_t \not= \sFR_t\right)- & \mu_{\gG,B_\cdot,-\infty, \infty}\left({\bf s}_t \not= \sFR_t\right)\Big| \\
  & \le \frac12 \left |m^\ssup{a, b}_t - m_t\right| \\ 
  & \le \frac12 \left| l_t^\ssup{a} - l_t\right |
    +\frac12 \left| r_t^\ssup{b} - r_t\right|\\ 
  & \le \, A_\gep(t-a)+A_\gep(b-t) \, ,
  \end{split}
  \end{equation}
  where for the first inequality we have used that $(1+\exp(\cdot))^{-1}$ is $\frac12$-Lipshitz and for the third inequality we have used Lemma \ref{th:OCbound}.
 \end{proof}

Now, we are ready to prove \eqref{eq:formain1} (with, as argued, $\sFR$ instead of $\sF$). From Lemma \ref{th:OCboundforformain} we have 
\begin{equation} \label{eq:0lto-infinf-1}
\begin{split}
  \bigg|
  \int_0^\ell \mu_{\gG,B_\cdot, 0,\ell}\left({\bf s}_t \not= \sFR_t\right)\, - \, &\int_0^\ell  \mu_{\gG,B_\cdot,-\infty, \infty}\left({\bf s}_t \not= \sFR_t\right) \dd t \bigg|
    \\
    & \le\int_0^\ell \left|\mu_{\gG,B_\cdot, 0,\ell}\left({\bf s}_t \not= \sFR_t\right)-  \mu_{\gG,B_\cdot,-\infty, \infty}\left({\bf s}_t \not= \sFR_t\right) \right| \dd t
  \\ &
   \le \, \int_0^\ell  A_\gep(t) +   A_\gep(\ell-t) \dd t 
   \\ & 
   \le \, 2 \int_0^\infty \,A_\gep(t) \dd t \, <\,  \infty\, .
  \end{split}
  \end{equation}

In view of \eqref{eq:0lto-infinf-1}
the proof of \eqref{eq:formain1} is straightforward  because, by time invariance of the Brownian motion, $\e[\mu_{\gG,B_\cdot,-\infty, \infty}\left({\bf s}_t \not= \sFR_t\right)]=\e[\mu_{\gG,B_\cdot,-\infty, \infty}\left({\bf s}_0 \not= \sFR_0\right)]$ for every $t\in\r$. 
 \medskip
  
 By  \eqref{eq:0lto-infinf-1} we have also that
  \eqref{eq:formain2} follows if we show that almost surely
  \begin{equation}
  \label{eq:tbp6wEBT}
   \frac1\ell \int_0^\ell \mu_{\gG,B_\cdot,-\infty, \infty}\left({\bf s}_t \not= \sFR_t\right) \dd t  \overset{\ell \to \infty}{\longrightarrow} \e\left[\mu_{\gG,B_\cdot,-\infty, \infty}\left({\bf s}_0 \not= \sFR_0\right)\right]\, .
   \end{equation}
   This is a direct application of the (continuous-time) Birkhoff Ergodic Theorem {(see for example \cite[p.~10]{cf:Krengel})}.
   In fact, observe that the family of operators $(\varphi^t)_{t\in\mathbb{R}}$ defined on the Wiener space, i.e. the space of continuous functions endowed with the topology of uniform convergence over compact sets,  by $\varphi^t(B_\cdot)=B_\cdot^\ssup{t}$ where we recall that
  \begin{equation}
  B^\ssup{t}_s\,=\, B_{t+s}-B_t\,, \qquad \text{for } s\in\mathbb{R}\, ,
  \end{equation}
  is an ergodic flow with respect to the (bilateral) Wiener measure $\bbP$. If we set
  \begin{equation}
  f(B_\cdot)\, :=\, \mu_{\gG,B_\cdot,-\infty, \infty}\left({\bf s}_0 \not= \sFR_0\right)\,,
  \end{equation}
 we have that $f(\cdot)\in \bbL ^\infty(\bbP)$. 
Therefore Birkhoff Ergodic Theorem yields that a.s.
  \begin{equation}
  \frac1\ell \int_0^\ell f\left(\varphi^t (B_\cdot)\right) \dd t \overset{\ell \to \infty}{\longrightarrow} \bbE\left[f(B_\cdot)\right]\, ,
  \end{equation}
  which is \eqref{eq:tbp6wEBT} and therefore \eqref{eq:formain2} is established.
  
  \medskip
  
  For the rest of the proof, i.e. establishing \eqref{eq:formain3},
   we need to be more explicit about boundary conditions. 
  Lemma \ref{th:OCboundforformain} can be extended to arbitrary boundary conditions by minimal changes in the proofs, see Remark~\ref{rem:bc0} and Remark~\ref{rem:bc}. Here we introduce 
  the notations for arbitrary (fixed) boundary conditions: for every $B_\cdot$, $\ell>0$ and for every boundary conditions $(bc)\in\{(++), (+-), (-+), (--)\}$, we denote by $\mu_\ell^\ssup{bc}=\mu_{\gG, B_\cdot, \ell}^\ssup{bc}$ the law introduced in Section \ref{sec:model}, but with the boundary conditions $(bc)$.
In particular, $\mu_{\gG, B_\cdot, \ell}$ (introduced in Section~\ref{sec:model}) coincides with  $\mu_{\gG, B_\cdot, \ell}^\ssup{++}$. 
We introduce also $\mu_{a, b}^\ssup{bc}=\mu_{\gG, B_\cdot, a, b}^\ssup{bc}$ which extends the notation $\mu_{a,b}=\mu_{\gG, B_\cdot, a, b}$ to general boundary conditions.
We note that the term $\mu_{\gG,B_\cdot,-\infty, \infty}\left({\bf s}_t \not= \sFR_t\right)$ in Lemma \ref{th:OCboundforformain} is independent of the boundary conditions.
For conciseness we formulated Proposition \ref{th:formain} and we now show \eqref{eq:formain3} with boundary conditions $(++)$ but in fact \eqref{eq:formain1},  \eqref{eq:formain2}  and \eqref{eq:formain3} hold for general boundary, with the same proof.

\smallskip

To prove the convergence in $\mu_{\gG, B_\cdot, \ell}$-probability of variable $I_\ell\left(({\bf s}_t )_{t\in[0, \ell]}\right)$, see \eqref{eq:defIwo(bc)},  we are going to control its variance under $\mu_{\gG, B_\cdot, \ell}$. Then we start by  computing the second moment. Using the fact that for $0\le t_1<t_2\le \ell$ and $\sigma_1, \sigma_2\in\{+1, -1\}$ we have
\begin{equation}
\mu_{\gG, B_\cdot, \ell}\left({\bf s}_{t_1}=\sigma_1, {\bf s}_{t_2}=\sigma_2\right)\,
=\, \mu_{\gG, B_\cdot, 0, \ell}^\ssup{++}\left({\bf s}_{t_1} =\sigma_1\right)
\mu_{\gG, B_\cdot, t_1, \ell}^\ssup{\sigma_1+}\left({\bf s}_{t_2}=\sigma_2\right)\, ,
\end{equation}
we establish that
\begin{multline}
\label{eq:spsfb}
\mu_{\gG, B_\cdot, \ell}\left( \left( \int_0^\ell \ind_{{\bf s}_t \neq \sFR_t} \dd t \right)^2 \right) \,
 = \\
 2 \int_0^\ell \int_{t_1}^\ell \mu_{\gG, B_\cdot, 0, \ell}^\ssup{++}\left({\bf s}_{t_1} \neq \sFR_{t_1}\right)
\mu_{\gG, B_\cdot, t_1, \ell}^\ssup{-\sFR_{t_1}+}\left({\bf s}_{t_2}\neq  \sFR_{t_2}\right) \dd t_2 \dd t_1\,.
\end{multline}
By using Lemma \ref{th:OCboundforformain} we bound the second moment in \eqref{eq:spsfb} by 
\begin{equation}
2 \int_0^\ell \int_{t_1}^\ell \mu_{\gG, B_\cdot,-\infty, \infty}\left({\bf s}_{t_1} \neq \sFR_{t_1}\right)
\mu_{\gG, B_\cdot, -\infty, \infty}\left({\bf s}_{t_2}\neq  \sFR_{t_2}\right) \dd t_2 \dd t_1\, ,
\end{equation}
i.e., $\left(\int_0^\ell \mu_{\gG, B_\cdot,-\infty, \infty}\left({\bf s}_{t} \neq \sFR_t\right)\right)^2$, plus an error term which can be bounded by
\begin{equation}
2 \int_0^\ell \int_{t_1}^\ell   
A_\gep(t_1) +  A_\gep(\ell- t_1) + A_\gep((t_2-t_1) + A_\gep(\ell-t_2)
\dd t_2 \dd t_1\,,
\end{equation}
where we used that if $x,x',y,y'\in[0,1]$, then $xy\le x'y'+|x-x'|+|y-y'|$. This error term is $O(\ell)$.

Then we observe that, by  \eqref{eq:0lto-infinf-1}, we have also that
\begin{equation}
\left(\mu_{\gG, B_\cdot, \ell}\left[ \int_0^\ell \ind_{{\bf s}_t \neq \sFR_t} \dd t\right]\right)^2\, 
= \, \left(\int_0^\ell \mu_{\gG, B_\cdot,-\infty, \infty}\left[{\bf s}_{t} \neq \sFR_t\right]\right)^2 + O(\ell).
\end{equation}

Therefore, we have shown that the variance of $I_\ell\left(({\bf s}_t )_{t\in[0, \ell]}\right)$ under $\mu_{\gG, B_\cdot, \ell}$ is $O(1/\ell)$. In particular it converges to 0 as $\ell$ goes to infinity and the proof of Proposition~\ref{th:formain} is complete. 
\end{proof}

\appendix

\section{Some results about Brownian motion}
\label{sec:B}

First of all, we need 
a control  on the continuity modulus of Brownian motion:  
there exists a constant $c_1>0$ such that for every  $t>a>0$ and $\lambda >0$, we have 
\begin{equation}
\label{BM1} 
\p\left( \sup_{0\le u\le s \le t, s-u \le a} \left\vert B_s-B_u\right\vert\ge \lambda\right) \,\le\,  c_1 \frac{t}{a} \exp\left(- \frac{\lambda^2}{3 a}\right)\,,
\end{equation} 
where $3$ may be replaced by any number larger than 2 (for a proof, see \cite{cf:CR81}).
\smallskip

The rest of the results of this appendix are more specific and related to the sequence of $\gG$-extrema.
Recall that, by Brownian scaling,  $(\ttdown(\Gamma), \tudown(\Gamma)) \law \Gamma^2 (\ttdown(1), \tudown(1))$
for any $\Gamma>0$. 
\smallskip

\begin{lemma} 
\label{f:1} Let $\Gamma>0$. 

(i)  The random variable $B_{\tudown(\Gamma)}$ is exponentially distributed with mean $\Gamma$ and  
\begin{equation} 
\e \left[e^{- \lambda \tudown(\Gamma)}\right] \,=\, \frac1{\Gamma\sqrt{2\lambda} \coth( \Gamma\sqrt{2\lambda})}\, , \qquad \e \left[e^{- \lambda \ttdown(\Gamma)}\right] \,=\, \frac1{\cosh( \Gamma\sqrt{2\lambda})}\, ,
\end{equation}
for every $\gl\ge 0$. In particular, $\e [e^{ \lambda \ttdown(\Gamma)}]<\infty$ if $\lambda < \frac{\pi^2}{8\Gamma^2}$. 

(ii) Let $(\tu_n(\Gamma))$ be the sequence of the times of $\Gamma$-extrema of $(B_t)_{t\ge 0}$. Then $(\tu_{n+1}(\Gamma)- \tu_n(\Gamma))_{n=1,2, \ldots}$ is an  i.i.d. sequence, with Laplace transform ($\gl>0$)
\begin{equation}
\e \left[e^{-\lambda(\tu_{n+1}(\Gamma)- \tu_n(\Gamma))}\right]\, =\,  \frac1{\cosh(\Gamma \sqrt{2 \lambda})}\, .
\end{equation} In particular $\e[\tu_{n+1}(\Gamma)- \tu_n(\Gamma)]= \Gamma^2$.


 (iii) The process  $(B_{\tudown(\Gamma)}- B_{\tudown(\Gamma)+t})_{0\le t \le \ttdown(\Gamma)-\tudown(\Gamma)}$ is distributed as a three-dimensional Bessel process $R_3$, starting from $0$ and running until $T_\Gamma(R_3):= \inf\{t>0: R_3(t)=\Gamma\}$. 
 Moreover
  \begin{equation}  \e \exp\left( a \int_{\tudown(\Gamma)}^{\ttdown(\Gamma)} e^{-2 (B_{\tudown(\Gamma)}- B_t)} \d t \right) \le 
 \begin{cases}\frac{1}{J_0(\sqrt{2a})}, \qquad &\mbox{if  $0< a < \frac{j_0^2}2$}, \\
    \frac{1}{\cos(\Gamma\sqrt{2 a})},  \qquad &\mbox{if  $0< a < \frac{\pi^2}{8\Gamma^2}$}\end{cases}  , \label{eq:int1}  \end{equation} 
\noindent where $J_0(\cdot)$ denotes the Bessel function of the first kind \cite[(10.2.2)]{cf:DLMF} with index $0$ and $j_0>0$ is its smallest positive zero.  
 \end{lemma}



\medskip

 We note that the first upper bound in \eqref{eq:int1} holds uniformly in $\Gamma$, and  the second one involves $\Gamma$ and will be useful when $\Gamma$ is small.

\begin{proof} Only $(iii)$   needs  a proof because $(i)$ and $(ii)$ are proven in \cite{cf:NP89}.  For the  first part of $(iii)$ we observe that, by L\'evy's identity,  $\beta_t:=\sup_{s\le t} B_s -B_t, t\ge 0$,  is distributed as a reflected Brownian motion. Then   $\ttdown(\Gamma)$  becomes the first hitting time of $\Gamma$ by $\beta$ and  $\tudown(\Gamma)$ is the last zero of $\beta$ before $\ttdown(\Gamma)$. Therefore, by applying Williams' path decomposition (\cite{RY}, Theorem VI.3.11) one recovers  the first part of $(iii)$, which in particular implies 
that
 \begin{equation}  \int_{\tudown(\Gamma)}^{\ttdown(\Gamma)} e^{-2 (B_{\tudown(\Gamma)}- B_t)} \d t \, \law\, \int_0^{T_\Gamma(R_3)} e^{-2 R_3(t)} \d t\, , \label{eq:R_3} \end{equation}
which is our first step toward \eqref{eq:int1}. Now, if 
 we denote by $L_{R_3}(\infty, x), x \ge 0, $ the local times of $R_3$ at position $x$ and time $\infty$, then by  \cite{legall-yor} we have that  $L_{R_3}(\infty, \cdot)$ is distributed as $Z_2(\cdot)$, i.e. the square of a two-dimensional Bessel process starting from $0$. Consequently $ \int_{\tudown(\Gamma)}^{\ttdown(\Gamma)} e^{-2 (B_{\tudown(\Gamma)}- B_t)} \d t$ is stochastically dominated by $\int_0^\Gamma e^{-2s} Z_2(s) \d s$.  
Now we use that    for every $\lambda>0$ \cite[p.~436]{PY82} 
\begin{equation} 
\e \exp\left( - \lambda \int_0^\infty e^{-2s} Z_2(s) \d s \right)\,=\,  \frac{1}{I_0(\sqrt{2\lambda})}, 
\end{equation}
 which after an analytic  continuation argument yields that for every  $a \in (0,{j_0^2}/2)$  
 \begin{equation} 
  \e \exp\left( a \int_0^\infty e^{-2s} Z_2(s) \d s \right)\,=\,  \frac{1}{J_0(\sqrt{2a})}\, .
  \end{equation}
  Since $ \int_{\tudown(\Gamma)}^{\ttdown(\Gamma)} e^{-2 (B_{\tudown(\Gamma)}- B_t)} \d t$ is stochastically dominated by $\int_0^\infty e^{-2s} Z_2(s) \d s$, we get the first upper bound in \eqref{eq:int1}. 
 
 For the second one, we use $\int_0^\Gamma e^{-2s} Z_2(s) \d s\le \int_0^\Gamma   Z_2(s) \d s$ and the Cameron--Martin formula  \cite[Corollary XI.1.8]{RY}: for every $\lambda>0$ 
 \begin{equation}
 \e \exp\left( - \lambda \int_0^\Gamma   Z_2(s) \d s \right)\,=\,  \frac{1}{\cosh(\Gamma\sqrt{2\lambda})}\, , 
 \end{equation}
 and 
 we conclude again by the analytic  continuation. 
\end{proof}


%

  \medskip

 Now, we collect some technical results. Recall Lemma \ref{th:invtimesimp} for the definitions of $\widehat \ell_0$ and $\widehat r_0$. 
 \medskip
 
 \begin{lemma}\label{th:excthlemma}

(i)   The largest drop before $\tudown(\Gamma)$,  i.e. $B^\downarrow_{\tudown(\Gamma)}$,  is uniformly distributed in $[0,\Gamma]$. 

(ii) The random variables  $\widehat \ell_0$ and $\widehat r_0$ are uniformly distributed in $[-\Gamma, \Gamma]$. 

(iii) There exists a constant $c_2>0$ such that  
\begin{equation}
\label{eq:expon1}
 \sup_{\Gamma\ge 1} \e \exp\left( c_2 \int_0^{\ttdown(\gG)}  e^{-2(B_{\tudown(\gG)}-B_t)} \d t \right) <\infty. 
\end{equation}

%

(iv) 
There exists a constant $c_3>0$ such that
\begin{equation}
\sup_{\Gamma\ge 1} \e\left[\exp\left(c_3 \int_0^{\tudown(\gG)}  e^{-2(\Gamma-(B_{\tudown(\gG)}-B_t))} \d t\right) \, \Bigg| \, \ttdown(\gG)<\ttup(\gG)\right]\,<\, \infty. \label{eq:expon2} 
\end{equation}



\end{lemma}

%

\begin{proof}
 The proof relies on the excursion theory. Let us review some facts and the context in which it applies. Denote by $(e_t, t>0)$  the excursion process  associated with {$\sup_{0\le s\le t} B_s - B_t, t\ge 0$}, the Brownian motion $B$ reflected at its maximum. 
 Specifically if  we define for $r>0$
\begin{equation}\label{def:Tr}  T_r:= \inf\{t>0: B_t >r\}\, ,  \end{equation}      
then the excursion process is defined as 
\begin{equation}\label{def:er} e_r(s):= r- B_{T_{r-}+s} \ \text { for } s \in [0, \zeta(e_r)]\, , 
\end{equation}
 if  $\zeta(e_r):= T_r- T_{r-}>0$ (and $e_r:=\partial$, i.e.  the cemetery point, if $T_r- T_{r-}=0$), {where $T_{r-}:=\lim_{s\uparrow t} T_s$}. 
 Let 
 \begin{equation} \label{def:H}    H\,:=\  \sup_{0\le t \le \ttdown(\Gamma)} B_t=B_{ \tudown(\gG)}\,. \end{equation}
Then 
\begin{equation}\label{eq:H}
H\, =\,  \inf\left\{r>0: \max e_r \ge  \Gamma\right\}\, ,
\end{equation}
where   $\max \gamma:= \sup_{s>0} \gamma(s)$ for every excursion $\gamma$. Denote by $n(d\gamma)$ the It\^{o} measure of the excursion process $(e_r)_{r>0}$. {Note that $n$ is equal to twice the standard It\^{o} measure restricted to positive Brownian excursions. By \cite{RY}, Proposition XII.3.6,   we have  $n(\{\gamma: \max\gamma  \ge  \Gamma\})= 1/\Gamma$,  it follows} that $H$ is exponentially distributed with mean $\Gamma$.

 Denote by $(e'_r)_{r>0}$   the restriction of $(e_r)_{r>0}$ to the set of excursions $\{\gamma: \max \gamma < \Gamma\}$. Then $e'$ is independent of $H$ and for $r< H$, $e'_r= e_r$.  The  It\^{o} measure of $e'$ is given by $n(\d \gamma, \max \gamma<\Gamma)$.

\subsubsection{Proof of $(i)$.}  Remark that $B^\downarrow_{\tudown(\Gamma)}= \sup\{ \max e_r: r < H\}$. It follows from the independence between  $(e'_t)$   and $H$ that for every $u \in (0, \Gamma)$ 
\begin{equation}
\p\left(B^\downarrow_{\tudown(\Gamma)}\le u\right)\,=\, \int_0^\infty e^{-\frac{h}\Gamma}   e^{- h (\frac1{u}- \frac1{\Gamma})} \frac{\d h}{\Gamma}\,=\,  \frac{u}{\Gamma}\,,
\end{equation}
so $(i)$ is proven.

\subsubsection{Proof of $(ii)$.} Plainly $\widehat  \ell_0$ and $\widehat  r_0$ have the same law. By \eqref{def-rs0}, 
  \begin{equation}\label{eq:r_0} 
   - \widehat r_0   \, =\, (-2\sup_{0\le s \le \ttdown(\Gamma)}B_s+\Gamma) \ind_{\{ \ttdown(\Gamma) < \ttup(\Gamma)\}}+ (-2\inf_{0\le s \le \ttup(\Gamma)}B_s-\Gamma) \ind_{\{\ttup(\Gamma)< \ttdown(\Gamma)\}} .
  \end{equation}
 Let us compute the distribution of $H := \sup_{0\le s \le \ttdown(\Gamma)}B_s$ conditioned to  $\{\ttdown(\Gamma) < \ttup(\Gamma)\}$. For every $s<\Gamma$ we have 
 \begin{equation}
 \left\{H \le s,\,  \ttdown(\Gamma) < \ttup(\Gamma)\right\}\,=\,  \left\{H \le  s \text{ and } H-r + \max e'_r < \Gamma \text{ for every } r\in [0, s) \right\}\,.
 \end{equation}
  It follows that 
  \begin{equation} 
  \begin{split}
  \p(H \le s, \, \ttdown(\Gamma) < \ttup(\Gamma))\,&=\,  \int_0^s e^{-\frac{h}{\Gamma}} \p\left(  \max  e'_r < \Gamma-h+r \text{ for every } r < h \right) \frac{\d h}\Gamma
 \\
 &=\, \int_0^s  e^{-\frac{h}{\Gamma}} e^{- \int_0^h  n(\{\gamma: \max\gamma \ge  \Gamma-h+r,\,  \max \gamma < \Gamma\}) \d r} \frac{\d h}\Gamma
 \\
 &=\, 
 \int_0^s e^{-\frac{h}{\Gamma}}  e^{- \int_0^h \left(\frac1{\Gamma-h+r}- \frac1\Gamma\right) \d r} \frac{\d h}\Gamma
 \\
 &=\,  \frac12- \frac12 \left(1- \frac{s}\Gamma\right)^2.
 \end{split}
 \end{equation}
  To complete the proof of $(ii)$, we observe that $(-2\inf_{0\le s \le \ttup(\Gamma)}B_s-\Gamma) \ind_{\{ \ttdown(\Gamma) > \ttup(\Gamma)\}}$ is distributed as $(2\sup_{0\le s \le \ttdown(\Gamma)}B_s-\Gamma) \ind_{\{ \ttdown(\Gamma) < \ttup(\Gamma)\}}$. Then for every $s \in [-1, 1]$, we have 
  \begin{equation}
  \begin{split}
\p\left(-\widehat r_0 \le s \, \Gamma\right) \,&=\,  \p\left( -2H+ \Gamma \le s \Gamma, \ttdown(\Gamma) < \ttup(\Gamma)\right) + 
\p\left(  2H - \Gamma \le s \Gamma, \ttdown(\Gamma) < \ttup(\Gamma)\right) 
\\
&=\,  \frac12 \left(1- \frac{1-s}2\right)^2+ \frac12-\left(1- \frac{1+s}2\right)^2 =\frac{1+s}2,
\end{split}
\end{equation}
and the proof of $(ii)$ is complete.

\subsubsection{Proof of $(iii)$.}  At first by  \eqref{eq:int1}, the proof of \eqref{eq:expon1} is reduced to show that for all small $c_2>0$,  \begin{equation}
\label{eq:expon3}
 \sup_{\Gamma\ge 1} \e \exp\left( c_2\,  \int_0^{\tudown(\gG)}  e^{-2\left(B_{\tudown(\gG)}-B_t\right)} \d t \right) \,<\,\infty\, . 
\end{equation}
 Observe that 
 \begin{equation}
 \int_0^{\tudown(\gG)}  e^{-2\left(B_{\tudown(\gG)}-B_t\right)} \d t\,=\, \sum_{0\le r < H} \int_0^{\zeta(e'_r)} e^{-2\left(H-r+e'_r(t)\right)} \d t\, .
 \end{equation}
 Since $\{(r, e'_r): e'_r\neq \partial\}$ forms a Poisson point process with intensity measure   $dr \otimes n(d \gamma, \max \gamma < \Gamma)$) and independent of $H$, we deduce from the exponential formula that 
 \begin{equation}
 \begin{split}   \e \exp&\Big( c_2  \sum_{0\le r < H} \int_0^{\zeta(e'_r)} e^{-2(H-r+e'_r(t))} \d t\Big) \\
&=\,
\int_0^\infty \frac{\d h}{\Gamma} e^{-\frac{h}{\Gamma}}\, \exp\left( \int_0^h \d r \int n(\d \gamma, \max\gamma< \Gamma) \left[ \left(e^{c_0 \int_0^{\zeta(\gamma)}  e^{-2(h-r+\gamma(t))} \d t}\right) -1 \right] \right) \\
&=\,
\int_0^\infty \frac{\d h}{\Gamma} e^{-\frac{h}{\Gamma}}\, \exp\left( \int_0^h \d r \int_0^\Gamma \frac{\d m}{m^2} \left[ \left(\e (e^{c_2 e^{-2 (h-r)} \int_0^{T_{R_3}(m)} e^{-2 R_3(t)} dt})\right)^2 -1\right]\right)\,,
\end{split}
\end{equation}
where in the last equality we use the Williams description of the It\^{o} measure (see \cite{RY}, Theorem XII.4.5) and $R_3$ denotes as before a three-dimensional Bessel process. By \eqref{eq:int1}, for  $0<c_2< \min({j_0^2}/3, {\pi^2}/9)$
\begin{equation}
\e \left[e^{c_2 e^{-2 (h-r)} \int_0^{T_{R_3}(m)} e^{-2 R_3(t)} \d t}\right] \,\le\, 1 + c_4\,  e^{-2 (h-r)} \min\left(1, m^2\right)\,,
\end{equation}
where $c_4=O(c_2)$. It follows that 
\begin{equation}
\int_0^h  \d r \int_0^\Gamma \frac{\d m}{m^2} \left[ \left(\e \left(e^{c_0 e^{-2 (h-r)} \int_0^{T_{R_3}(m)} e^{-2 R_3(t)} \d t}\right)\right)^2 -1\right] \,\le\, c_5 \int_0^h   e^{-2 (h-r)}  \d r\,  \le\, c_5\,,
\end{equation}
for a suitable choice of the positive constant $c_5$. This yields that 
\begin{equation}
\e \exp\left( c_2  \sum_{0\le r < H} \int_0^{\zeta(e'_r)} e^{-2(H-r+e'_r(t))} \d t\right) \, \le \, e^{c_5}\, ,
\end{equation} 
proving \eqref{eq:expon3}, hence $(iii)$ is established.

\subsubsection{Proof of $(iv)$.} The proof of \eqref{eq:expon2} bears some similarities with that of \eqref{eq:expon3}. 
%
%
%
 Like in the proof of $(iii)$, we have 
 \begin{equation}
 \int_0^{\tudown(\gG)}  e^{-2\left(\Gamma-\left(B_{\tudown(\gG)}-B_t\right)\right)} \d t \,=\,\sum_{0\le r < H} \int_0^{\zeta(e'_r)} e^{-2\left(\Gamma-H +r -e'_r(t)\right)} \d t\, ,
 \end{equation}
 where we recall that $H$ is defined in \eqref{def:H}.  The condition $\{\ttdown(\gG)<\ttup(\gG)\}$ yields that   $\max e'_r \le \Gamma - H +r$ for every $r \in (0, H)$ (in particular $H \le \Gamma$). It follows that for every $c>0$ 
 \begin{multline} 
\e \left[e^{c \,  \int_0^{\tudown(\gG)}  e^{-2(\Gamma-(B_{\tudown(\gG)}-B_t))} \d t } \, \ind_{\{\ttdown(\gG)<\ttup(\gG)\}}\right]  \label{eq:expon5} 
\\
\le\,
\int_0^{\Gamma} \frac{\d h}{\Gamma} e^{-\frac{h}{\Gamma}} \, \exp\left(\int_0^h \d r \int n(\d \gamma) \ind_{\{\max \gamma \le  \Gamma - h +r\}} \left[ e^{c \int_0^{\zeta(\gamma)} e^{-2 (\Gamma-h+r-\gamma(t))} dt} -1\right]\right) \,, 
\end{multline}
where we have omitted the restriction $\max\gamma< \Gamma$ in the It\^{o} measure $n$ because it is automatically satisfied when $\max \gamma \le  \Gamma - h +r< \Gamma$. Again using Williams' description of $n$, we see that the integral with respect to $n(d \gamma)$ in the right-hand-side of \eqref{eq:expon5} is equal to 
\begin{equation} 
\int_0^{\Gamma - h +r} \frac{\d m}{m^2} \left[ \left(\e e^{c \, e^{-2 (\Gamma-h+r-m)} \int_0^{T_{R_3}(m)} e^{-2 (m-R_3(t))} \d t}\right)^2 - 1\right]\, .
\end{equation}
By time-reversal, $m- R_3(T_{R_3}(m)-t)$, for $ 0\le t \le T_{R_3}(m)$, has the same law as $R_3(t)$, for $ 0\le t \le T_{R_3}(m)$. Therefore  it follows as in the proof of $(ii)$ that for   $c_6>0$ sufficiently small  there exists 
 $c_7>0$ such that for every $\Gamma, h, r, m$, as long as  $\Gamma-h+r-m>0$, we have 
 \begin{equation}
  \left(\e \, e^{c_6\,  e^{-2 (\Gamma-h+r-m)} \int_0^{T_{R_3}(m)} e^{-2 (m-R_3(t))} \d t}\right)^2 - 1 \,\le\, c_7\,  \min \left(1, m^2\right) \, e^{-2 (\Gamma_1-h+r-m)}\, .
  \end{equation} 
 Then for every $h \le \Gamma$ 
 \begin{equation}
 \begin{split}  
  \int_0^h \d r \int n(\d \gamma) \ind_{\{\max \gamma \le  \Gamma - h +r\}} & \left[ e^{c_3 \int_0^{\zeta(\gamma)} e^{-2 (\Gamma-h+r-\gamma(t))} \d t} -1\right]
\\
&\le\,
 c_7\, \int_0^h \d r \int_0^{\Gamma - h +r}   \min (1, m^2) \, e^{-2 (\Gamma-h+r-m)}   \frac{\d m}{m^2}
 \\
 &\le\,
 c_7\, c_8\,    \int_0^h \min\left(1, \left(\Gamma - h +r\right)^{-2}\right)  \d r 
 \\
 &\le\, c_9\,, 
 \end{split}
 \end{equation}
 for a suitable choice of the  positive constant $c_9$ independent of $\Gamma$ and $h$. In the second inequality we have used the existence of  $c_8>0$ such that  $\int_0^\lambda   \min (1, m^2) \, e^{-2 (\lambda-m)} {\d m}/{m^2}  \le   c_8 \, \min(1, \lambda^{-2})$ for every $\lambda>0$.  Going back to \eqref{eq:expon5}, we see that the expectation term in the left-hand-side  is bounded above by $c_9$. This implies $(iv)$ and  completes the proof of Lemma \ref{th:excthlemma}. 
 \end{proof}

\section{On the $\gG$-extrema of the bilateral Brownian motion}
\label{sec:alt}

In this Appendix, we define the (bi-infinite) sequence of $\gG$-extrema of a bilateral Brownian motion $(B_t)_{t\in\R}$ and then we discuss its links with the algorithmic approach we presented in Section~\ref{sec:Gextrema} to define $s_\cdot^\ssup{F}$, i.e., Fisher's trajectory for the one-sided Brownian motion $(B_t)_{t\ge 0}$. 

\subsection{A global approach to the  $\gG$-extrema of the bilateral Brownian motion}
\label{sec:alt-1}




\begin{definition}
Let $B_\cdot=(B_t)_{t\in\r}$ be an element of the Wiener space and let $\tu\in \R$. We say that $\tu$ is a local maximum of $B_\cdot$ if there exist real numbers $a$ and $b$ satisfying $a < \tu < b$ and such that $B_\tu=\sup_{[a,b]}B_\cdot$. If we further have $B_a< B_\tu-\gG$ and $B_b< B_\tu-\gG$, then $\tu$ is said to be $\gG$-maximum of $B_\cdot$, and we say that this is testified by couple $(a, b)$.

The notions of local minimum and $\gG$-minimum are defined accordingly: $\tu$ is a local minimum (respectively, a $\gG$-minimum) of $B_\cdot$ if and only if $\tu$ is a local maximum (respectively, a $\gG$-maximum) of $-B_\cdot$.
\end{definition}

Now, we describe the structure of the sets of $\gG$-maxima and $\gG$-minima of $B_\cdot$.

\begin{lemma}\label{th:AppBlembilatext}
We consider a Brownian trajectory that satisifies the following properties (note that almost every trajectory of $B_\cdot$ does):
\begin{itemize}
\item Trajectory $B_\cdot$ has no distinct local maxima, neither distinct local minima, sharing the same height;
\item Trajectory $B_\cdot$ is recurrent. 
\end{itemize}
Then, the $\gG$-maxima and $\gG$-minima of $B_\cdot$ form two discrete and unbounded (in both directions) subsets of $\R$ which are intertwined, in the sense that they are disjoint and in between two consecutive $\gG$-maxima there is  exactly one $\gG$-minimum (and vice-versa). Furthermore, if $\tu$ is a $\gG$-minimum of $B_\cdot$, then  $(\tu_1, \tu_2)$, where $\tu_1$ and $\tu_2$ are the two adjacent $\gG$-maxima of $B_\cdot$, with $\tu_1<\tu_2$, testifies that $\tu$ is a $\gG$-minimum, and vice-versa.
Finally, considering $\tu<\tu'$ two consecutive $\gG$-extrema,
\begin{itemize}
\item if $\tu$ is a $\gG$-maximum (so $\tu'$ is a $\gG$-minimum), then $\sup_{\tu\le r\le s\le \tu'} B_s-B_r\le \gG$;
\item if $\tu$ is a $\gG$-minimum (so $\tu'$ is a $\gG$-maximum), then $\sup_{\tu\le r\le s\le \tu'} B_r-B_s\le \gG$.
\end{itemize}
\end{lemma}

\begin{proof}
Let us consider $\tu_1<\tu_2$ two consecutive $\gG$-maxima of $B_\cdot$ and associated $(a_1, b_1)$ and $(a_2, b_2)$ coming from the definition of a $\gG$-maximum. We consider $\tu$ the (unique by assumption) time at which $B_\cdot$ reaches its minimum over $[\tu_1, \tu_2]$. We are going to show that $\tu$ is a $\gG$-minimum.
Observe that $\tu_1$ and $\tu_2$ are local maxima of $B_\cdot$, hence $B_{\tu_1}\neq B_{\tu_2}$ by assumption. Without loss of generality, we consider the case that $B_{\tu_1} < B_{\tu_2}$. Then $B_{\tu_2}=\sup_{[a_2,b_2]}B_\cdot$ implies that $a_2>\tu_1$ and therefore we have $ B_{\tu}=\inf_{[\tu_1, \tu_2]} B_\cdot\le B_{a_2}<B_{\tu_2}-\gG$. Since $B_{\tu_1} < B_{\tu_2}$, we also derive that $B_{\tu} < B_{\tu_1}-\gG$. Of course, this implies that $\tu_1<\tu<\tu_2$ and thus couple $(\tu_1, \tu_2)$ testifies that $\tu$ is a $\gG$-minimum.

Hence, we have shown that in between two consecutive $\gG$-maxima there is  \emph{at least} one $\gG$-minimum. Using the same argument (or replacing $B_\cdot$ by $-B_\cdot$) one can show that in between two consecutive $\gG$-minima there is  \emph{at least} one $\gG$-maximum, and this is enough to conclude that in between two consecutive $\gG$-maxima there is  \emph{exactly} one $\gG$-minimum (and vice-versa). By construction, the two $\gG$-maxima surrounding a $\gG$-minimum testify that it is a $\gG$-minimum (and vice-versa), and the two sets are disjoint.

The fact that the set of $\gG$-maxima is discrete is an immediate consequence of the continuity of the trajectory of $B_\cdot$. The unboundedness is a consequence of recurrence, as we explain now. Indeed, if $a<b<c$ are such that $B_b>B_a+\gG$ and $B_c<B_b-\gG$ then considering $\tu$ the time at which $B_\cdot$ reaches its maximum over $[a,c]$, we see that couple $(a, c)$ testifies that $\tu$ is a $\gG$-maximum. By recurrence, there exist such triples $(a,b,c)$ arbitrarily far from the origin. 
Replacing $B_\cdot$ by $-B_\cdot$, the set of $\gG$-minima of $B_\cdot$ is also discrete and unbounded.

Finally, consider $\tu<\tu'$ two consecutive $\gG$-extrema, with $\tu$ a $\gG$-maximum and $\tu'$ a $\gG$-minimum. Assume that there exist $r$ and $s$ such that $\tu\le r\le s\le \tu'$ and $B_s-B_r > \gG$, then we see that $(\tu, s)$ testifies that $r$ is a $\gG$-minimum, hence $r=s=\tu'$, which is absurd. Replace $B_\cdot$ by $-B_\cdot$ to handle the converse case and conclude the proof.
\end{proof}

Based on Lemma \ref{th:AppBlembilatext}, we can index the $\gG$-extrema of the bilateral Brownian motion to obtain a bi-infinite increasing sequence $(\tu^\ssup{\r}_n(\gG))_{n\in\Z}$, together with alternating labels $\arrow^\ssup{\r}_n\in \{-1, +1\}$: 
if $\arrow^\ssup{\r}_n=+1$ (respectively $-1$) then $(B_t)_{t \in \bbR}$ has a $\gG$-maximum (respectively, a $\gG$-minimum) at $\tu^\ssup{\r}_n(\gG)$. A convention is needed to fix the indexation: we choose $\tu^\ssup{\r}_1(\gG)$ to be the smallest non-negative extremum. We then define the Fisher trajectory $\sFR$ for the bilateral Brownian motion by setting, for almost every trajectory of $B_\cdot$,
\begin{equation}
\label{eq:sFRalt}
 \sFR_t\, =\,  \sum_{n=-\infty}^\infty \arrow^\ssup{\r}_{n} \ind_{ (\tu^\ssup{\r}_{n-1}(\gG), \tu^\ssup{\r}_{n}(\gG))}(t)
 \,,
 \end{equation}
for every $t\in \bbR $. 


\begin{rem}
The reader may refer to \cite{cf:Brox}  for a slightly different presentation. 
If $u_0$ is a $\gG$-minimum, then $(u_{-1}, u_0, u_1)$ is, in the terminology of  \cite{cf:Brox}, the \emph{valley} of depth $\gG$ containing $0$, and if $u_0$ is a $\gG$-maximum, then $(u_0, u_1, u_2)$ is the \emph{valley} of depth $\gG$ containing $0$. Denoting, as in \cite{cf:Brox}, by $m$ the location of the bottom of the \emph{valley} containing 0, we have $\sFR_0=-\sign(m)$, with the convention $\sign(0)=0$.
\end{rem}

\subsection{Links with Section \ref{sec:Gextrema} and a local approach to $\gG$-extrema}
\label{sec:alt-2}

In this section, we aim at understanding how one can concretly determine the full sequence of $\gG$-extrema of the bilateral Brownian motion. We will work with a trajectory of $(B_t)_{t\in\R}$ that meets the assumptions of Lemma \ref{th:AppBlembilatext}. 
We turn to the procedure presented in Section \ref{sec:Gextrema}, which allowed us to define there trajectory $\sF$ for the one-sided Brownian motion $(B_t)_{t\ge 0}$. Recall that we built the sequence of stopping times $(\ttt_n(\gG))_{n\ge 1}$ (finite by recurrence) and the sequence of $\gG$-extrema with arrows $(\tu_n(\gG), \arrow_n)_{n\ge 1}$ by exploring the Brownian motion progressively from 0 to infinity. We call this procedure the \emph{forward Neveu-Pitman procedure}.

Recall the definition of $\ttdown(\gG)$, $\ttup(\gG)$, $\tudown(\gG)$, $\tuup(\gG)$, $\tu_n(\Gamma)$ and $\arrow_n(\gG), n\ge 1$ in Section~\ref{sec:Gextrema}, just before \eqref{eq:sFisher}.
They depend of course on the Brownian trajectory and the complete notation is, as in 
\eqref{eq:completenotation1} and \eqref{eq:completenotation2},  $\ttdown(\gG)=\ttdown_B(\gG)$ and so on.
This cumbersome notation will be useful here because 
we are going to consider the analogous random variables for the time-shifted Brownian motion $B^\ssup{a}=(B_{a+t}-B_a)_{t\in\r}$, for $a\in\r$, and for the time reversed Brownian motion $\Brev=(B_{-t})_{t \in\R}$. Note that the assumptions of Lemma \ref{th:AppBlembilatext} are invariant by time-shifting and time reversal. 

\subsubsection{Forward Neveu-Pitman procedure}

We apply the forward Neveu-Pitman procedure to $(B_t)_{t\ge 0}$, the restriction of $B_\cdot$ to $[0, \infty)$. We observe that, for each $n\ge 2$, time $\tu_n(\gG)$ is a $\gG$-extremum of the bilateral Brownian motion: indeed, if $\arrow_n=+1$ (respectively, $\arrow_n=-1$), then $(\tu_{n-1}(\gG), \ttt_n(\gG))$ testifies that $\tu_n(\gG)$ is a $\gG$-maximum (respectively, a $\gG$-minimum). Conversely, consider $\tu$ a $\gG$-extremum of $B_\cdot$ belonging to $[0, \infty)$, say $\tu=\tu_n^\ssup{\r}(\gG), n\ge 1$, we are going to show that there exists $k\ge 1$ such that $\tu=\tu_{k}(\gG)$. We proceed by recurrence. To initialize the recurrence, assume that $\tu=\tu_1^\ssup{\r}(\gG)$, the smallest positive $\gG$-extremum of the bilateral $B_\cdot$. Without loss of generality, further assume that $\tu$ is a $\gG$-maximum. We use Lemma \ref{th:AppBlembilatext} bearing in mind that $\tu_0^\ssup{\r}(\gG)<0$. We have $\sup_{0\le r\le s \le \tu} B_r-B_s \le \gG$, hence $\ttdown(\gG)\ge \tu$, and we also have $B_\tu=\sup_{[0, \tu_2^\ssup{\r}(\gG)]} B_\cdot$ and $B_{\tu_2^\ssup{\r}(\gG)}<B_{\tu}-\gG$, hence we derive that $\ttdown(\gG)=\inf\{t>\tu : B_t<B_\tu - \gG\}$ together with $\tu=\tudown(\gG)$, which is either $\tu_1(\gG)$ or $\tu_2(\gG)$.
To conclude the recurrence, assume that $\tu_n^\ssup{\r}(\gG)= \tu_{k}(\gG)$ and that, say, $\arrow_n=+1$, then using the same kind of arguments, we can show that $\ttt_{k+1}(\gG)=\inf\{t>\tu_{n+1}^\ssup{\r}(\gG): B_t>B_{\tu_{n+1}^\ssup{\r}(\gG)} + \gG\}$ and $\tu_{n+1}^\ssup{\r}(\gG)= \tu_{k+1}(\gG)$.

So the forward Neveu-Pitman procedure identifies correctly the $\gG$-extrema  that are  inside $[0, \infty)$, except maybe for $\tu_1(\gG)$ which may or may not be a \emph{true} $\gG$-extremum (i.e., a $\gG$-extremum of the bilateral Brownian motion): as we will discuss at the end of this section,  this depends also on the Brownian trajectory on $(-\infty, 0]$. In particular, $\sF_\cdot$ and $\sFR_\cdot$ coincide on $(\tu_1(\gG), \infty)$.

We generalize this construction by noticing that we can very well start the procedure at any time $a$ instead of time 0 (notably for $a<0$). Indeed, let us apply the forward Neveu-Pitman procedure to $B^\ssup{a}_\cdot$, and shift the resulting random variables by $a$. This yields a sequence of stopping times $(a+\ttt_{B^\ssup{a}, n})_{n\ge 1}$ and a sequence of $\gG$-extrema $(a+\tu_{B^\ssup{a}, n}(\gG), \arrow_{B^\ssup{a},n})_{n\ge 1}$. These are exactly the $\gG$-extrema of the bilateral Brownian motion that are in $[a, \infty)$, except maybe for the first one, $a+\tu_{B^\ssup{a},1}$, which remains uncertain. Thus, taking $a$ arbitrarily large and negative, we can recover the full sequence $(\tu^\ssup{\r}_n(\gG), \arrow^\ssup{\r}_n(\gG))_{n\in\Z}$ of $\gG$-extrema of the bilateral Brownian motion. 
Note that stopping times $(a+\ttt_{B^\ssup{a}, n})_{n\ge 2}$ remain consistent (up to indexation) as $a\to -\infty$. This allows us to define the bi-infinite increasing sequence of times $(\ttt_n^\ssup{\r}(\gG))_{n\in\Z}$, where we fix the indexation by requiring that $\tu^\ssup{\r}_n(\gG) < \ttt^\ssup{\r}_n(\gG) \le \tu^\ssup{\r}_{n+1}(\gG)$. These times are explicitly given by:
\begin{equation}\label{eq:deftjR}
\ttt^\ssup{\r}_n(\gG)=\inf\{\tu^\ssup{\r}_n<t\le \tu^\ssup{\r}_{n+1} \colon B_t=B_{\tu^\ssup{\r}_n(\gG)}-\arrow_n\gG\}.
\end{equation}

\subsubsection{Backward Neveu-Pitman procedure}

We describe yet another construction. Observe that the $\gG$-extrema of $B_\cdot$ are just minus the $\gG$-extrema of $\Brev_\cdot$. Hence applying the forward Neveu-Pitman procedure to $\Brev_\cdot$ allows us to recover the $\gG$-extrema of the bilateral Brownian motion  that are  in $(-\infty, 0]$, just with a doubt concerning the closest to 0.

Recall \eqref{eq:reversed-rv}. We further define 
$\ts_n(\gG)  :=\, -\ttt_{\Brev,n}(\Gamma)$,
$\tv_n(\gG)  :=\, -\tu_{\Brev,n}(\Gamma)$ and
$\arrow'_n\, :=\, \arrow_{\Brev, n}$ for $n=1,2, \ldots$.
The times of $\gG$-extrema of $B_\cdot$  that are in $(-\infty, 0]$ are the times $\tv_n(\gG), n\ge 1$, except maybe for $\tv_1(\gG)$, they are decreasingly indexed. More precisely $\tv_n(\gG)$ is a $\gG$-maximum when $\arrow'_n=+1$, a $\gG$-minimum when $\arrow'_n=-1$.

Again, instead of applying this procedure to $B$ restricted to $(-\infty, 0]$, we can conduct this procedure on $(-\infty, b]$ for any $b$ arbitrarily large and positive, and thus recover all $\gG$-extrema in $\R$. 
We also recover a bi-infinite increasing sequence of times $(\ts^\ssup{\r}_n(\gG))_{n\in\Z}$ (we reversed indexation), given by:
\begin{equation}\label{eq:defsjR}
\ts^\ssup{\r}_n(\gG)=\sup\{\tu^\ssup{\r}_n \le t<\tu^\ssup{\r}_{n+1} \colon B_t=B_{\tu^\ssup{\r}_{n+1}(\gG)}+\arrow_n\gG\}.
\end{equation}

\subsubsection{Matching the forward and backward Neveu-Pitman procedures around 0}

Finally, we focus on one last construction, which puts together the two one-sided approaches. Consider the forward Neveu-Pitman procedure on $[0, \infty)$ and the backward Neveu-Pitman procedure on $(-\infty,0]$, they allow us to locate the $\gG$-extrema in $\r$, with a doubt concerning $\tv_1(\gG)$ and $\tu_1(\gG)$. 
We determine wether $\tv_1(\gG)$ and $\tu_1(\gG)$ are true $\gG$-extrema by distinguishing cases. 
If $\arrow_1=\arrow_1'$, one of the two is a true $\gG$-extremum, which one depends on the sign of $B_{\tu_1(\Gamma)}- B_{\tv_1(\gG)}$. 
If $\arrow_1\ne \arrow_1'$, then they are both true $\gG$-extrema or both not true $\gG$-extrema, depending on the gap $\vert B_{\tu_1(\Gamma)} -B_{\tv_1 (\gG)}\vert$ compared to $\gG$.
Restraining ourselves to $\arrow_1=+1$, each situation is illustrated in Figure \ref{fig:4fig}, where we precised the value of the $\gG$-extrema closest to 0. For conciseness, we do not elaborate, and do not give the proofs, which are elementary. We simply point out that $\sFR_0$ is given by
\begin{equation}
\label{eq:sFR}
\sFR_0\, :=
\begin{cases} 
\arrow_1 & \text{ if } \arrow_1= \arrow'_1  \text{ and } \sign\left( B_{\tu_1(\Gamma)}- B_{\tv_1(\gG)}\right)=a_1 \, ,
 \\
 -\arrow_1 & \text{ if } \arrow_1= \arrow'_1  \text{ and } \sign\left( B_{\tu_1(\Gamma)}- B_{\tv_1(\gG)}\right)=-a_1\, ,
 \\
\arrow_1 & \text{ if }  \arrow_1\neq  \arrow'_1 \text{ and } \left\vert B_{\tu_1(\Gamma)} -B_{\tv_1 (\gG)}\right\vert > \gG \,,
\\
-\arrow_1 & \text{ if }  \arrow_1\neq  \arrow'_1 \text{ and } \left\vert B_{\tu_1(\Gamma)} -B_{\tv_1 (\gG)}\right\vert \le \gG \, ,
\\ 0  & \text{otherwise} 
\, .
\end{cases}
\end{equation}
\emph{Otherwise} consists in the case in which $\arrow_1= \arrow'_1$ and $B_{\tu_1(\Gamma)}- B_{\tv_1(\gG)}=0$, but using the assumption on the trajectory of $B$, this second condition reduces to $\tu_1(\Gamma)=\tv_1(\gG)=0$, which implies that 0 is a $\gG$-extremum of the bilateral Brownian motion (a maximum if $\arrow_1=+1$, a minimum otherwise). This case almost surely does not happen.

\begin{rem}
This last procedure can be used not only around 0, but also around any time $t$, simply by applying it to $B^\ssup{t}_\cdot$ instead of $B_\cdot$. Formula \eqref{eq:sFR} holds \emph{mutatis mutandis}. The \emph{otherwise} case then corresponds to the fact that $t$ is a $\gG$-extremum. If $t$ is a deterministic time, this case can be discarded almost surely, but if $t$ is random, of course we cannot exclude it a priori.

The reader can develop a full comprehension of which case in \eqref{eq:sFR} will occur depending on $t$, by considering the position of $t\in[\tu_j^\ssup{\r}(\gG), \tu_{j+1}^\ssup{\r}(\gG)]$ with respect to $\ttt_j^\ssup{\r}(\gG)$ and $\ts_j^\ssup{\r}(\gG)$ introduced in \eqref{eq:deftjR} and \eqref{eq:defsjR}. Since both $\ttt_j^\ssup{\r}(\gG)\le \ts_j^\ssup{\r}(\gG)$ and $\ts_j^\ssup{\r}(\gG)\le \ttt_j^\ssup{\r}(\gG)$ are possible, there are a variety of possible situations.
\end{rem}

\begin{figure}[ht!]   
 \includegraphics[width=.9\linewidth]{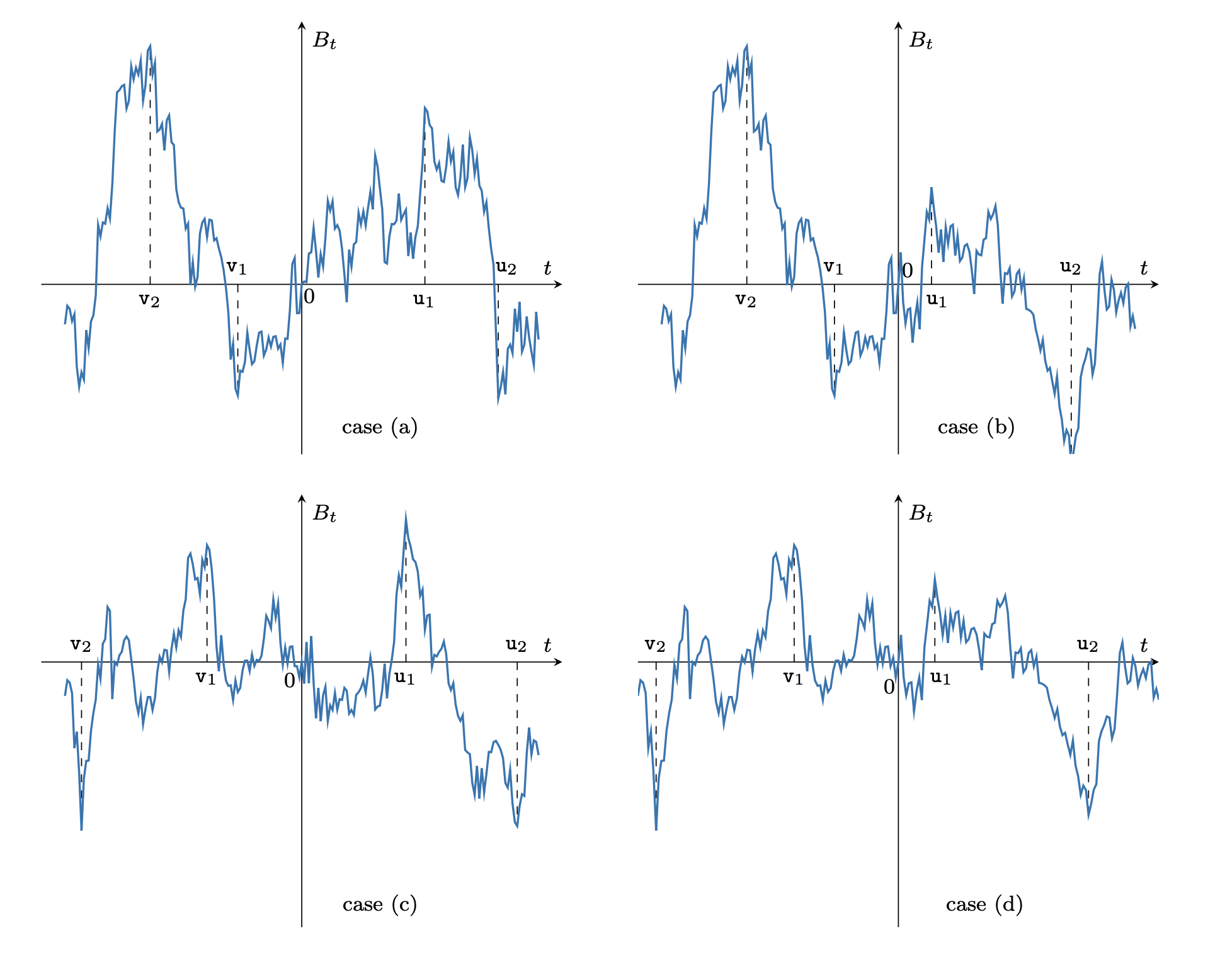}
\caption{
\scriptsize Four illustrations for $\sFR_0$ in \eqref{eq:sFR}. The dependence on $\Gamma$ in ${\tt v}_1, {\tt v}_2, {\tt u}_1, {\tt u}_2$ is omitted.
\\
Case (a):  $1=\arrow_1\neq \arrow'_1=-1$ and $|B_{\tu_1(\Gamma)}- B_{\tv_1(\Gamma)}| > \Gamma.$ Then $\sFR_0= \arrow_1= 1$, $u^\ssup{\r}_0(\gG)=\tv_1(\Gamma)$ and $\tu^\ssup{\r}_1(\gG)=\tu_1(\Gamma)$.
\\
Case (b): $1=\arrow_1\neq \arrow'_1=-1$ and $|B_{\tu_1(\Gamma)}- B_{{\tt v}_1(\Gamma)}| < \Gamma.$ Then $\sFR_0=-\arrow_1=-1$, $u^\ssup{\r}_0(\gG)=\tv_2(\Gamma)$ and $u^\ssup{\r}_1(\gG)={\tt u}_2(\Gamma)$.
\\
Case (c): $\arrow_1= \arrow'_1=1$ and $\sign\left( B_{\tu_1(\Gamma)}- B_{{\tt v}_1(\gG)}\right) = \arrow_1.$ Then $\sFR_0=  \arrow_1= 1$, $u^\ssup{\r}_0(\gG)={\tt v}_2(\Gamma)$ and $u^\ssup{\r}_1(\gG)={\tt u}_1(\Gamma)$.
\\
Case (d): $\arrow_1= \arrow'_1=1$ and $\sign\left( B_{\tu_1(\Gamma)}- B_{{\tt v}_1(\gG)}\right) =- \arrow_1.$ Then $\sFR_0= -\arrow_1= -1$, $u^\ssup{\r}_0(\gG)={\tt v}_1(\Gamma)$ and $u^\ssup{\r}_1(\gG)={\tt u}_2(\Gamma)$.}
\label{fig:4fig}
\end{figure}

\section{From RFIC to Continuum  RFIC}
\label{sec:C0}

We introduce a small parameter $\gD>0$ and we consider  for $\ell>0$, $\ga \in \bbR$ and $\gG>0$
\begin{equation}
\label{eq:scaling}
N=\frac \ell \gD\,, \  \  \gd= \sqrt{\gD}\, , \  \ h= \ga \gD \,  \text{ and }  \, J=\frac \gG 2 -\frac 12 \log \gD\, ,
\end{equation}
where for $N$ we omit the integer part symbol. 
We introduce the partition function of the RFIC (recall \eqref{eq:RFIC})
\begin{equation}
\label{eq:Z-RFIC}
\mathtt{Z}_{N, \go, J, h, \gd}\, :=\, \sum_{\gs\in \{-1,1\}^N}\exp\left( J\sum_{j=1}^{N} \gs _j \gs_{j-1} + \sum_{j=1}^N (\gd\go_j +h) \gs_j \right)
\end{equation}
where $\gs_0=1$ and $(\go_j)$ is an IID sequence. We assume that
$\bbE[ \exp(t \go_1)] < \infty$ for $t$ in a neighborhood of $0$. Moreover $\bbE[\go_1]=0$ and var$(\go_1)=1$.
\medskip

\begin{proposition}
\label{th:PoissonBM}
For every choice of $\ell$, $\ga$ and $\gG$ and 
with $N$, $\gd$, $h$ and $J$ as in  in \eqref{eq:scaling} we have the convergence in law 
\begin{equation}
\label{eq:PoissonBM}
\lim_{\gD \searrow 0}
\frac{\mathtt{Z}_{N, \go, J, h, \gd}}{\mathtt{Z}_{N, J, 0}}\, =\, 
\bE_{\gG}\left[ 
\exp \left( \int_0^\ell \mathbf{s}_t \left( \dd B_t +\ga\dd t \right) \right)
\right]\, .
\end{equation}
\end{proposition}

\medskip 

Note that the right-hand side of \eqref{eq:PoissonBM} coincides with 
  $Z^{\mathtt{f}}_{\gG, B_\cdot, \ell}$ (see \eqref{eq:Z0}). The choice of free boundary condition is just for conciseness.  

\medskip

\begin{proof}
We aim at showing that 
\begin{equation}
\label{eq:rewrite}
\lim_{\gD \searrow 0} \bE_{\gG , \gD}\left[
\exp \left(  \sum_{j=0}^{(\ell/\gD)-1} \gs_j \left(\sqrt{\gD}\go_j + \ga \gD\right)\right) \right]
\, =\, 
\bE_{\gG}\left[ 
\exp \left( \int_0^\ell s_t \left( \dd B_t +\ga\dd t \right) \right)
\right]\, ,
\end{equation}
where with respect to \eqref{eq:PoissonBM} there is the minor change of summing from $0$ up to $N-1$ instead of from $1$ to $N= \ell/ \gD$
and there is, above all, the introduction of the probability measure  $\bE_{\gG , \gD}$ under which 
$( \gs_j)$ is a Markov chain with two states ($-1$ and $+1$) and transition probabilities $Q(+1,+1)= Q(-1,-1)= 1/(1+ \epsilon)$ and 
$Q(+1,-1)=Q(-1,1)= \epsilon/(1+ \epsilon)$, $\epsilon= \exp(-2 J)= \gD \exp(-\gG)$. This is a  consequence of the fact that $\mathtt{Z}_{N, J,0}$ is equal to 
$(1,0)Q^N(1,1)^\mathrm{t}$ times $(1+ \epsilon)^{N}$. 
We recall that $\gs_0=1$ and we remark that, in matrix notation, $(1,0)^\mathrm{t}$ corresponds to spin up and $(0,1)^\mathrm{t}$ to spin down.

Before going into the convergence issues let us remark that it suffices to prove   \eqref{eq:rewrite} with $\exp(\cdot)$ replaced by
$\exp(\cdot ) \wedge L$, for every $L >0$. This is because if we call $Z_{\gD}$ the  partition function in the left-hand side of 
\eqref{eq:rewrite} and $\cZ$ the one in the right hand side, we have that $Z_\gD=  \bE_{\gG , \gD}[\exp(H)]$ with a suitable choice of $H$.
Hence 
\begin{equation}
0 \,\le\,  \bE_{\gG , \gD}[\exp(H)]- \bE_{\gG , \gD}[\exp(H)\wedge L]\,\le\,   \bE_{\gG , \gD}[\exp(H); \exp(H)>L]\, .
\end{equation}
In the same way $\cZ=  \bE_{\gG }[\exp(\cH)]$ and 
\begin{equation}
0 \,\le\,  \bE_{\gG }[\exp(\cH)]- \bE_{\gG }[\exp(\cH)\wedge L]\,\le\,   \bE_{\gG }[\exp(\cH); \exp(\cH)>L]\, .
\end{equation}
Therefore for this preliminary step  it suffices to show that 
\begin{equation}
\label{eq:tbs93}
\lim_{L \to \infty}
\sup_{\gD \in (0, \gD_0)}  \bbE\bE_{\gG , \gD}[\exp(H); \exp(H)>L]= 0\, ,
\end{equation} 
and that $\lim_L \bbE  \bE_{\gG }[\exp(\cH); \exp(\cH)>L]=0$. Equation \eqref{eq:tbs93} follows if we show that 
$\sup_{\gD \in (0, \gD_0)}\bE_{\gG , \gD}[\exp(2H)]< \infty$ and the second statement  follows from $ \bE_{\gG }[\exp(2\cH)] < \infty$. 
For the first bound 
we apply   the Fubini-Tonelli Theorem and we see that 
\begin{equation}
\bbE\bE_{\gG , \gD}[\exp(2H)]=  \bE_{\gG , \gD}\bbE\left[
\exp \left(  2\sum_{j=0}^{(\ell/\gD)-1} \gs_j \left(\sqrt{\gD}\go_j + \ga \gD\right)\right) \right]\overset{\gD\in (0, \gD_0)}\le   
\exp\left((2 \ga+4c) \ell \right), 
\end{equation}
where we have used that the hypotheses on $\go_1$ yield  the existence of $c>0$ and $t_0>0$ such that $\bbE[\exp( t \go_1)] \le 
\exp(c t^2)$ for $\vert t \vert \le t_0$. In the same way we obtain that $ \bbE\bE_{\gG }[\exp(2\cH)]$ is bounded by $ \exp(2 (\ga +4) \ell)$. 
This concludes the argument that shows that it suffices to  show 
\eqref{eq:rewrite} for $\exp(\cdot)$ replaced by
$\exp(\cdot ) \wedge L$.

In order to obtain \eqref{eq:rewrite}, let us remark that, given the initial condition,  the spin sequence $(\gs_j)$  is fully encoded by the sequence of the times of spin switch 
which is given by the renewal sequence $( T^\gD_n)_{n=0,1, \ldots}$ where we set fo convenience $T^\gD_0:=0$
and  $(T^\gD_n- T^\gD_{n-1})_{n=1,2, \ldots}$ is an IID sequence Geom$( \gD \exp(- \gG)/ (1+ \gD \exp(-\gG)))$. 
Note that $( \gD T^\gD_n)_{n=0,1, \ldots}$ converges in law for $\gD \searrow 0$  to $(\tau_n)_{n=0,1, \ldots}$, with $\tau_0=0$ and $( \tau_n-\tau_{n-1})_{n=1,2, \ldots}$ is 
an IID sequence of exponential random variables with mean $\exp(\gG)$, where the topology in the space of non decreasing sequences $\cS_0$ (with zero as first entry, that is for $t_\cdot \in \cS_0$ we have $t_0=0$)  is the 
one induced by the 
standard product topology on the sequence of increments: a distance for this topology is $d(t_\cdot, t'_\cdot)= \sum_{n=1}^\infty 2^{-n} (\vert (t_n-t_{n-1} -(t'_n -t'_{n-1}) \vert \wedge 1)$.  In connection with this sequence we introduce also the random variable $N_{\gD, \ell}: N_\ell(  \gD T^\gD_\cdot)$ and 
$N_\ell(  t_\cdot)=  \sup\{n=0,1, \ldots: \, t_n < \ell\}$. 
  where $N_\ell(\cdot)$ is continuos
at every  sequence $t_\cdot \in \cS_0$ such that $\vert\{j:\, t_j< x\} \vert < \infty$ for every $x>0$ and  such that $t_j \neq \ell$ for every $j$,
we readily see that
the convergence of the sequence $( \gD T^\gD_n)_{n=0,1, \ldots}$ entails the convergence of $N_{\gD, \ell}$ to
$N_\ell (\tau_\cdot)=\sup\{n:\, \tau_n <\ell\}$. Finally, we introduce for $t \ge 0$
\begin{equation}
B_t^\gD\,:=\, \sqrt{\gD} \sum_{j=0}^{\lfloor t/ \gD \rfloor} \go_j\,,   
\end{equation} 
and, by Donsker Theorem, $( B^\gD_{t})_{t \ge 0}$ converges in law to a simple Brownian motion $B$ as $\gD \searrow 0$.
In this case $B^\gD_\cdot$, as well as $B_\cdot$, is an element of the Skorhohod space of CADLAG trajectories, equipped with the standard 
Skorohod topology  \cite[Ch.~3]{cf:Bill}. Of course we are just interested in $t \in [0, \ell]$: with standard notation the function space is $D([0, \ell]; \bbR)$.

Now let us remark that with these notations we can write
\begin{equation}
\begin{split}
  \sum_{j=0}^{(\ell/\gD)-1} \gs_j \left(\sqrt{\gD}\go_j + \ga \gD\right)\, =\, & 
 \sum_{n=1}^{N_{\gD, \ell}} (-1)^{n}  \left( \left(B^\gD_{\gD T^\gD_n}-  B^\gD_{\gD T^\gD_{n-1}}\right)+ \ga \gD\left( T^\gD_n- T^\gD_{n-1}\right) \right) 
 \\
 &+ (-1)^{N_{\gD, \ell} }
  \left( \left(B^\gD_{\ell}-  B^\gD_{\gD T^\gD_{{N_{\gD, \ell} }}}\right)+ \ga \left( \ell- \gD T^\gD_{{N_{\gD, \ell} }}\right) \right)
  \\
  =: &\,  H_\ell\left(\gD T^\gD_\cdot,B^\gD_\cdot \right)\, ,
\end{split}
\end{equation} 
where the first sum should be read as $0$ if $N_{\gD, \ell}=0$. Moreover, let us give explicitly
\begin{multline}
H_\ell\left(t_\cdot,f_\cdot \right)\, =
\\
\sum_{n=1}^{N_\ell(t_\cdot)} (-1)^{n}  \left(f_{t_n}- f_{t_{n-1}}+ \ga \left( t_n- t_{n-1}\right) \right) 
+ (-1)^{N_\ell(t_\cdot) }
  \left( f_{\ell}-  f_{t_{N_\ell(t_\cdot)}}+ \ga \gD\left( \ell- t_{N_\ell(t_\cdot)}\right) \right),
\end{multline}
for $t_\cdot \in \cS_0$, with $N_\ell(t_\cdot)< \infty$, and for $f_\cdot \in D([0, \ell]; \bbR)$.

Since (also) on the limit measure  $N_\ell (\tau_\cdot)< \infty$ a.s.
and since the limit trajectory $B_\cdot$ (a Brownian motion) is continuous, we see that the function $H_\ell: 
\cS\times D([0, \ell]; \bbR)\to \bbR$
is continuous at each $(t_\cdot, f_\cdot)$ with $N_\ell(t_\cdot)< \infty$ and $f_\cdot \in C^0([0, \ell]; \bbR)$  (the topology for $ \cS\times D([0, \ell]; \bbR)$ is the product one).
From this we conclude that we have the convergence in law
\begin{equation}
\label{eq:wconv34}
\exp\left( H_\ell\left(\gD T^\gD_\cdot,B^\gD_\cdot \right)\right)\wedge L\overset{ \gD \to 0}{
\Longrightarrow
}\exp\left( H_\ell\left(\tau_\cdot,B_\cdot \right)\right)\wedge L\,,
\end{equation}
and by the Skorohod representation \cite[Ch.~1, Th.~6.7]{cf:Bill} and the Dominated Convergence Theorem we obtain  
\begin{equation}
\label{eq:wconv34-2}
\bE_{\gG , \gD}\left[
\exp\left( H_\ell\left(\gD T^\gD_\cdot,B^\gD_\cdot \right)\right)\wedge L\right] \overset{ \gD \to 0}{\Longrightarrow}
\bE_{\gG }\left[\exp\left( H_\ell\left(\tau_\cdot,B_\cdot \right)\right)\wedge L\right]\,,
\end{equation}
and the proof of Proposition~\ref{th:PoissonBM} is complete.
\end{proof}
\medskip

We include in this appendix also a quick discussion of overlap issues. 

We start with 
\begin{equation}
\partial_ \gl \log \mathtt{Z}_{N, \go, J, \gl h, \gl \gd}\big \vert_{\gl=1} \,=\, \sum_{\gs\in \{-1,1\}^N} \nu_{N, \go}( \gs) \sum_{j=1}^N \left(\gd \go_j+h \right) \gs_j \, ,
\end{equation}
where $ \nu_{N, \go}(\cdot)$ is the (Gibbs) probability associated to the partition function  
\eqref{eq:Z-RFIC}, that is the probability measure that defines the RFIC.
If $\go_j$ is a standard Gaussian random variable by performing an integration by parts (this is an important and well
known step in the statistical mechanics of disordered systems, see for example \cite[p.~182]{cf:Bov}) we obtain 
\begin{equation}
\label{eq:overlapIsing}
\partial_ \gl \bbE \log \mathtt{Z}_{N, \go, J, h, \gl \gd}\big \vert_{\gl=1} \,=\, h
\bbE 
\sum_{\gs} \nu_{N, \go} (\gs) \sum_{j=1}^N \gs_j +
\bbE 
\sum_{\gs, \gs '} \nu_{N, \go}( \gs)\nu_{N, \go}( \gs') \sum_{j=1}^N \gd^2 \left(1-\gs_j \gs'_j \right) \, ,
\end{equation}
and of course $1-\gs_j \gs'_j= 2\, \ind_{\gs_j \neq \gs'_j}$ so a suitable derivative of the free energy is directly related to the \emph{overlap} between two independent spin realizations that share the same disorder.
This result holds also for the Continuum RFIC, where explicit expressions can be obtained. For conciseness we limit ourselves to  $\ga=0$:

\medskip

\begin{lemma}
\label{th:overlap}
 For $\ga=0$
 \begin{equation}
 \label{eq:overlap-th}
  \lim_{\ell \to \infty} \frac1\ell \e \left[ \mu^{\otimes 2}_{\gG, B_\cdot,\ell}   \left(  \int_0^\ell \ind_{\left\{ {\bf s}_t^\ssup{1} \neq  {\bf s}_t^\ssup{2}\right\}} \dd t   \right)\right] \,=\, \frac{\varepsilon^2}{2} \left(1- \frac{2   K_1(\varepsilon)^2}{K_0(\varepsilon)^2} +   \frac{K_2(\varepsilon)}{K_0(\varepsilon)}\right)\overset{\gG \to \infty} \sim \frac1\Gamma,
  \end{equation}
   where ${\bf s}^\ssup{1}_\cdot,   {\bf s}^\ssup{2}_\cdot$ denote two independent copies of ${\bf s}_\cdot$ under $\mu^{\otimes 2}_\ell$. 
\end{lemma}
\medskip

\begin{proof} We just give a sketch of the argument.
First of all we remark that, by scaling properties, if we replace  $B_\cdot$ with $\gl B_\cdot$ ($\gl>0$) in the Continuum RFIC, 
the free energy density of the model becomes
$\gl^2\tf_0(\gG+2\log \gl)$ and
\begin{equation}
\label{eq:derFla}
\partial_\lambda \left( \gl^2\tf_0\left(\gG+2\log \gl\right)\right)\big\vert_{\gl=1}\, =\, \lim_{\ell \to \infty}
\bbE\left[ \mu_{\gG, B_\cdot,\ell} \left(\frac 1 \ell  \int_0^\ell \mathbf{s}_t \dd B_t \right)\right]\, .
\end{equation}
In analogy with \eqref{eq:overlapIsing} one guesses that 
\begin{equation}
\label{eq:overlap2C0}
\bbE\left[ \mu_{\gG, B_\cdot,\ell} \left(  \int_0^\ell \mathbf{s}_t \dd B_t \right)\right]\, =\, 
\bbE\left[\mu^{\otimes 2}_{\gG, B_\cdot,\ell} \left(  \int_0^\ell \left(  1- {\bf s}^\ssup{1}_t {\bf s}^\ssup{2}_t\right) \dd t \right)
\right]\, ,
\end{equation}
and \eqref{eq:derFla}-\eqref{eq:overlap2C0}, together with the Bessel  function expression for $\tf_0(\gG)$ in \eqref{eq:fe}, lead to  
\eqref{eq:overlap-th}.
A proof of \eqref{eq:overlap2C0} follows by exploiting the result in the discrete framework, i.e. \eqref{eq:overlapIsing},  and  by applying a mild generalization of the statement in 
Proposition~\ref{th:PoissonBM}. In alternative, one can construct a sequence of processes 
$\left(\mathbf{s}_{k,\cdot}\right)_{k \in \bbN}$,
with $\mathbf{s}_{k,\cdot}= \left(\mathbf{s}_{k,t}\right)_{t \in [0, \ell]}$ such that $\mathbf{s}_{k,\cdot}$ has only a finite number of jumps and approaches $\mathbf{s}_\cdot$ for $k \to \infty$. The integration by parts formula may be applied to 
 the model built on $\mathbf{s}_{k, \cdot}$ and \eqref{eq:overlap2C0} is recovered in the $k\to \infty$ limit.
\end{proof}

\section*{Acknowledgments and declarations}
We are grateful to Bernard Derrida for several fruitful exchanges. Y.H. acknowledges the support of ANR (Agence Nationale de la Recherche), project  ANR-22-CE40-0012.
No other funding was received for conducting this study.

\end{document}